\documentclass[12pt,leqno]{article}

\usepackage[pagebackref,hypertexnames=false]{hyperref} %,hypertexnames=false,colorlinks,[pagebackref]
\usepackage[backrefs]{amsrefs}
\usepackage{amsmath,amstext,amsthm,amssymb,amsxtra}
\usepackage{txfonts,pxfonts} %also txfonts

\usepackage{mathtools}
\mathtoolsset{showonlyrefs,showmanualtags} 
 \allowdisplaybreaks 
\swapnumbers

\usepackage{boxedminipage}

\theoremstyle{plain} % definition
\newtheorem{lemma}[equation]{Lemma}
\newtheorem{proposition}[equation]{Proposition}
\newtheorem{theorem}[equation]{Theorem}
\newtheorem{corollary}[equation]{Corollary}
\newtheorem{mainTheorem}[equation]{Main Theorem}
\newtheorem{FurstenbergKatznelson}[equation]{Furstenberg-Katznelson Theorem} 
\newtheorem{shkredov}[equation]{Shkredov's Two Dimensional Theorem}
\newtheorem{paleyZygmund}[equation]{The Paley-Zygmund Inequality}
\newtheorem{GCSI}[equation]{Gowers-Cauchy-Schwartz Inequality}

\newtheorem{DensityIncrement}[equation]{Density Increment Lemma}
\newtheorem{uniformizing}[equation]{Uniformizing Lemma}
\newtheorem{von}[equation]{The von Neumann Lemma}
\newtheorem{inverse}[equation]{Inverse Theorem for the Gowers $ U (3)$ Norm}

\newtheorem{BoxPaleyZygmund}[equation]{The Paley-Zygmund Inequality for the Box Norm}

\newtheorem{conserve}[equation]{First Proposition on Conservation of Densities} 
\newtheorem{conserve2}[equation]{Second Proposition on Conservation of Densities} 

\theoremstyle{definition}
\newtheorem{definition}[equation]{Definition}
\newtheorem{GowersBox}[equation]{Definition of Gowers Box Norms}

\theoremstyle{remark}
\newtheorem{remark}[equation]{Remark}
\newtheorem*{Acknowledgment}{Acknowledgment}

\numberwithin{equation}{section}

\def\norm#1.#2.{\lVert#1\rVert_{#2}}
\def\Norm#1.#2.{\bigl\lVert#1\bigr\rVert_{#2}}
\def\NOrm#1.#2.{\Bigl\lVert#1\Bigr\rVert_{#2}}
\def\NORm#1.#2.{\biggl\lVert#1\biggr\rVert_{#2}}
\def\NORM#1.#2.{\Biggl\lVert#1\Biggr\rVert_{#2}}

\def\ip#1,#2,{\langle #1,#2\rangle}
\def\Ip#1,#2,{\bigl\langle#1,#2\bigr\rangle}
\def\IP#1,#2,{\Bigl\langle#1,#2\Bigr\rangle}

\def\mid{\,:\,}

\def\abs#1{\lvert#1\rvert}
\def\Abs#1{\bigl\lvert#1\bigr\rvert}
\def\ABs#1{\biggl\lvert#1\biggr\rvert}
\def\ABS#1{\Biggl\lvert#1\Biggr\rvert}

\def\XXint#1#2#3{{\setbox0=\hbox{$#1{#2#3}{\int}$}
     \vcenter{\hbox{$#2#3$}}\kern-.5\wd0}}

\def\eqdef{\stackrel{\mathrm{def}}{{}={}}}

\def\mid{\;|\;}  

\begin{document}
%%%%%%%%%%%%%%%%%%%%%%%%%%%%%  Title
\title {Three Dimensional Corners: \\ A Box Norm Proof}

\author{Michael T. Lacey \and William  McClain}

%\address{Michael T. Lacey \\
%School of Mathematics\\
%Georgia Institute of Technology\\
%Atlanta GA 30332\\
%}
%
%
%
%\email{lacey@math.gatech.edu}
%

%\author[W. McClain]{William McClain}

%\address{William McClain \\
%School of Mathematics\\
%Georgia Institute of Technology\\
%Atlanta GA 30332\\
%}
%
%
%
%\email{bill@math.gatech.edu}

\maketitle

\begin{abstract}
For any discrete additive abelian group $ (G,+)$, we define a \emph{$ d$-dimensional corner}
 to be the $ d+1$ points in $ G ^{d}$ given by 
\begin{gather*}
g\,,\, g+ h \operatorname e _{r}\,,  \qquad 1\le r\le d\,, 
\qquad h\in G- \{0\}\,, 
\\
\operatorname e _r = \underbrace {(0 ,\dotsc, 1 ,\dotsc, 0)}_ {\textup{$ d$ dimensional vector}}
\,, \qquad 1\le r \le d\,.
\end{gather*}
 The  Ramsey numbers of interest are 
 $ R (G,d)$,  the maximum cardinality of a subset $ A \subset G ^{d}$ which 
does not contain a $ d$-dimensional corner.  

We give a new proof of a special case of the Theorem of Furstenberg and Katznelson 
\cite{MR833409} that in dimension $ d=3$, for the group $ G$ a finite field of characteristic 5, 
\begin{equation*}
R ( \mathbb F _{5} ^{n}, 3)= o ( \lvert  \mathbb F _5 ^{n}\rvert ^{3} )\,, \qquad n \to \infty . 
\end{equation*}
Our proof, specialized to one dimension, would 
reduce to Gowers' proof \cite{MR1631259} of four term arithmetic progressions in dense subsets 
of the integers. (Also see \cite{MR1844079}.)   Nevertheless, there are significant difficulties 
to overcome, and as a result this proof does not 
yield new quantitative bounds.
\end{abstract}

\setcounter{tocdepth}{2}

\vfill\eject\tableofcontents

%%%%%%%%%%%%%%%%%%%%%%%%%%%%%% SECTION  SECTION SECTION
%%%%%%%%%%%%%%%%%%%%%%%%%%%%%% SECTION  SECTION SECTION 
\section{Introduction} %\label{s.}

 For any discrete abelian group $ (G,+)$, we define a \emph{$ d$-dimensional corner}
 to be the $ d+1$ points in $ G ^{d}$ given by 
\begin{equation*}
g\,,\, g+h(1,0,0, ,\dotsc, 0)\,,\, 
g+ h(0,1,0 ,\dotsc, 0)\, ,\dotsc, g+ h(0,0,0 ,\dotsc, 1)\,, \qquad h\in G- \{0\}\,. 
\end{equation*}
The  Ramsey numbers of interest are 
 $ R (G,d)$,  the maximum cardinality of a subset $ A \subset G ^{d}$ which 
does not contain a $ d$-dimensional corner.  

The principal result in the subject is the Theorem of Furstenberg and Katznelson
\cite{MR833409}, a generalization 
of the Szemer{\'e}di Theorem \cite{MR0369312} to arbitrary dimension. 

\begin{FurstenbergKatznelson} We have the estimate below, for any dimension $ d$. 
\begin{equation*}
R (\mathbb Z _N, d)= o ( N ^{d})\,, \qquad N \to \infty . 
\end{equation*}
\end{FurstenbergKatznelson}

Our principal result of this result is a new proof of this Theorem, in dimension 
$ d=3$, for a finite field.  

\begin{mainTheorem} \label{t.main} We have this estimate, where $ N= 5^n= \lvert  F_5^n\rvert $, 
\begin{equation*}
R (\mathbb F _5 ^{n}, 3)= o ( N ^{3})\,, \qquad  n \to \infty \,. 
\end{equation*}
\end{mainTheorem}

The quantitative bound we provide is of Ackerman type, and accordingly we do not attempt 
to specify it.  In the two dimensional case, there is a much better quantitative bound, 
doubly logarithmic in nature, due to Shkredov
\cites{MR2266965,shkredov-2007}.

\begin{shkredov} There is a $ 0<c<1$ for which we have the estimate below in the two 
dimensional case. 
\begin{equation*}
R (\mathbb Z _N, 2) \lesssim \frac {N ^{2}} { (\log \log N) ^{c}}\,, \qquad N \to \infty \,. 
\end{equation*}
\end{shkredov}

In the simpler case of the finite field, one can get a better estimate, in that the 
constant $ c$ can be specified. See \cite{MR2289954}, also \cite{MR2187732}.  
Indeed it would appear that any improvement in the constant below would require new ideas.

%%%%%%%%%%%%%%%%%%%%%%%%%%%%%% THEOREM THEOREM THEOREM
\begin{theorem}\label{t.laceyMcclain} In the finite field setting, we have the estimate 
below in the two dimensional case.  Set $ N = p ^{n}$ for prime $ p$. 
\begin{equation*}
R ( \mathbb F _p ^{n}, 2) \lesssim  {   N  ^{2}} 
\frac {\log \log \log N}{ \log \log N }\,, \qquad N \to \infty \,. 
\end{equation*}
\end{theorem}
%%%%%%%%%%%%%%%%%%%%%%%%%%%%%% THEOREM THEOREM THEOREM

Our methods of proof are those of arithmetic combinatorics, which in most instances 
give better quantitative bounds.   However in this proof, our bounds are of Ackerman type. 
It took some time for a purely combinatorial proof of the 
Furstenberg-Katzneslon proof to be found \cites{MR2167756,gowers-2007,MR2195580} and the commentary in
\cite{MR2167755}.  Thus, our proof using the Gowers norms \cite{MR2167755}, and the double recursion 
argument of Shkredov \cite{MR2266965}, might have some independent interest.

The Theorem we discuss is the first `hard' case, as it corresponds to four-term arithmetic
progressions \cites{MR0245555,MR1631259}.   The `hardness' is expressed in terms of the 
very weak information that we get from the Box Norm, an issue we go into in more depth 
in the next section, see also \S~\ref{s.Tbox}. 
The rigorous results on Box Norm are Lemma~\ref{l.BPZ} below, and a more sophisticated variant 
Lemma~\ref{l.Tbox}.

A central question in the subject of Ergodic Theory concerns the identification of the  
characteristic factors for multi-linear ergodic averages, especially in the 
sense of Host and Kra \cites{MR2150389,MR2090768,MR1827115}. 
In the case of commuting transformations, the only complete information about 
these factors is in the case of two commuting transformations, a result of 
Conze and Lesigne \cite{MR788966}, also \cite{MR1827115}.  Incorporating their results in to a proof of Shkredov's 
Theorem is of substantial interest.  Our ignorance of these factors is also 
a hindrance in the result of Bergelson, Leibman and Lesigne \cite{0710.4862}.  Perhaps this 
approach can shed some light on this question.

There should be no essential difficulty in rewriting this proof to treat the estimate $ R (\mathbb Z _N, 3)= 
o (N ^{3})$. We have adopted the finite field setting just as a matter of convenience, making the arguments of
\S~\ref{s.uniform} technically a little easier (though admittedly there is little gain  in simplicity by 
this choice.)  It appears to be an interesting question, requiring additional insight, to extend this argument 
to higher dimensions.  

\begin{Acknowledgment}
The first author completed part of this work while in residence at the Fields Institute, Toronto Canada, 
as a George Eliot Distinguished Visitor. 
Support and hospitality of that Institute is gratefully acknowledged.   The second author has been supported 
by a NSF VIGRE grant at the Georgia Institute of Technology. 
\end{Acknowledgment}

%%%%%%%%%%%%%%%%%%%%%%%%%%%%%% SECTION  SECTION SECTION
%%%%%%%%%%%%%%%%%%%%%%%%%%%%%% SECTION  SECTION SECTION 
\section{Overview of the Proof} %\label{s.}

There is a substantial jump in difficulty of the proof in passing from the two 
dimensional case to the three case.  The three dimensional case, 
projected back to one dimension, gives a result about four term arithmetic progressions, 
explaining part of this difficulty.  Accordingly, we begin with a description of the 
two dimensional case.  

In two dimensions, the are three important coordinate directions:  
$ \operatorname e_1=(1,0)$, $ \operatorname e_2=(0,1)$, and $ \operatorname e_3=
\operatorname e_1 +\operatorname e_2$, associated with the endpoints of the corners. 
 
We exploit these three choices of coordinate directions by this mechanism.  
Consider three functions $ \lambda _j \;:\;  \mathbb Z _N ^{3} \longrightarrow \mathbb Z_N ^2 $ 
given by 
\begin{align} \label{e.olambda}
\lambda _j (x_1,x_2,x_3)= \sum _{k \;:\; k\neq j} x_k \operatorname e_k
\end{align}
The point of these definitions is that $ \lambda _j$ is \emph{not} a function of $ x_
j$. 

For a given set $ A\subset \mathbb Z_N ^2 $, the expected number of corners in $ A$ is 
\begin{align*}
\mathbb E _{x_1,x_2,x_3\in \mathbb Z _N} &
A (x_1,x_2)A (x_1+x_3,x_2)A (x_1,x_2+x_3)
\\&= 
\mathbb E _{x_1,x_2,x_3\in \mathbb Z _N} 
A (x_1,x_2)A (x_3-x_2,x_2)A (x_1,x_3-x_1) 
& (x_3\to x_3-x_1-x_2)
\\
&= 
\mathbb E _{x_1,x_2,x_3\in \mathbb Z _N} 
\prod _{j=1} ^{3} A \circ \lambda _j (x_1,x_2,x_3)\,. 
\end{align*}
Each of the three functions is a function of just two of the three variables 
$ x_1,x_2,x_3$.  

There is a specific mechanism to address expectations of such products: the Gowers Box 
norms.  Define one of these norms on a function $ g$ of $ x_1,x_2$ as follows.  
\begin{equation} \label{e.oBox}
\norm g .\Box \{1,2\}. 
= \bigl[ \mathbb E _{x_1,x_1',x_2,x_2'\in \mathbb Z _N} 
g (x_1,x_2)g (x_1',x_2)g (x_1,x_2')g (x_1',x_2')
\bigr] ^{1/4} 
\end{equation}
which is the cross-correlation of $ g$ at the four points of an average rectangle selected 
from $ \mathbb Z _N \times \mathbb Z _N$.  
Write $ \delta = \mathbb P (A)$, and 
$ f= A - \delta $, which is, following Gower's terminology, 
the balanced function of $ A$.  We then expand one of the $ A$'s in 
the expectation above as $ A= \delta + f$,  
\begin{gather*}
\mathbb E _{x_1,x_2,x_3\in \mathbb Z _N} 
\prod _{j=1} ^{3} A \circ \lambda _j (x_1,x_2,x_3)
= C_1+C_2
\\
C_1=\delta 
\mathbb E _{x_1,x_2,x_3\in \mathbb Z _N} 
\prod _{j=1} ^{2} A \circ \lambda _j (x_1,x_2,x_3)
\\
C_2= \mathbb E _{x_1,x_2,x_3\in \mathbb Z _N} f \circ \lambda _3
\prod _{j=1} ^{2} A \circ \lambda _j (x_1,x_2,x_3)
\end{gather*}

For the first of these terms, one can check directly that 
\begin{equation*}
C_1 \ge \delta \mathbb E _{x_1} \lvert  \mathbb E _{x_2} A (x_1,x_2)\rvert ^2 \ge \delta
^{3}\,.  
\end{equation*}
For sets $ A$ with the number of corners approximately equal to the number of corners
 that 
one would naively expect, this should be the dominant term.  On the other hand, 
it is the import and power of the Gowers Box Norms that we have the inequality 
\begin{equation} \label{e...0<}
\lvert   C_0\rvert \le \norm f . \Box \{1,2\}.   
\end{equation}
Thus, if this last quantity is less than, say, $ \tfrac12 \delta ^{3}$, the $ A$ has 
at least one-half of the expected number of corners.  

There is however, the alternative that $ \norm f . \Box \{1,2\}. \ge \tfrac12 \delta 
^{3}$,
which point brings us to an unfortunate fact concerning these Box Norms:  
The definition in \eqref{e.oBox} makes perfect sense on the product of arbitrary probability
spaces.  Accordingly, the consequence of the Box Norm being large can only have a
probabilistic consequence.  In the two dimensional case, it is this:  
There is are subsets $ R_1, R_2\subset \mathbb Z _N$ so that $ A$ correlates with 
the product set $ R_1 \times R_2$, namely $ \mathbb P (A\mid R_1 \times R_2)
\ge \delta + \tfrac 14 \delta ^{12}$, and the product set $ R_1 \times R_2$ is 
non-trivial, in that we have the estimates $ \mathbb P (R_1), \mathbb P (R_2)\ge 
c\delta ^{12}$, for appropriate constant $ c$.  There is however no additional 
structure on the sets $ R_1$ and $ R_2$.

The natural path, originating in Roth's proof \cite{MR0051853} for three term arithmetic progressions, 
is to iterate this alternative.  We can only hope to achieve an increment in density 
of 
$ A$ by an amount of $ \delta ^{12}$ a finite number of times.  
But without an additional insight, the iteration cannot go forward
as the use of the Gowers Box Norms requires at least a little arithmetic information 
through the use of the change of variables.  Shkredov \cites{MR2266965}
found a solution to this problem by 
introducing a secondary iteration, the result of which is that one finds further subsets 
$ R_1'\subset R_1$ and $ R_2'\subset R_2$ which satisfy three conditions.  First, 
we maintain the property that $ A$ has a higher density on $ R_1' \times R_2'$, 
namely $ \mathbb P (A\mid R_1' \times R_2')
\ge \delta + \tfrac 18 \delta ^{12}$.  Second, the sets $ R_1'$ and $ R_2'$ are non-trivial, 
in that they have a lower bound on their probabilities.  Third, $ R_1'$ and $ R_2'$ 
have arithmetic properties, in that their one-dimensional Box Norms are small. 
Specifically, $ R_1, R_2$ are subsets of a subspace $ H\le \mathbb F _2 ^{n}$, where there 
is a lower bound on the dimension of $ H$, and the norms 
\begin{equation*}
\norm R_j  (x_1+x_2)- \mathbb P (R_j\mid H) H (x_1+x_2). \Box ^{ \{1,2\}} H \times H. \,, \qquad j=1,2
\end{equation*}
are small.  
The first two conditions are certainly required.  It is the third property that permits 
the iteration to continue, as a subtle refinement of the inequality \eqref{e...0<} is
available.  

There is one additional feature of this discussion that we should bring forward, as it plays 
a decisive role in the three-dimensional case.  Namely, the discussion above placed a distinguished 
role on the standard basis $ (\operatorname e_1, \operatorname e_2) $, whereas the formulation of the 
question makes sense any any choice of basis from the three vectors $ \{\operatorname e_1, \operatorname e_2, 
\operatorname e_3\}$.  One can phrase a `coordinate-free' version of Shkredov's argument, which is the 
viewpoint of \cite{MR2289954}.  This is the viewpoint we adopt in the three-dimensional case.

\medskip 

We turn to the three dimensional case.    We again have the 
the standard basis $ \operatorname e _{j}$, for $ j=1,2,3$ in $ \mathbb Z _N ^{3}$.  
The fourth relevant basis element is $ \operatorname e_4=\sum _{j=1} ^{3} \operatorname e_j$
associated to the endpoints of the corner.  The analogs of the functions $ \lambda _j
$ in 
\eqref{e.olambda} are now four distinct functions from $ \mathbb Z _N ^{4}\longrightarrow
\mathbb Z _N ^{3}$ given by 
\begin{equation*}
\lambda _{j} (x_1, x_2,x_3,x_4)= \sum _{k \;:\; k\neq j} x_k \operatorname e_k\,. 
\end{equation*}
The point to exploit is that $ \lambda _j$ is \emph{not} a function of $ x_j$.

For a given set $ A\subset \mathbb Z _N ^{3}$, the average number of corners in $ A$ 
is given
by 
\begin{equation*}
\mathbb E _{x_1 , x_2,x_3,x_4\in \mathbb Z _N  } 
A (x_1,x_2,x_3) \prod _{j=1} ^{3} A ((x_1,x_2,x_3)+ x_4 \operatorname e_j)
=
\mathbb E _{x_1 , x_2,x_3,x_4\in \mathbb Z _N  } 
\prod _{j=1} ^{4} A \circ \lambda _j (x_1, x_2,x_3,x_4)\,. 
\end{equation*}
This is a four-linear term, which each of the four terms being dependent upon 
just three variables.

Again, there is a Gowers Box Norm that is relevant. 
This norm, of a function $ g (x_1,x_2,x_3)$ 
has a definition that can be given recursively as 
\begin{equation*}
\norm g(x_1,x_2,x_3). \Box \{1,2,3\}. ^{8}
=    \Norm  \abs{\mathbb E _{x_3\in \mathbb Z _N} g (x_1,x_2,x_3)} ^2  . \Box \{1,2\}. ^{4}
\end{equation*}
It has a similar interpretation as the average cross-correlation of $ g$ at the eight
corners of a randomly chosen box in $ \mathbb Z _N ^{3}$. 
To exploit the norm, 
we make the same expansion of  $A $. Setting $ \delta = \mathbb P (A\mid \mathbb Z _N
 ^{3})$,
and write $ A=\delta +f$.  Use this expansion just on $ A \circ \lambda _4$ above, so
that we can write 
\begin{gather*}
\mathbb E _{x_1 , x_2,x_3,x_4\in \mathbb Z _N  } 
\prod _{j=1} ^{4} A \circ \lambda _j =C_1+C_0
\\
C_1= \delta 
\mathbb E _{x_1 , x_2,x_3,x_4\in \mathbb Z _N  } 
\prod _{j=1} ^{3} A \circ \lambda _j 
\\
C_0= \mathbb E _{x_1 , x_2,x_3,x_4\in \mathbb Z _N  }  
f \circ \lambda _4 
\prod _{j=1} ^{3} A \circ \lambda _j \,. 
\end{gather*}

The Box Norm is introduced because it controls the second term. 
\begin{equation} \label{e.OverBox}
\lvert  C_0\rvert  
\le 
\norm f . \Box \{1,2,3\}. \,. 
\end{equation}
Thus, if the Box Norm is sufficiently small, $ C_0$ should be negligible.  
Turning to the term $ C_1$,  typically we would expect $ C_1$ to be of the order of 
$ \delta ^{4}$, but we do not have any simple recourse to establishing 
such a bound. Indeed, $ C_1$ is an instance of the two-dimensional question, 
as $ C_1$ is $ \delta $ times the average number of two-dimensional corners 
in $ A$, with the two-dimensional corners located on hyperplanes of the form 
$ (x_1,x_2,x_3) \cdot \operatorname e_4=c$, for some $ c$.  

This suggests to us that we will need to use a two-dimensional Box Norm 
on the hyperplanes just described. Namely, and this is an essential point, 
control of the Box Norm in \eqref{e.OverBox} is not sufficient to control 
the number of corners in $ A$.  Control of one more Box Norm, in a second set of 
coordinates, is required.  This situation can be avoided in the two-dimensional case.

We  
adopt a method that places the four coordinate vectors $ \{\operatorname e_j \mid 
1\le j\le 4\}$ on equal footing.  For each choice of subset $ I\subset \{1,2,3,4\}$, 
we have a Box Norm corresponding  to the   basis for $ \mathbb Z _N$ 
given by $ \{\operatorname e_j\mid j\in I\}$.   A sufficient condition for 
$ A$ to have a corner is that 
\begin{equation*}
\max _{\substack{I\subset \{1,2,3,4\}\\ \lvert  I\rvert=3  }} 
\norm f . \Box I . < 2^{-8}\delta ^{4}\,. 
\end{equation*}
These norms are distinct, namely that one can have $\norm f . \Box \{1,2,3\} . $ 
very small, while $ \norm f . \Box \{1,2,4\} .$ is much larger, a situation that 
does not arise in the one-dimensional case, as all of these norms turn out to be 
the same after a change of  variables.

Turning to the alternative, suppose that we have $ \norm f . \Box \{1,2,3\} . > 2^{-8
}\delta
^{4}$.  Again, the Box Norm admits a formulation on the three-fold product of probability 
spaces.  Accordingly we can only have a probabilistic consequence of the Box Norm being 
large, and it is a dramatically weaker statement than in the two-dimensional case.  It is 
this:  Associate $ \mathbb Z _N ^{3}$ to $ \mathbb Z _N ^{ \{1,2,3\}}$, with the 
superscripts signifying the coordinates. For $ J\subset \{1,2,3\}$ of cardinality $ 2$, 
associate $ \mathbb Z _N ^{J}$ to the corresponding face of $ \mathbb Z _N ^{ \{1,2,3
\}}$. 
For each such $ J$, there is a subset $ R_J\subset \mathbb Z _N ^{ J}$.  Consider the
fibers that lie above this set, denoted by 
\begin{equation*}
\overline R_J = \bigl\{ (x_1,x_2,x_3) \in \mathbb Z _N ^{ \{1,2,3\}} \mid 
\{ (x_1,x_2,x_3) \cdot \operatorname e_j \mid j\in J\}\in R_J
\bigr\}\,. 
\end{equation*}
Then, the conclusions are two fold. First, 
$ A$ has a higher density in $ \prod _{\substack{J\subset \{1,2,3\}\\ \lvert  J\rvert
=2  }} 
\overline  R_J$, and  second the latter set is non-trivial, in that it admits a lower
 bound on its 
probability.  Namely, the conclusions are 
\begin{gather} \label{e.oo}
 \mathbb P \Bigl( A \mid  \prod _{\substack{J\subset \{1,2,3\}\\ \lvert  J\rvert=2  }
} \overline R_J \Bigr)
\ge \delta + c \delta ^{C}\,, 
\\ \label{e.op}
\mathbb P \Bigl( \prod _{\substack{J\subset \{1,2,3\}\\ \lvert  J\rvert=2  }} \overline R_J \Bigr)
 \ge c \delta ^{C}\,. 
\end{gather}
Here $ 0<c,C$ are absolute constants.  Note that both conclusions are substantive.  
There is no \emph{a priori} reason that the set in \eqref{e.op} should admit this 
lower bound in its probability.   The other conclusion \eqref{e.oo} gives a correlation with a set, 
unfortunately, this set has substantially less structure than in the two-dimensional 
case.

Another essential complication arises from the fact that one must consider the $ 6$ sets 
$ R_J$, for $ J\subset \{1,2,3,4\}$, $ J$ consisting of two elements. If we consider the 
three-fold intersection 
$\prod _{\substack{J\subset \{1,2,3\}\\ \lvert  J\rvert=2  }}  \overline  R_J\,,$
one can see that it is well-behaved with respect to corners if the individual 
sets $ R_J$ are well-behaved with respect to two-dimensional Box Norms, and their one-dimensional 
projections are well-behaved with respect to the $ U (3)$ norm.  

But, there is no reason 
that the 3-dimensional set formed from the $ 6$-fold intersection 
$ \prod _{J\subset \{1,2,3,4\}} \overline R_J$ should be well-behaved 
with respect to any Box Norm.  To overcome this difficulty, we introduce an auxiliary set $ T\subset 
\overline R_J$ for all $ J$.   This set is required to be uniform with respect to all four 
three-dimensional Box Norms, but the Box Norm is taken relative to the sets $ R_J$.  

We are left with the following task: Find the appropriate `uniformity' conditions on the 
sets $ R_J$  and the set $ T$ so that these conditions are met.  First, we can 
obtain a variant of the inequality \eqref{e.OverBox}, namely 
if the set $ A$ is uniform in the `Box Norms adapted to $ T$' then $ A $ has a corner. 
Second, assuming that $ A$ is not uniform with respect to a `Box Norms adapted to $ T$,'
then we can find suitable variants  of \eqref{e.oo} and \eqref{e.op}. 

This must be done in a manner that is consistent with the choice of any of the four 
possible coordinate systems from $ \{\operatorname e_1,\operatorname e_2,\operatorname
e_3,\operatorname e_4\}$.

The remainder of the paper is organized as follows. 
%%  ITEMIZE
\begin{itemize}
\item  \S~\ref{s.lemmas} presents the most important definitions and three Lemmas which combine to prove 
our main result, Theorem~\ref{t.main}.  These three Lemmas set out, in broad terms the iteration scheme of 
Shkredov \cite{MR2266965}, but the formulation of the definitions is hardly clear. 
%%  ITEMIZE
\begin{itemize}
\item A critical definition is that of a corner-system, Definition~\ref{d.cornersystem}.  Such a system consists of 
the set $ A$, in which we seek a corner, and a number of auxiliary sets, such as the sets $ R_J$ mentioned above.  
If the  auxiliary sets are `suitably uniform' the the corner-system is called \emph{admissible}, see
Definition~\ref{d.admissible}.   

\item  A `generalized von Neumann Lemma,' to use the phrase of Ben Green and Terrance Tao \cite{math.NT/0404188}. 
Lemma~\ref{l.3dvon} states that if the corner-system is admissible, and 
$ A$ is suitably `uniform' in a non-obvious sense (and $ A$ is not too small, a weak condition)  
then $ A$ has a corner. 

\item An `increment Lemma,'  Lemma~\ref{l.dinc}.  This Lemma tells us that in the event that the hypothesis of 
of Lemma~\ref{l.3dvon} fails, we can find a new corner-system, which is non-trivial, in which $ A$ has a larger 
density.   It is this step that provides termination in our iteration, as the density of a set can never exceed 
one. The non-triviality comes from suitable lower bounds on the probabilities associated to the sets 
in the corner-system.  This Lemma, probabilistic in nature does \emph{not} provide for an admissible 
corner-system. 

\item A `Uniformizing Lemma,'  Lemma~\ref{l.uni}, 
in which a non-admissible corner-system is made admissible, permitting the recursion to
continue.

\end{itemize}
%% ITEMIZE
These three Lemmas are combined, in a known way see \S~\ref{s.algorithm}, to prove the Main Theorem. 

\item \S~\ref{s.box} sets out notation for the Box Norms which are essential for the entire paper, 
in particular the Gowers-Cauchy-Schwartz Inequality~\ref{gcsi}.   
These considerations have to be set out in some generality, as the later arguments will encounter a 
variety of Box Norms, and multi-linear forms consisting of up to $ 56$ functions.  
Most, but not all, of this section is standard, but worked out in a setting in which the underlying sets 
have relatively large probabilities.  

\item \S~\ref{s.linearForms} applies the results on the Box Norm to some classes of linear forms which 
arise in the context of the three-dimensional Box Norm.  These results have 
proofs which are appropriate refinements of the proof of the Gowers-Cauchy-Schwartz Inequality, taking into 
account the fact that the underlying sets we are interested have very small probabilities.  This section introduces 
a notion of uniformity with respect to linear forms of a bounded complexity, Definition~\ref{d.U}.  
An important component of the argument, is that the sets we consider only have 
a uniformity in the sense of Definition~\ref{d.U} of a bounded complexity.   Also in this section, and 
particularly 
important,  is the First Proposition on Conservation of Densities, Proposition~\ref{p.con}, and its corollary 
Lemma~\ref{l.Zvar}.

\item \S~\ref{s.formsCorners} is a reprise of the previous section.  In principle, we could have written 
the one section to encompass both this section and \S~\ref{s.linearForms}, but felt that this might make the 
paper harder to read.  This section contains the Second  Proposition on Conservation of Densities,
Proposition~\ref{p.con2}.   Both of these sections are central to the remainder of the argument.

\item \S~\ref{s.von} will prove the first of the three Lemmas, Lemma~\ref{l.3dvon}, by a subtle reworking of a 
standard Box Norm inequality.  In its simplest form, this argument was found by Shkredov \cite{math.NT/0404188}, but 
has a more refined elaboration in the current context.

\item \S~\ref{s.Tbox}
presents a Lemma we refer to as a `Paley-Zygmund inequality for the Box Norm,' 
see Lemma~\ref{l.BPZ}.   Namely, assuming that the Box Norm is big, deduce, e.\thinspace g.\thinspace,
the conclusions \eqref{e.oo} and \eqref{e.op} above.  This Lemma is presented in the simplest  context 
in the two dimensional setting.  
We then present  the same Lemma as above, but in the `weighted context.'  That is, in a context where 
the underlying spaces is \emph{not} just a tensor product space.  See Lemma~\ref{l.Tbox}.   Both of these Lemmas a
are stated in some generality, as the more general formulation is required in \S~\ref{s.uniform}.   The main 
result of this section, Lemma~\ref{l.Tbox}, requires a  careful elaboration of the proof in the `unweighted' case.

\item  \S~\ref{s.uniform} we address the fact that the data provided to us from Lemma~\ref{l.BPZ} and 
Lemma~\ref{l.Tbox} does not have any uniformity properties.   This is remedied by selecting a variety of 
partitions of the underlying space, with most of the `atoms' of the partitions are sufficiently uniform. 
It is in this section that the Ackerman function will arise.  The main Lemma is 
Lemma~\ref{l.uni}.  

\item The three Lemmas of \S~\ref{s.lemmas} are combined to prove our main Theorem in \S~\ref{s.algorithm}. 

\end{itemize}
%% ITEMIZE

 %%%%%%%%%%%%%%%%%%%%%%%%%%%%%% SECTION  SECTION SECTION
\section{Principal Lemmata} %\label{s.lemmas}
\label{s.lemmas}

Our proof is recursive, with each step in the recursion identifying a 
new subspace $ H\le \mathbb F _5 ^{n}$ in which we work.  $ H$ is of course 
a copy of $ \mathbb F _5 ^n$, just with a smaller value of $ n$.  We maintain a lower bound 
on the dimension of $ H$.

$ H\times H\times H$ has the standard basis elements 
$ \operatorname e _1$, $ \operatorname e _2$, and $ \operatorname e _3$.  
We also use the basis element 
\begin{equation}\label{e.4}
\operatorname e _4= \operatorname e _1+\operatorname e _2+\operatorname e _3\,, 
\end{equation}
which is the element associated with the `endpoints' of the corner.  A corner has 
an equivalent description in terms of any three elements of the four basis 
elements $ \{\operatorname e _i \mid  1\le i\le 4\}$.

Below, we will work with sets $ S_i$, $ 1\le i \le 4$.  They can be viewed as
 elements of the field $ H$.  But in addition, we view them as subsets of 
 $ H\times H\times H$, as follows: 
 \begin{equation} \label{e.overline}
\overline S _ i = \{ x\in H\times H\times H\mid x \cdot \operatorname e_i \in S_i  \}
\qquad 1\le i\le 4\,. 
\end{equation}
Thus, the fibers over $ \overline S_i$ are copies of $ H\times H$.

Likewise we will work with sets $ R_{i,j}\subset S_i \times S_j$.  They can be 
viewed as subsets of $ H\times H\times H$ by setting 
\begin{equation} \label{e.Overline}
\overline R_{j,k}= \{ x\in H\times H\times H\mid  (x \cdot \operatorname e 
_j , x \cdot \operatorname e_k) \in R_{j,k}\}\,, 
\qquad 1\le i<j\le 4\,. 
\end{equation}
Thus, the fibers of $ \overline R_{j,k}$ are copies of $ H$. 

%%%%%%%%%%%%%%%%%%%%%%%%%%%%%%  DEFINITION DEFINITION DEFINITION
\begin{definition}\label{d.cornersystem} 
By an  \emph{corner-system} we mean the data 
\begin{equation} \label{e.Asystem}
\mathcal A=\{H\,,\, S_i\,,\, R_{i,j}\, ,\, T\,,\, A \mid 1\le i,j\le 4\}
\end{equation}
where these conditions are met. 
%%  ITEMIZE
\begin{enumerate}
\item $ H$ is a subspace of $ \mathbb F _5 ^{n}$. 
\item  $ S_i \subset H$, $  1\le i\le 4$. 
\item $ R_{j,k}\subset S_j \times S_k$, $  1\le j<k\le 4$. 
\item $ T\subset \overline  R_{j,k}$, $  1\le j<k\le 4$. 
\item $ A\subset T$.  
\end{enumerate}
By a \emph{$ T$-system} we mean the data 
\begin{equation} \label{e.Tsystem}
\mathcal T=\{H\,,\, S_i\,,\, R_{i,j}\, ,\, T \mid 1\le i,j\le 4\}
\end{equation}
which is the same as a corner system, except that the set $ A$ is not listed, 
and so condition (5) above is not needed.

For such systems we use the notations 
\begin{gather} \label{e.Tlll}
T_ \ell \coloneqq   \bigcap _{\substack{
 1\le j<k\le 4\\  j,k\neq \ell   }}
\overline R_{j,k } ,\,  \quad  1\le  \ell  \le 4 \,, 
\\
\label{e.delta1} 
 \delta _j \coloneqq \mathbb P (S_j\mid H)\,, \qquad 
   \delta _{j,k} \coloneqq \mathbb P (R_{j,k}\mid S_j \times S_k)\,, \qquad 
  1\le j<k\le 4\,,
 \\ \label{e.delta2}
  \delta _{T \mid \ell } \coloneqq \mathbb P (T\mid T _{\ell })\,,  \qquad  1\le \ell \le 4\,. 
\end{gather}

\end{definition}
%%%%%%%%%%%%%%%%%%%%%%%%%%%%%%  DEFINITION DEFINITION DEFINITION

The sets $ T _{\ell }$ play an essential role in this proof for the following reason.  They are 
built up from lower dimensional objects in a natural way, and presuming that the lower dimensional 
objects are themselves well behaved with respect to box norms, then the $ T_ \ell $ is as well.  
The same conclusion does not seem to hold for the $ 6$-fold intersection 
$ \cap _{1\le i<j\le k} \overline R _{j,k}$.  That in turn lead us to the introduction of the 
auxilary set $ T\subset \overline R _{j,k}$.  Working on this indeterminant set $ T$ leads to 
most of the complications of this paper.

We use the notation $ R _{j,k}\subset S_j \times S_k$ rather than the (more natural) $ S _{j,k}$, as 
we will use the notation $ S _{j,k} \coloneqq S_j \times S_k$, in association with a number of Box Norms throughout 
the paper.

%%%%%%%%%%%%%%%%%%%%%%%%%%%%%%  DEFINITION DEFINITION DEFINITION
\begin{definition}\label{d.admissible}
Let $ C_{\textup{admiss}}\ge 64$ be a fixed large constant, and $ 0< \kappa _{\textup{admiss}} <1$ be a fixed small constant.   
Given  $ 0<\varepsilon<1$, and  $ T$-system $ \mathcal T$ as in \eqref{e.Tsystem}, 
we say that $ \mathcal T$ is
\emph{$ \varepsilon $-admissible} iff   
\begin{gather}
\label{e.Ad3Box}
\frac{\norm T  - \delta _{T \mid  \ell  } T _{\ell } . \Box \{ i\mid i\neq \ell \}. }
{\norm T _{\ell }. \Box { \{ i\mid i\neq \ell \} } . } \le  \kappa _{\textup{admiss}} \varepsilon  ^{C_{\textup{admiss}}} \cdot
\mathbb P (T\mid T _{\ell }) ^{C_{\textup{admiss}}} 
\,, \qquad {1\le \ell \le 4}\,,  
 \\
\label{e.Ad2Box}
{ {\norm R_{i,j}- \delta _{i,j} . \Box ^{\{i,j\}} 
 (S_i \times S_j). } } 
 \le  \kappa _{\textup{admiss}}  \varepsilon  ^{C_{\textup{admiss}}} 
 \mathbb P (T\mid  H \times H \times H ) ^{C_{\textup{admiss}}} 
 \,, \qquad {1\le i < j\le 4}\,,
\\
\label{e.Ad1Box} 
{ \norm S_{i} - \delta _i . U (3).} 
 \le  \kappa _{\textup{admiss}} \varepsilon  ^{C_{\textup{admiss}}} 
 \mathbb P (T\mid  H \times H \times H) ^{C_{\textup{admiss}}}
\,, \qquad {1\le i\le 4} \,. 
\end{gather}  

\end{definition}
%%%%%%%%%%%%%%%%%%%%%%%%%%%%%%  DEFINITION DEFINITION DEFINITION

All conditions require uniformity of the objects in terms of the density of $ T$ in that object.  
But the condition in  \eqref{e.Ad3Box} can not be strengthened in any way, and it is the condition that 
turns out to be the most subtle.  In particular, it will turn out that we can compute the 
expression $ {\norm T _{\ell }. \Box { \{ i\mid i\neq \ell \} } . }$ in \eqref{e.Ad3Box}, but 
it is also the case that $ T _{\ell }$ is \emph{not} uniform with respect to the norm $ \Box { \{ i\mid i\neq \ell \}
}$.

The norms in \eqref{e.Ad3Box} and \eqref{e.Ad2Box} 
are detailed in Definition~\ref{d.GowersBox} and \eqref{e.norm-simple}, 
but also given explicitly in the next definition.  

%%%%%%%%%%%%%%%%%%%%%%%%%%%%%%  DEFINITION DEFINITION DEFINITION
\begin{definition}\label{d.boxes}  Let $ X$, $ Y$ and $ Z$ be finite sets. 
For any function $ f \;:\; X\to \mathbb C $, we use the notation for expectation, namely  
\begin{equation*}
\mathbb E _{x\in X} f (x) = \lvert  X\rvert ^{-1} \sum _{x\in X} f (x)\,.  
\end{equation*}
Corresponding notation for probability $ \mathbb P (A)$, conditional probabilities, and 
conditional expectations, and conditional variance are also used. 

For a function $ f \;:\; X \times Y \longrightarrow \mathbb R
$, define 
\begin{equation}\label{e.2Boxexplicit}
\norm f . \Box ^{ \{x,y\}} (X \times Y). ^{4} 
\coloneqq  \mathbb E _{\substack{x,x'\in X\\ y,y'\in Y}} 
f (x,y) f (x,y') f (x',y) f (x',y') \,. 
\end{equation}
Note that the right hand side is the average of the cross-correlation of $ f$ over all
combinatorial rectangles in $ X \times Y$.  

For a function $ f \;:\; X \times Y \times Z 
\longrightarrow \mathbb R $, define 
\begin{equation}\label{e.3Boxexplict} 
\begin{split}
\norm f . \Box ^{x,y,z} (X \times Y \times Z). ^{8} 
& \coloneqq  \mathbb E _{\substack{z,z'\in Z}}
 \norm f ( \cdot , \cdot ,z)f ( \cdot , \cdot ,z') . \Box ^{x,y} (X \times Y). ^{4} 
 \\
 & = 
 \mathbb E _{\substack{x,x'\in X\\ y,y'\in Y \\ z,z'\in Z}}  
f (x,y,z) f (x,y',z) f (x',y,z) f (x',y',z)
\\ & \qquad \times 
f (x,y,z') f (x,y',z') f (x',y,z') f (x',y',z') \,. 
\end{split}
\end{equation}
This has a similar interpretation as the norm in \eqref{e.2Boxexplicit}. 
In \eqref{e.Ad3Box}, we use the notation 
\begin{equation} \label{e.norm-simple}
\norm g. \Box { \{ i\mid i\neq \ell \}}. 
\coloneqq 
\norm g. \Box ^ { \{ i\mid i\neq \ell \}}( H\times H\times H) . \,. 
\end{equation}
This notation is consistent with \eqref{e.norm-simple} below. 
\end{definition}  
%%%%%%%%%%%%%%%%%%%%%%%%%%%%%%  DEFINITION DEFINITION DEFINITION

The $ U (3)$ norm used in \eqref{e.Ad1Box} has a definition that is similar to the 
Box Norms, but has an additive component. 

%%%%%%%%%%%%%%%%%%%%%%%%%%%%%%  DEFINITION DEFINITION DEFINITION
\begin{definition}\label{d.U3} For $ f \;:\;  H \longrightarrow \mathbb \mathbb R  $, we define 
\begin{equation*}
\norm f. U (3). \coloneqq  \norm f (x+y+z). \Box ^{x,y,z} H \times H \times H .  
\end{equation*}

\end{definition}
%%%%%%%%%%%%%%%%%%%%%%%%%%%%%%  DEFINITION DEFINITION DEFINITION

In these definitions, observe
%%  ITEMIZE
\begin{itemize}
\item A $ \delta $ represents a `density,' and this will most frequently be a 
relative density.  Thus, $ \delta _{i,j}$ is the density of $ R_{i,j}$ in $ S_i \times 
S_j$.   In some of these notations, this relative density is indicated explicitly, as in 
the definition for $ \delta _{T \mid \ell }$. 

\item Likewise, the Box Norms in \eqref{e.Ad3Box} and \eqref{e.Ad2Box} are relative 
Box Norms.  In \eqref{e.Ad2Box}, this relative norm is indicated in the notation.  But, 
in \eqref{e.Ad3Box} this is indicated by the division by $ \norm T _{\ell }. \Box \{i\mid i
\neq \ell \}. $. 

\item Notice that the uniformity conditions \eqref{e.Ad3Box}---\eqref{e.Ad2Box} are 
phrased relative to the the `higher dimensional objects in question.'  Thus, the uniformity 
condition on $ T $ in \eqref{e.Ad3Box} is phrased in terms of the 
densities of $ T$ in $ T _{\ell }$. 

\item The previous point, not anticipated by the two-dimensional version of this 
Theorem, is important to the proof of our critical Lemma~\ref{l.uni} below.  
And it complicates the proof of Lemma~\ref{l.3dvon}.  

\item It is possible that the degree of uniformity require on $ S_i$ in \eqref{e.Ad1Box} and 
$ R _{i,j}$ in \eqref{e.Ad2Box} is too high.  For instance, one could imagine that 
\eqref{e.Ad1Box} should be replaced by 
\begin{equation} \label{e.ad1box}
{ \norm S_{i} - \delta _i . U (3).} 
 \le  \kappa  \varepsilon  ^{C_{\textup{admiss}}} 
 \mathbb P (T\mid  \overline S_i) ^{C_{\textup{admiss}}}
\,, \qquad {1\le i\le 4} \,. 
\end{equation}
As it turns out, the conditions \eqref{e.Ad1Box} and \eqref{e.Ad2Box} are available to us 
by this proof, and so we use them.  The distinction between \eqref{e.ad1box} and \eqref{e.Ad1Box} 
could be important in extensions of this argument to higher dimensions. 
\end{itemize}
%% ITEMIZE

The three Lemmas are very much as in \cites{MR2266965,MR2289954}, 
though  with more complicated statements in the current setting. 
The first Lemma asserts that for admissible corner-systems, 
if dimension is not too small, and the Box Norms $ \norm  A - \delta _{A\mid T} T 
.\Box\{ i\mid i\neq \ell \}.$ are sufficiently small, uniformly in $ \ell $ then $ A$ has a corner.

%%%%%%%%%%%%%%%%%%%%%%%%%%%%%% LEMMA LEMMA LEMMA
\begin{von} \label{l.3dvon} Suppose that we are given an corner-system $ \mathcal A$ as in 
\eqref{e.Asystem}.  Set  $ \delta _{A\mid T}= \mathbb P (A \mid T)$, 
and   assume that $ \mathcal A$ is  $ \delta _{A \mid T}$-admissible.  
The following  two conditions are then  sufficient for $ A$ to have a corner.
\begin{gather}
 \label{e.BigEnough}
\delta _{A\mid T} \cdot \prod _{j=1} ^{4} \delta _j \cdot 
\prod _{1\le j<k\le 4} \delta _{j,k} \cdot \prod _{\ell =1} ^{4} \delta _{T\mid \ell } \cdot \lvert  H\rvert ^{4} 
>4 \lvert  A\rvert \,,  
\\
\label{e.UniformEnough}
\max _{1\le \ell \le 4} \frac {\norm  A - \delta _{A\mid T} T. \Box \{ i\mid i\neq \ell \}. }
{\norm T . \Box \{ i\mid i\neq \ell \}.} \le \kappa  \delta _{A \mid T} ^{4} \,. 
\end{gather}
\end{von}
%%%%%%%%%%%%%%%%%%%%%%%%%%%%%%%%% LEMMA LEMMA LEMMA

The condition \eqref{e.BigEnough} is the condition, typical to the subject, that the `average number of corners' 
in $ A$ exceed the number of `trivial corners' in $ A$. The second condition \eqref{e.UniformEnough} is the 
all important uniformity condition. 
The second Lemma is the alternative if \eqref{e.UniformEnough} does not hold. 

%%%%%%%%%%%%%%%%%%%%%%%%%%%%%% LEMMA LEMMA LEMMA
\begin{DensityIncrement}\label{l.dinc} There is an absolute constant $ \kappa $ for 
which the following holds. 
Suppose  that  the corner-system in \eqref{e.Asystem} is  $ \delta _{A\mid T} $-admissible, and that
(\ref{e.UniformEnough}) \emph{does not hold.}  Then,
there are sets
\begin{equation*}
S_i'\subset S_i\,,
\quad
R_{i,j}'\subset R_{i,j}\,,
\quad
T'\subset T_\ell ' =  \prod _{
\substack{1\le i,j\le 4\\  i,j\neq \ell }}S '_{i,j}
\end{equation*}
These sets satisfy the estimates 
$
\mathbb P (T'\mid T)
\ge \delta _{A\mid T} ^{1/ \kappa}
$ and 
 $ \mathbb P (A\mid T')\ge \delta_{A\mid T}+ \
\delta_{A\mid T} ^{1/ \kappa}$.   
\end{DensityIncrement}
%%%%%%%%%%%%%%%%%%%%%%%%%%%%%%% LEMMA LEMMA LEMMA

It is the last estimate that provides a termination for our algorithm in \S~\ref{s.algorithm}. 
The previous Lemma, which is  probabilistic in nature, does not supply us with
admissible data.  This is rectified in the next Lemma.

%%%%%%%%%%%%%%%%%%%%%%%%%%%%%% LEMMA LEMMA LEMMA
\begin{uniformizing}\label{l.uni} 
There is are  functions 
\begin{equation*}
\Psi _{\operatorname {dim}} \,,\, \Psi _{T}\;:\; [0,1] ^{3} \longrightarrow \mathbb N 
\end{equation*}
for which the following holds
for all $0< v<\delta <1 $.  Let
$ \mathcal A$ be an corner-system as in \eqref{e.Asystem}.  Assume that 
  $ \mathbb P (A\mid T)\ge \delta +v$.
There is a new corner-system 
\begin{equation*}
\mathcal A' =\{H'\,,\, S_i'\,,\, R_{i,j}'\, ,\, T'\,, A '\mid 1\le i,j\le 4\}
\end{equation*}
so that for some $ x\in H$,   $ A' \subset A+x$, and similarly for $ T'\subset T+x$. 
More importantly, we have:
\begin{gather}
\label{e.H'} 
\operatorname {dim} (H')\ge \operatorname {dim} (H)-\Psi _{\operatorname {dim}} (v, \delta )
\\
\label{e.inc5}
\mathbb P (A'\mid T')\ge \delta +\tfrac v 4 
\\ \label{e.inc0} 
 \textup{ $ \mathcal A'$ is $ \delta $-admissible,}
\\
\label{e.inc2}
\mathbb P (T'\mid H'\times H'\times H')\ge 
\Psi _{T} (\delta , v, \mathbb P (T \mid H\times H\times H) )\,. 
%\\\label{e.inc-ij}
%\mathbb P ( S' _{i,j} \mid S _i' \times R_{j}')\ge  ???\,,  \qquad 1\le i,j\le 4  \,,
%\\\label{e.inc3}
%\mathbb P (S'_i \mid W) \ge ???\,, \qquad 1\le i\le 4 \,.
\end{gather}
\end{uniformizing}

We remark that in \eqref{e.H'}, if the dimension of $ H$ is too small, then 
$ \mathcal A '$ will be trivial in that $ T'$ consists of only one point. 
These Lemmas are combined in a standard way to prove our Main Theorem.  The details 
are in \S~\ref{s.algorithm}.

\section{Box Norms}\label{s.box}

It will be helpful to recall the Gowers uniformity  or Box Norms in 
a more general form.  In this we follow the the presentation in the 
appendices of \cite{math.NT/0606088}, with most, but not all, Lemmas 
similar in statement to that reference.  The notion of a Box Norm is 
critical to all the principal arguments of this paper; accordingly, 
we have pulled these general results together into their own section. 

\begin{GowersBox} \label{d.GowersBox}  Let $\{ X_{u} \}_{ u  \in U}$ 
be a finite non-empty collection of finite non-empty sets indexed by $  u \in U$. 
For any $V\subseteq U$ write $X_{V} := \prod_{ v  \in V} X_{v} $ for the Cartesian product. 
For a complex-valued function 
$f_{U}: X_{U} \to \mathbb C $, we define the \emph{Gowers Box Norm}  (or just \emph{Box Norm}) 
$\norm f_{U} . {\Box^U(X_{U})} . \in {\mathbb R}^+$ to be 
\begin{equation}\label{fa}
\lVert f_{U}\rVert_{\Box^U(X_{U})}^{2^{\lvert  U\rvert }} := \mathbb E_{x^{0}_{U}, x^{1}_{U} \in X_{U}} 
\prod_{\omega_{U} \in \{0,1\}^U} {\mathcal C}^{\vert\omega_{U}\rvert} f_U( x^{\omega_{U}}_{U} ) 
\end{equation}
where ${\mathcal C}: z \mapsto \overline{z}$ is complex conjugation,
and for any $x^{0}_{U} = (x^{0}_{u} )_{ u\in U}$ and 
$x^{1}_{U} = (x^{1}_{u} )_{ u\in U}$ in $X_{U}$ and
$\omega_{U} = (\omega_{u} )_{ u\in U}$ in $\{0,1\}^U$, we write
$x^{\omega}_{U} := (x^{\omega_{u} }_{u} )_{ u\in U}$
and $|\omega_{U}| := \sum_{ u\in U} \omega_{u} $. 
In the special case that  $U$ is empty, forcing $f_{U}$ to be  a constant, we have 
$\lVert f_{U}\rVert_{\Box^U(X_{U})} := \lvert  f_{U}\rvert $.
\end{GowersBox}

Above, we  use the notation $ A^B$ for the class of maps from $ B$ into $ A$, which notation will be used throughout 
the paper. 
If $ U= \{u\}$, then $ \lVert f_{U}\rVert_{\Box^U(X_{U})} = \lvert  \mathbb E _{X_u} f \rvert
$.  In particular this is non-negative, and can be zero.  Note that if $ A\subset X _{U}$,
$\lVert A\rVert_{\Box^U(X_{U})} ^{2 ^{\lvert  U\rvert }}  $ is the 
average number of `boxes'  in $ A$.  Thus, 
$ \lVert A- \mathbb P (A\mid X_U)\rVert_{\Box^U(X_{U})} $ measures the degree to which 
$ A$ behaves as expected, in regards to the number of boxes it contains. 
It is also easy to verify that if $ A$ is a randomly selected subset of $ X _{U}$, 
then $  \lVert A- \mathbb P (A\mid X_U)\rVert_{\Box^U(X_{U})} $ is small.  A similar 
 point is essential to this section:  Sets which are small with respect to this semi-norm behave in a manner 
similar to randomly selected subsets.  A set $ A$ for which $\lVert A- \mathbb P (A\mid
X_U)\rVert_{\Box^U(X_{U})}  $ is small we will call \emph{uniform}.

The Box Norms 
arise through the following inequality, proved by inductive application of 
the Cauchy-Schwartz inequality.  For this Lemma, see \cite{math.NT/0606088}*{Lemma B.2}.

\begin{GCSI}\label{gcsi} 
Let $ U$ be non-empty, and  $\{X_{u}\}_{u\in U}$ be a finite collection of finite non-empty sets. 
For every $\omega_{U} \in \{0,1\}^U$ let $f_{U}^{\omega_{U}}: X_{U} \to \mathbb C $ be a function.
Then
\begin{equation}\label{gcz-box}
\biggl\lvert\mathbb E_{x^{0}_{U}, x^{1}_{U} \in X_{U}}
\prod_{\omega_{U} \in \{0,1\}^U} {\mathcal C}^{\lvert\omega_{U}\rvert} f_{U}^{\omega_{U}}(x^{\omega_{U}})\biggr\rvert
\leq \prod_{\omega_{U} \in \{0,1\}^U} \lVert f_{U}^{\omega_{U}} \rVert_{\Box^U(X_{U})}\,.
\end{equation}
\end{GCSI}

From this, it follows that one has the \emph{Gowers Triangle Inequality.}
\begin{equation}\label{e.gti}
\norm f_{U}+g_{U}. \Box ^{U} (X_{U}). 
\le \norm f_{U}. \Box ^{U} (X_{U}). +\norm g_{U}. \Box ^{U} (X_{U}). 
\end{equation}
Indeed, raise both sides of the equation above to the power of $ 2 ^{\lvert  U\rvert }$
and use \eqref{gcz-box}.

We will also refer to this corollary to the Gowers-Cauchy-Schwartz inequality. 	

%%%%%%%%%%%%%%%%%%%%%%%%%%%%%% COROLLARY COROLLARY COROLLARY
\begin{corollary}\label{c.gcsi}  
Let $\{X_{u} \}_{ u\in U}$ be a finite collection of finite non-empty sets.  
For $ V\subset U$, let $ f _{V} \;:\;  X_{V} \to \{z\in \mathbb C 
\mid \lvert  z\rvert\le 1 \} $.  Then, 
\begin{align}\label{e.gcsi1}
\Abs{ \mathbb E _{x\in X_{U}}  \prod _{V\subset U} f_{V} (x_{V}) } 
&\le \norm f_{U} . \Box^U (X_{U}). \,.
\end{align}
That is, only the Box Norm associated to the largest set $ U$ is needed. 
Here, for $ x\in X_{U}$, $ x_{V}$ is the restriction of the sequence $ x= \{x_{u}  \mid 
 u\in U\}$ to the set $ V\subset U$. 
\end{corollary}
%%%%%%%%%%%%%%%%%%%%%%%%%%%%%%  COROLLARY COROLLARY COROLLARY

The  inequality  \eqref{e.gcsi1} is \cite{math.NT/0606088}*{(B.7)}, and it suggests 
that the $ \Box^U$ norm is insensitive to `lower order' perturbations.  
We single out a more general inequality that is important to us.  

%%%%%%%%%%%%%%%%%%%%%%%%%%%%%% LEMMA LEMMA LEMMA
\begin{lemma}\label{l.gcsi1}  Under the hypotheses of Corollary~\ref{c.gcsi}, 
 for $ V_0\subsetneq U$, we have 
\begin{equation} \label{e.gcsi2} 
\Abs{ \mathbb E _{x\in X_{U}}  \prod _{\substack{V\subset U\\ \lvert V\rvert\le \lvert  V_0\rvert  }} 
f_{V} (x_{V}) } 
\le \norm f_ {V_0} . \Box^ {V_0} (X_ {V_0}). \,.
\end{equation}
\end{lemma}
%%%%%%%%%%%%%%%%%%%%%%%%%%%%%% LEMMA LEMMA LEMMA

The inequality \eqref{e.gcsi2} has a proof  similar to \eqref{e.gcsi1}, and 
we omit the proof. (Our proof of the von Neumann Lemma below could provide a proof, 
as we comment when we arrive there.) 
It has a similar interpretation to the first inequality: the 
$ \Box ^{V_0}$ norm is insensitive to perturbations of the same order in distinct 
variables. 

%%%%%%%%%%%%%%%%%%%%%%%%%%%%%% COROLLARY COROLLARY COROLLARY
\begin{corollary}\label{c.gcsi1}  
For all $ \epsilon >0$ and all integers $ k$,  and finite sets $ U$ with $ \lvert  U\rvert\ge k $
there is a $ C_1 = C_1 (\lvert  U\rvert, k, \epsilon  ) $ for which the following holds. 

Let $\{X_{u} \}_{ u\in U}$ be a finite collection of finite non-empty sets, 
and $ X_{V}= \prod _{ u\in V} X_{u}   $, for $ V\subset U$.   
Let $ \mathcal U_k$ be the collection of subsets of $ U$ of cardinality $ k$, and 
for each $ V\in \mathcal U_k$ let $ S_V\subset X_V$ satisfy 
\begin{equation}\label{e.xgc1}
\norm S_V - \mathbb P (S_V). \Box ^{V} X_V. 
\le \bigl( \tfrac12 \mathbb P (S_V) \bigr) ^{C_1} \,, \qquad V\in \mathcal U _k\,. 
\end{equation}
Then, we have the inequality 
\begin{equation}\label{e.xgc2}
\ABs{ \mathbb E _{X_U} \prod _{V\in \mathcal U_k}  S_V 
- \prod _{V\in \mathcal U_k}  \mathbb E _{X_V} S_V }
\le 
\epsilon \prod _{V\in \mathcal U_k}  \mathbb E _{X_V} S_V  \,. 
\end{equation}
\end{corollary}
%%%%%%%%%%%%%%%%%%%%%%%%%%%%%%  COROLLARY COROLLARY COROLLARY

Thus, if all the sets $ S_{V}$ are \emph{very uniform} with respect to the natural 
Box Norms, the expectation of the products of the $ S_{V}$ behaves as if the 
sets are randomly selected.

%%%%%%%%%%%%%%%%%%%%%%%%%%%%%% PROOF PROOF PROOF
\begin{proof} 
We induct on the number $ w$ of elements of $ V\in \mathcal U_k$ 
for which $ S_V\neq X_V$.  That is, we prove that for all
all $ \epsilon >0$, integers $ k$, and $ 1\le w \le \lvert  \mathcal U_k\rvert $
there is a $ C_1 (\lvert  U\rvert, k , \epsilon ,w )$ so that if 
 for collections $ S_V$, with at most $ w$ choices of $ V
\in \mathcal U_k$ do we have $ S_V \neq X_V$ satisfying \eqref{e.xgc1}
 we have \eqref{e.xgc2}.

The case of $ w=1$ is obvious.  Let us suppose 
that this holds for $ 1\le w < \lvert  \mathcal U_k\rvert $, and prove the claim for 
$ w+1$.  We take 
\begin{equation*}
C_2=C_2 (\lvert  U\rvert, k, \epsilon , w+1 ) 
= 
w+3 + \log _2 1/\epsilon + C_1 ( \lvert  U\rvert, k, \epsilon /2, w )\,.  
\end{equation*}
Considering the collections $ S_V$ for $ V\in \mathcal U_k$, we select 
$ V_0$ so that $ \mathbb P (S_{V_0})$ minimal.  Thus, in particular 
we must have $ S _{V_0} \subsetneq X _{V_0}$.    Write 
$ S_{V_0}= \mathbb P (S _{V_0})+f _{V_0}$.
Since all the sets in $ \mathcal U_k$ 
have the same cardinality, we have the inequality 
\begin{equation*}
\Abs{\mathbb E _{x_U\in X_U} f _{V_0} \prod _{\substack{V\in \mathcal U_k - \{V_0\} }}
S_V } \le \norm f _{V_0}. \Box ^{V_0} X_ {V_0}. 
\le \bigl( \tfrac 12 \mathbb P (S_{V_0}) \bigr) ^{C_2} 
\le \frac \epsilon 4  \prod _{V\in \mathcal U_k}\mathbb E _{x_V\in X_V} S_V \,. 
\end{equation*}
The last line follows from the selection of $ V_0$.  

We can the apply the induction hypothesis to estimate 
\begin{align*}
\ABs{ \mathbb E _{X_U} \prod _{V\in \mathcal U_k}  S_V 
- \prod _{V\in \mathcal U_k}  \mathbb E _{X_V} S_V }
& \le
\frac \epsilon 4  \prod _{V\in \mathcal U_k}\mathbb E _{x_V\in X_V} S_V
\\& \qquad + \mathbb P (S_{V_0}) 
\ABs{ \mathbb E _{X_U} \prod _{V\in \mathcal V_k - \{V_0\}}  S_V 
- \prod _{V\in \mathcal U_k- \{V_0\}}  \mathbb E  _{X_V}S_V } 
\\
& \le \epsilon \prod _{V\in \mathcal U_k}  \mathbb E _{X_V} S_V \,. 
\end{align*}
So the induction is complete.

We can then conclude the Lemma by taking $  C_1 (\lvert  U\rvert, k, \epsilon  )
=  C_2 (\lvert  U\rvert, k, \epsilon/2  , \lvert  \mathcal U_k\rvert )$.  

\end{proof}
%%%%%%%%%%%%%%%%%%%%%%%%%%%%%% PROOF PROOF PROOF

 We frequently use this corollary of the Gowers-Cauchy-Schwartz inequality. 
 %  %%%%%%%%%%%%%%%%%%%%%%%%%%%%%% LEMMA LEMMA LEMMA
 \begin{lemma}\label{l.Ugcsi} 
 Let $\{X_{u} \}_{ u\in U}$ be a finite collection of finite non-empty sets.
 For $ V\subset U$, let $ S_{V} \subset X_{V}$.  Then, for an integer $ k \le \lvert  U\rvert $
 \begin{equation}\label{e.Ugcsi}
 \ABs{ \mathbb E _{x\in X_{U}}  \prod _{\substack{V\subset U\\ \lvert V\rvert\le k  }} 
 S_V (x_{V}) - \prod _{\substack{V\subset U\\ \lvert V\rvert\le k  }} 
 \mathbb E _{x_{V}\in X_{V}} S_V (x_{V})} 
 \le  2 ^{\lvert  U\rvert } \cdot \max _{ \substack{V\subset U\\ \lvert V\rvert\le k  }}
 \norm S_{V} - \mathbb E _{x_{V}\in X_{V}} S_{V} . \Box^{V} (X_V). \,.
 \end{equation}
 \end{lemma}
 %%%%%%%%%%%%%%%%%%%%%%%%%%%%%% LEMMA LEMMA LEMMA
 %  Thus, if all the sets $ S_{V}$ are \emph{very uniform} with respect to the natural 
 Box Norms, the expectation of the products of the $ S_{V}$ behaves as if the 
 sets are randomly selected. In order for this inequality to be non-trivial, we need 
 \begin{equation*}
 \max _{ \substack{V\subset U\\ \lvert V\rvert\le k  }}
 \norm S_{V} - \mathbb E _{x_{V}\in X_{V}} S_{V} . \Box^{V} (X_V).
 \le 2 ^{- \lvert  U\rvert } 
 \prod _{\substack{V\subset U\\ \lvert V\rvert\le k  }} 
 \mathbb E _{x_{V}\in X_{V}} S_V (x_{V})
 \end{equation*}
 Of course, the Lemma is trivial if $ k=1$, and 
 for $ k>1$, this uniformity requirement is quite 
 restrictive if the sets $ S_{V}$ have small probabilities.  
 This is exactly the situation in our proof. 
 %  
 %  %%%%%%%%%%%%%%%%%%%%%%%%%%%%%% PROOF PROOF PROOF
 \begin{proof}
 We view 
 \begin{equation} \label{e.1B}
 \mathbb E _{x\in X_{U}}  \prod _{\substack{V\subset U\\ \lvert V\rvert\le k  }} 
 S_V (x_{V})
 \end{equation}
 as a multi-linear form, with the order of the multi-linearity being 
 $
 \sum _{j=1} ^{k} \binom { \lvert  U\rvert  } j\,,  
 $
 a term which we have crudely estimated by $ 2 ^{\lvert  U\rvert }$ in \eqref{e.Ugcsi}. 
   For each set $ V\subset U$, we consider the expansion of the function $ S_{V}$ 
 as $ S_{V}=g_{V,0}+g_{V,1}$ where $ g_{V,0}= \mathbb P (S_{V} \mid X_{V}) \cdot X_{V}$, 
 and $ g_{V,1} $ is the balanced function.  
   We expand the term in \eqref{e.1B}.  Let $ \mathcal I $ be the collection of subsets 
 of $ A$ of cardinality at most $ k$. We have 
 \begin{equation*}
 \eqref{e.1B}= 
 \sum _{\epsilon \in \{0,1\} ^{\mathcal I}} 
 \mathbb E _{x_U\in X_{U}}  \prod _{\substack{V\subset U\\ \lvert V\rvert=\le k  }} 
 g_{V, \epsilon (V)} (x_{V})\,. 
 \end{equation*}
 The leading term arises from the choice of $ \epsilon_0 $ which takes the value $ 0$ 
 for all choices of sets $ V$.  For this function we have 
 \begin{equation*}
 \mathbb E _{x_U\in X_{U}}  \prod _{\substack{V\subset U\\ \lvert V\rvert\le k  }} 
 g_{V, \epsilon_0 (V)} (x_{V})
 =
 \prod _{\substack{V\subset U\\ \lvert V\rvert\le k  }} 
 \mathbb E _{x_{V}\in X_{V}} S_V (x_{V})\,,
 \end{equation*}
 which is part of the expression on the left in \eqref{e.Ugcsi}.  
 %  For each choice of $ \epsilon \in  \{0,1\} ^{\mathcal I}$ which is not $ \epsilon _0$, 
 let $ B_1\subset A$ be a maximal cardinality set for which $ \epsilon (B_1)=1$. 
 Then, for any subset $ V\subset U$ with $ \lvert  B_1\rvert< \lvert V\rvert\le k  $, 
 we have $ \epsilon (V)=0$, so that $ g_{V, \epsilon (V)}$ is a constant 
 function, taking a value of at most one. 
 It follows from \eqref{e.gcsi2} that we have 
 \begin{equation*}
 \Abs{\mathbb E _{x_U\in X_{U}}  \prod _{\substack{V\subset U\\ \lvert V\rvert\le k  }} 
 g_{V, \epsilon (V)} (x_{V})}
 \le 
 \Abs{\mathbb E _{x_U\in X_{U}}  
 \prod _{\substack{V\subset U\\ \lvert V\rvert\le  \lvert  B_1\rvert   }} 
 g_{V, \epsilon (V)} (x_{V})}
 \le 
 \norm g_{V_1,1} . \Box ^B (X_{V}). \,. 
 \end{equation*}
 From this, \eqref{e.Ugcsi} follows. 
 \end{proof}
 %%%%%%%%%%%%%%%%%%%%%%%%%%%%%% PROOF PROOF PROOF
 %  
 We note the following Corollary to the proof above, with the main distinction 
 being that some of the functions are indicators of uniform sets as before, while 
 others are arbitrary bounded functions.   The conclusion is that the uniform 
 sets matter little to the computation of the expectation. 
 %  %%%%%%%%%%%%%%%%%%%%%%%%%%%%%% LEMMA LEMMA LEMMA
 \begin{corollary}\label{c.Ugcsi} 
 Let $\{X_{u} \}_{ u\in U}$ be a finite collection of finite non-empty sets and let 
 $ k$ be a non-zero integer.  
 Let $ \mathcal V_1$ and $ \mathcal V_2$ be two collections of subsets of $ U$, with 
 all members of $ \mathcal V_1$ and $ \mathcal V_2$ having cardinality at most $ k$. 
 For $ V\in \mathcal V_1$, let $ S_{V} \subset X_{V}$.  For $ W\in \mathcal V_2$ let 
 $ f_W \;:\; X_W \longrightarrow [-1,1]$ be a bounded function. 
 Then, 
 \begin{equation}\label{e.cUgcsi}
 \begin{split}
 \ABs{ \mathbb E _{x\in X_{U}}  \prod _{\substack{V \in \mathcal V_1}} 
 S_V (x_{V}) 
 \prod _{\substack{V \in \mathcal V_2}} f_W  (x_W)
 &- \prod _{\substack{V \in \mathcal V_1}}
 \mathbb E _{x_{V}\in X_{V}} S_V (x_{V}) 
 \times 
 \mathbb E _{x_U\in X_U} \prod _{\substack{V \in \mathcal V_2}} f_W (x_W)
 }  
 \\&\qquad \le  2 ^{\lvert  U\rvert } \cdot \max _{ \substack{V\in \mathcal V_1}}
 \norm S_{V} - \mathbb E _{x_{V}\in X_{V}} S_{V} . \Box^{V} (X_V). \,.
 \end{split}
 \end{equation}
 \end{corollary}
 %%%%%%%%%%%%%%%%%%%%%%%%%%%%%% LEMMA LEMMA LEMMA

 We turn to a more complicated version of these Lemmas and Corollaries.  
 
 %%%%%%%%%%%%%%%%%%%%%%%%%%%%%% LEMMA LEMMA LEMMA
\begin{lemma}\label{l.01} Let $ U$ be a finite set, and $ X_u$ for $ u\in U$ 
another finite set.  Fix $ 1<k< \lvert  U\rvert $, and let $ \mathcal V$ be a 
collection of subsets of $ U$ of cardinality at most $ k$.  Let $ S_U\subset X_U$, 
and write $ \delta = \mathbb P (S_U)$.  Assume that   
\begin{equation}\label{e.0v} 
\sup _{V\in \mathcal V_k}
\mathbb E _{x ^{0} _{U-V} \in X _{U-V}} \norm f _{U} (x^V_U). \Box ^{V} X_V .=\tau 
<\delta \lvert \mathcal  V\rvert ^{-1} \,, 
\qquad 
f_U \coloneqq S_U- \delta \,. 
\end{equation}
We emphasize that, in the expansion of the Box Norm above, the Box Norm is taken 
over the variables associated to $ V$ and the expectation is taken over \emph{all} 
variables in $ U$.  The conclusion is that we have the inequality below. 
\begin{equation}\label{e..0<} 
\mathbb E  _{x ^{0} _{U} } \ABs{ \delta ^ {\lvert  \mathcal V\rvert }
-\mathbb E _{x ^1 _{U}}\prod _{V\in \mathcal V}  S _U ( x^V _{U})} 
\lesssim \tau     \,. 
\end{equation}
The implied constant depends upon $ \lvert  V\rvert $.  Above, by very slight abuse 
of notation, we mean 
\begin{equation*}
x^V _{U} = \begin{cases}
x^1 _{v} & v\in V
\\
x ^0 _v  & v\not\in V 
\end{cases}
\end{equation*}
\end{lemma}
%%%%%%%%%%%%%%%%%%%%%%%%%%%%%% LEMMA LEMMA LEMMA

This is a `conditional' version of Corollary~\ref{c.Ugcsi}. 
In particular, note that in \eqref{e..0<}, we impose the Box Norms in 
the variables $ X_V$, and take the expectation over all of $ X_U$. 
The conclusion is again that if the set is suitably small with respect to 
a family of relevant Box Norms, then a range of products of these sets behave 
as if the set were randomly selected.

%%%%%%%%%%%%%%%%%%%%%%%%%%%%%% PROOF PROOF PROOF
\begin{proof} Let us begin by noting that for $ V\in \mathcal V$, the monotonicity of the 
Box Norms as the variables increase imply that 
\begin{equation*}
\mathbb E  _{x ^{0} _{U} } \Abs{ \delta 
-\mathbb E _{x ^1 _{U}} S _V ( x^V _{U})} 
\le \norm S_U - \delta . \Box ^{V} X_U. \le \tau \,. 
\end{equation*}
It follows by the assumption on the magnitude of $ \tau $ that we can estimate 
\begin{align*} 
\Abs{\delta ^ {\lvert  \mathcal V\rvert }
-\prod _{V\in \mathcal V}  \mathbb E _{x ^1 _{U}} S _V ( x^V _{U})}
&\le (\delta + \tau ) ^{\lvert  V\rvert }- \delta ^{\lvert  V\rvert } 
\\
& \le \delta ^{\lvert  V\rvert } \bigl[(1+ \tau \delta ^{-1} )^{ \lvert  V\rvert} -1\bigr]
\le \tau 
\end{align*}
Also note that we can estimate, using Lemma~\ref{l.Ugcsi}, 
\begin{align*}
\mathbb E _{x ^{0}_U} 
\Abs{ \mathbb E _{x_U ^{1} } 
\prod _{V\in \mathcal V}    S_{U} (x^V _{U}) 
- \prod _{V\in \mathcal V}   \mathbb E _{x_U ^{1} }  S _{U} (x^V _{U})  }
& \lesssim 
\mathbb E _{x ^{0}_U}  \sup _{V\in \mathcal V} \Norm S_U (x_V) -  \mathbb E _{x_U ^{1} }  S
_{U} (x^V _{U}) . \Box ^V X_V. 
\lesssim  2 \tau \,.
\end{align*}
Putting these inequalities together proves the Lemma.

\end{proof}
%%%%%%%%%%%%%%%%%%%%%%%%%%%%%% PROOF PROOF PROOF

%%%%%%%%%%%%%%%%%%%%%%%%%%%%%% SECTION  SECTION SECTION
%%%%%%%%%%%%%%%%%%%%%%%%%%%%%% SECTION  SECTION SECTION 
\section{Linear Forms for the Analysis of Box Norms} \label{s.linearForms}

Box Norms, and counting corners in sets are examples of multi-linear forms that we will work with. 
Their analysis will lead to forms in as many as $ 24$ functions, leading to the need for some general 
remarks on such objects.  Moreover, we are analyzing these forms on objects that are far 
from tensor products.  This is the primary focus of this section.   

We will be making a wide variety of approximations to different expectations.  In order to codify these 
approximations, let us make this definition.  

%%%%%%%%%%%%%%%%%%%%%%%%%%%%%%  DEFINITION DEFINITION DEFINITION
\begin{definition}\label{d.=} Fix $ 0<\upsilon < 3 ^{-28}$ be a small constant. 
For $ A,B>0$ we will write  $ A \stackrel u = B$  if $ \lvert  A-B\rvert < \upsilon  A$.  
(We stack a `$ u$' on the equality, as this relation will always come about from uniformity.)
In those (few) instances, where it is important emphasize the role of $ \upsilon  $, we will 
write $ A \stackrel {u, \upsilon  }  = B$. 
\end{definition}
%%%%%%%%%%%%%%%%%%%%%%%%%%%%%%  DEFINITION DEFINITION DEFINITION

We will only use the notation for quantities between $ 0$ and $ 1$.  Observe the following.  
Let $ 0<A,B,\alpha , \beta <1$.  If $ A \stackrel {u, \upsilon } = \alpha  $ and $ B \stackrel {u, \upsilon } = \beta $, then 
we have 
\begin{align*}
\lvert  A- \alpha \cdot \beta \rvert &\le  
 \lvert  A- \alpha B\rvert+   \alpha  \cdot \lvert  \beta -B\rvert 
 \\
 & \le \upsilon A+ \alpha \upsilon  B \le 3 \upsilon A \,.  
\end{align*}
Thus,  we can write $ A \stackrel {u, 3\upsilon } = \alpha \cdot \beta $, that is this relationship 
is weakly transitive. We will need to 
use a finite chain of inequalities of this type, with the longest chain associated with the analysis 
of a $ 28$-linear form in Lemma~\ref{l.Z} below.   By abuse of notation, we will adopt the convention 
$ A \stackrel u = B$ and $ B \stackrel u = C$ implies $ A \stackrel u = C$.  This transitivity will 
only be applied a finite number of times, so that taking an initial $ \upsilon $ in Definition~\ref{d.=} will 
lead to a meaningful inequality at every stage of our proof. 

A second situation we will have is this.  Suppose that $ A \stackrel {u, \upsilon } = A'$ and $ B \stackrel {u,\upsilon } = B'$.  
Then, 
\begin{align*}
\lvert  AA'-BB'\rvert & \le \lvert  A-B\rvert A' + \lvert  A'-B'\rvert B 
\\ & 
\le \upsilon (AA'+A'B) \le 3 \upsilon A A' \,. 
\end{align*}
Thus,  we can write $ A A' \stackrel {u, 3\upsilon } = B B'$, thus this relationship is 
weakly multiplicatively transitive.  We will need to use a finite chain of these inequalities, mostly related 
to computing conditional expectations.   By abuse of notation, we will adopt the convention that 
$ A \stackrel u = A' $ and $ B \stackrel u = B'$ implies $  A A' \stackrel u = B B'$.  This observation is 
closely linked with the fact that our definition of admissibility, Definition~\ref{d.admissible} includes 
relative measures of uniformity.

Our Lemmas and Definitions 
should be coordinate-free, but to ease the burden of notation, we state them distinguishing  
the coordinate $ x_4$ for a special role.  They will be applied in their more general formulations, which 
are left to the reader.

We are concerned with  the evaluation of certain multi-linear forms, especially 
those associated with Box Norms.   For a collection of maps  $ \Omega \subset \{0, ,\dotsc, \lambda -1\} ^{\{1,2,3\}}$, 
where $ \lambda \ge 2$ is an integer, let $ \{f _{\omega } \mid \omega \in \Omega \}$ be a collection of functions. 
The linear forms we are interested in are 
\begin{equation}\label{e.Ldef}
\operatorname L( f _{\omega } \mid  \Omega )
= \mathbb E _{x _{1,2,3} ^{\ell } \in S _{1,2,3} \,, \, 0\le \ell\le 2 } 
\prod _{\omega \in \Omega } f _{\omega } (x _{1,2,3} ^{\omega }) \,. 
\end{equation}

This next definition is concerned with the uniform evaluation of forms of this type, where the $ f_\omega $ 
are particularly simple. 

%%%%%%%%%%%%%%%%%%%%%%%%%%%%%%  DEFINITION DEFINITION DEFINITION
\begin{definition}\label{d.U} 
Let $ \lambda \ge 3$ be an integer, and $ 0< \vartheta <1$. 
A subset  $ U\subset T_4$  is called \emph{$ (\lambda , \vartheta,4)  $-uniform} if the following 
holds.   Set $ \Omega_{3\to \lambda } = \{0 ,\dotsc, \lambda -1 \}^{ \{1,2,3\}}$ . 
For any subset $ \Omega \subset \Omega_{3\to \lambda }$ we have the inequalities 
\begin{equation} \label{e.Uuniform}
 \operatorname L _{\Omega } (U\mid \Omega ) 
 \stackrel {u,  \vartheta }  =
 \bigl[ \delta _{ 4} \delta _{U\mid 4}
\bigr] ^{ \lvert  \Omega \rvert }  
\prod _{1\le j<k\le 3}  
\delta _{j,k} ^{ \lvert \{\omega \vert _{ \{j,k\}} \mid \Omega \}  \rvert } 
\end{equation}
Here, $ \delta _{U\mid 4} =\mathbb P (U\mid T_4)$.  
That is, the percentage error between the two terms is at most $ \vartheta $. 
\end{definition}
%%%%%%%%%%%%%%%%%%%%%%%%%%%%%%  DEFINITION DEFINITION DEFINITION

It is an important point that we index this notion on the number of linearities that we permit the form to have, as  
we must provide an upper bound on this notion of complexity.  
Our primary objective is that $ T$ be well-behaved with respect to the Box Norm, in particular that 
Lemma~\ref{l.Tbox} holds.  This will require that $ T$ be $ (4, \vartheta_1 , 4)$-uniform, where 
$ \vartheta _1$ is specified in that Lemma.  But this will in turn require us to require $ T_4$ is 
$ (12, \vartheta _2, 4)$-uniform.  It is one purpose of this section to explain this relationship. 
See Lemma~\ref{l.Uuniform}.   

While we will use these results several times, there are two points where either these results 
apply, but would lead to an increased order of complexity, as in the proof of \eqref{e.L10}, 
or the results of this section are not stated in enough generality, as in the proof of \eqref{e.B4var}. 
A full understanding of these issues would likely be an aid to extending this argument to higher 
dimensions.

In this definition, examining the product of densities, we see that $ \delta _{U\mid 4}= 
\mathbb P (U\mid T_4)$ has the power $ \lvert  \Omega \rvert $, that is the total number of 
terms in the product.  The power on the density $ \delta _{j,k}$ is the number of distinct 
maps of the form $ \omega $, restricted to $ \{j,k\}$ in the set $ \Omega $.  To set out 
an example, a typical 
term to which we will apply this definition is to the set $ U=T_4$, in 
\begin{equation*}
 \mathbb E _{\substack{x _{1}\in S_1, 
\\ x_{2,3} ^0, x _{2,3}^1\in S _{2,3}} }
\prod _{ \epsilon \in \{0,1\} ^ {\{2,3\}} } 
T_4 (x_1, x_{2,3} ^{\epsilon }) 
\end{equation*}
Here, it is clear that $ \lvert  \Omega \rvert=4 $, while 
\begin{equation*}
\lvert \{\omega \vert _{ \{1,2\}} \mid  \Omega \}  \rvert = 2\,, 
\qquad 
\lvert \{\omega \vert _{ \{1,3\}} \mid  \Omega \}  \rvert = 2\,, 
\qquad 
\lvert \{\omega \vert _{ \{2,3\}} \mid  \Omega \}  \rvert = 4\,.
\end{equation*}
The parameter $ \vartheta $ appears on the right in \eqref{e.Uuniform}, and represents 
how close, in terms of percentages, the expectation behaves with respect to its expected behavior. 

 A set $ U$ is $ (\lambda, \vartheta,4) $-uniform if 
a wide set of expectations of $ U$ `behave as expected.'  
It is hardly obvious that even the set $ T_4$ satisfies this definition, but it 
does, and  we prove in Lemma~\ref{l.Uuniform} that both $ T_4$ and $ T$ are uniform. 

%%%%%%%%%%%%%%%%%%%%%%%%%%%%%% LEMMA LEMMA LEMMA
\begin{lemma}\label{l.Uuniform}   We have the following two assertions.  
For  constants $  C_1> C_0>0$ that depend only on $ C _{\textup{admiss}}$ in Definition~\ref{d.admissible} 
the following are true.  

%%  ENUMERATE
\begin{enumerate}
\item  For $ \vartheta = \delta _{T\mid T_4} ^{C_0}$, the set  $ T_4$ is  $ (12, \vartheta , 4)$-uniform.  
\item For $ \vartheta = \delta _{T\mid T_4} ^{C_1}$,  the set $ T$ is $ (6, \vartheta ,4)$-uniform.  
\end{enumerate}
%% ENUMERATE
In fact, $ C_1, C_0$ can be taken to be a small constant multiple of $ C _{\textup{admiss}}$. 

\end{lemma}
%%%%%%%%%%%%%%%%%%%%%%%%%%%%%% LEMMA LEMMA LEMMA

As the statement of the Lemma indicates, there is a link between the complexity of the linear forms 
we need to consider for $ T$ and $ T_4$.

%%%%%%%%%%%%%%%%%%%%%%%%%%%%%% PROOF PROOF PROOF
\begin{proof}
Let us discuss $ T_4$ first.  Note that by \eqref{e.Ad1Box} and \eqref{e.gcsi2}, 
\begin{align} \notag
\operatorname L (T_4 \mid \Omega ) & = 
\mathbb E _{\substack{x_{1,2,3} ^\ell  \in S _{1,2,3} \\  0\le \ell\le11  }}  
\prod _{\omega \in \Omega } T_4 (x _{1,2,3} ^{\omega }) 
\\  \label{e.454}
&=
\mathbb E _{\substack{x_{1,2,3} ^\ell  \in S _{1,2,3} \\  0\le \ell\le11  }}  
\prod _{\omega \in \Omega } 
S_4 (x_1 ^{\omega (1)}+ x_2 ^{\omega (2)}+ x_3 ^{\omega (3)})
\prod _{1\le j<k\le 3} S _{j,k} (x _{j,k} ^{\omega })
\\ 
&= 
\delta _{4} ^{\lvert  \Omega \rvert } \cdot 
\mathbb E _{\substack{x_{1,2,3} ^\ell  \in S _{1,2,3} \\  0\le \ell\le11  }}  
\prod _{\omega \in \Omega } 
\prod _{1\le j<k\le 3} S _{j,k} (x _{j,k} ^{\omega })
\label{e.ER1}
+O(\mathbb P (T\mid H \times H \times H) ^{C _{\textup{Admiss}} -12}) \,. 
\end{align}
The power on $\mathbb P (T\mid H \times H \times H)  $ accounts for the fact that implicitly the 
condition \eqref{e.Ad1Box} is an expectation over $ H$, while above we are taking integration over 
$ S _{1,2,3}$.  

We continue with the analysis of the expectation above.  We can use \eqref{e.gcsi2} and \eqref{e.Ad2Box} to 
estimate 
\begin{align}
\mathbb E _{\substack{x_{1,2,3} ^\ell  \in S _{1,2,3} \\  0\le \ell\le11  }}  
\prod _{\omega \in \Omega } 
\prod _{1\le j<k\le 3} S _{j,k} (x _{j,k} ^{\omega })
&= 
\prod _{1\le j<k\le 3}   
\delta _{j,k} ^{ \lvert \{\omega \vert _{ \{j,k\}} \mid \Omega \}  \rvert }  
  \label{e.ER2}
+O(\mathbb P (T \mid S_{1,2,3,4}) ^{C _{\textup{admiss}}}) \,. 
\end{align}
The leading terms of the expectations are exactly as desired.  The two error terms in \eqref{e.ER1} and \eqref{e.ER2} 
should be as small as desired, namely that they contribute at most $ \vartheta \operatorname L (T_4 \mid \Omega ) $. 
But it is straight forward to see that we can take $ C_0$ of the Lemma to be $ C _{\textup{admiss}}-12 - \lvert  \Omega \rvert 
\ge C _{\textup{admiss}}-12 - 3 ^{12}$, with $ 3 ^{12}$ being the cardinality of $ \Omega _{3\to 12} = \{0 ,\dotsc, 11\}
^{\{1,2,3\}}$.

We turn to the second conclusion of the Lemma.  Let $ \Omega \subset \Omega _{3\to 6}$, 
and consider the multi-linear expression $ \operatorname L (T\mid \Omega )$. 
Each occurrence of $ T$ is expanded as $ T=f_1+f_0$ where $ f_1= \delta _{T\mid 4} T_4$.  
The leading term is when each $ T$ is replaced by $ f_1$, which leads to $ \delta _{T\mid 4} ^{\lvert  \Omega \rvert }$ 
times the expectation in \eqref{e.454}.   There are $ 2 ^{\lvert  \Omega \rvert }-1 $ terms remaining.  Each of them 
has an occurrence of $ f_0$.  All of these terms can be controlled by the assumption \eqref{e.Ad3Box}, and
importantly, the inequality \eqref{e.3Box} below. (We have not yet proved \eqref{e.3Box}, part of 
Lemma~\ref{l.goodBox}, but its proof is independent of this argument.) 
This last Lemma is applied with $ \lambda =6$, $ V=T_4$, which as we have just seen in the first half of the 
proof, is $ (12, \vartheta ', 4)$-uniform, for a very small choice of $ \vartheta '$. 
This gives us  
\begin{align*}
\Abs{ \operatorname L (T\mid \Omega )- \delta _{T\mid T_4} ^{\lvert  \Omega \rvert } 
\operatorname L (T_4\mid \Omega )
}
&\le 2 ^{\lvert  \Omega \rvert+1 } \operatorname L (T_4 \mid \Omega ) 
\cdot \frac { \norm f_0 . \Box ^{1,2,3} S _{1,2,3}.} { \norm T_4. \Box ^{1,2,3} S _{1,2,3}.}
\\ & 
\le  2 ^{\lvert  \Omega \rvert+1 }  \delta _{T\mid 4} ^{C _{\textup{admiss}}} \cdot \operatorname L (T_4 \mid \Omega ) \,. 
\end{align*}
And this completes the proof.  
\end{proof}
%%%%%%%%%%%%%%%%%%%%%%%%%%%%%% PROOF PROOF PROOF

Here is a  corollary to the previous Lemma that is certainly relevant for us. 

%%%%%%%%%%%%%%%%%%%%%%%%%%%%%% LEMMA LEMMA LEMMA
\begin{lemma}\label{l.UTbox} We have this estimate 
\begin{align} \notag 
 \norm T_4 . \Box ^ {1,2,3} H _{1,2,3}.  ^{8} 
& =  
\mathbb E _{x _{1,2,3}^1, x _{1,2,3}^0\in H _{1,2,3}} 
\prod _{\omega \in \{0,1\} ^{1,2,3}} T_4 \circ \lambda _4 (x _{1,2,3} ^{\omega }) 
\\  \notag 
&= 
\mathbb E _{x _{1,2,3}^1, x _{1,2,3}^0\in H _{1,2,3}} 
\prod _{\omega \in \{0,1\} ^{1,2,3}} 
S_4 \circ \lambda _4 (x _{1,2,3} ^{\omega })  \prod _{1\le j<k\le 3} S _{j,k} (x _{1,2,3} ^{\omega }) 
\\ \label{e.T4Box}
&\stackrel u = 
\prod _{j=1} ^{3} \delta _j ^{2} \cdot \delta _4 ^{8} \cdot  \prod _{1\le j<k\le 3} \delta _{j,k} ^{2} \,. 
\end{align}

\end{lemma}
%%%%%%%%%%%%%%%%%%%%%%%%%%%%%% LEMMA LEMMA LEMMA

We return to general considerations, and make a remark that we will refer to several times. 
Let $ V\subset T_4$ be $ (\lambda , \vartheta ,4)$-uniform.  Let $ \Omega \subset \Omega _{3\to \lambda -1}$, 
and assume that  the set $ \Omega _{1\to0} $ is non-trivial. 
\begin{equation*}
\Omega _{1\to 0}= \{\omega \in \Omega \mid \omega (1)=0\}\,, \qquad \Omega _{1\not\to0} = \Omega - \Omega _{1\to 0}\,. 
\end{equation*}
Consider the estimate below obtained by applying the Cauchy-Schwartz inequality in all variables except $ x_1 ^0$.  
\begin{align}  \label{e.con1}
\operatorname L (V\mid \Omega ) &\le  \bigl[ \operatorname L (\Omega _{1\not\to0} ) \cdot U_2 \bigr] ^{1/2} 
\\ \label{e.con2}
U_2 & = \mathbb E \prod _{\omega \in \Omega _{1\not\to0}} V ({x _{1,2,3} ^{\omega }}) 
\cdot \Abs{\mathbb E _{x_1 ^{0} \in S_1}  \prod _{\omega \in \Omega _{1\to0}} 
\prod _{\omega \in \Omega _{1\to0}} V ({x _{1,2,3} ^{\omega }}) } ^2 \,. 
\end{align}
Use \eqref{e.elem} to write the last term as $ U_2 = \operatorname L (V \mid \Omega ^{1})$, where 
we define 
\begin{align} \label{e.con3}
\overline \omega  (j) =
\begin{cases}
 \lambda & j=1 
 \\
 \omega (j) & j=2,3 
\end{cases}
\\ \label{e.con4}
\Omega ^{1} = \Omega _{1\not\to0} \cup \{\omega , \overline \omega \mid \omega \in \Omega _{1\to0}\}\,. 
\end{align}

\begin{conserve} \label{p.con} If 
$ V\subset T_4$ be $ (\lambda , \vartheta ,4)$-uniform,  $ \Omega \subset \Omega _{3\to \lambda -1}$, 
with the notation in \eqref{e.con1}---\eqref{e.con4} we have the equality 
\begin{equation}\label{e.con5}
\operatorname L (V\mid \Omega ) \stackrel {u, \sqrt \vartheta } = \operatorname L (V\mid \Omega _{1\not\to0}) ^{1/2} \cdot 
\operatorname L (V\mid \Omega ^1 ) ^{1/2} \,. 
\end{equation}
\end{conserve}

%%%%%%%%%%%%%%%%%%%%%%%%%%%%%% PROOF PROOF PROOF
\begin{proof}
The proof is almost trivial.  Each $ \omega \in \Omega $ on the contributes $ 1$ to the densities 
$ \delta _{V\mid 4}, \delta _4, \delta _{j,k}$ for $ 1\le j<k\le3$.   If $ \omega (1)\neq 0$, 
it contributes to both terms on the right, so the square root makes contribution $ 1$.  
If $ \omega (1)=0$, then it contributes nothing to $ \operatorname L (V\mid \Omega _{1\not\to0} )$, 
but contributes $ 2$ to the other term $ \operatorname L (V\mid \Omega ^1 )$.  
\end{proof}
%%%%%%%%%%%%%%%%%%%%%%%%%%%%%% PROOF PROOF PROOF

The previous Lemma plays a decisive role in all our applications of the Cauchy-Schwartz inequality, to prove 
our weighed versions of these inequalities.  This Conservation of Densities has an essentially equivalent 
formulation, also important to us, that we give here.  With the notation of \eqref{e.con1}---\eqref{e.con4}, 
set 
\begin{equation} \label{e.Zdef}
Z [ \Omega _{1\not\to0} \;:\; \Omega _{1\to0} ] = \mathbb E _{x_1 ^{0} \in S_1} 
\prod _{\omega \in \Omega _{1\to0}} V (x _{1,2,3} ^{\omega }) 
\end{equation}

%%%%%%%%%%%%%%%%%%%%%%%%%%%%%% LEMMA LEMMA LEMMA 
\begin{lemma}\label{l.Zvar}  
Let $ \lambda =1 ,\dotsc, 6$. 
Suppose that the set $ V\subset T_4$ is $ (\lambda, \vartheta , 4)$-uniform, 
where $ \vartheta \le \mathbb P (V\mid T_4) ^{2 \cdot  3 ^{\lambda }}$. 
Then, for all choices of $ \Omega \subset \Omega _{3\to \lambda -1 }$ as above, we have 
\begin{gather}\label{e.Zvar}
\begin{split}
\operatorname {Var} _{x_j ^{\ell }\in \Omega } \Bigl(Z[\Omega _{1\not\to0} \;:\;  \Omega _{1\to0}] 
&\mid   \prod _{\omega \in \Omega _{1\not\to0}} V (x _{1,2,3} ^{\omega })\Bigr)
\\&
\le  K\sqrt{\vartheta  } \cdot 
\Bigl[ \mathbb E  
 \Bigl(Z[\Omega _{1\not\to0} \;:\;  \Omega _{1\to0}] \mid  
\prod _{\omega \in \Omega _{1\not\to0}} V (x _{1,2,3} ^{\omega })\Bigr) \Bigr] ^2 \,. 
\end{split}
\end{gather}
Here, $ K$ is an absolute constant. 
\end{lemma}
%%%%%%%%%%%%%%%%%%%%%%%%%%%%%% LEMMA LEMMA LEMMA

Of course the conditional expectation of $ Z$ can be computed.  

%%%%%%%%%%%%%%%%%%%%%%%%%%%%%% PROOF PROOF PROOF
\begin{proof}
We use the standard formula for the variance of a random variable $ W$ supported on a set $
Y$. 
\begin{equation} \label{e.elemVar}
\operatorname {Var} (W\mid Y) = 
\mathbb P (Y) ^{-1} 
 \mathbb E W ^2 - (\mathbb P (Y) ^{-1}  \cdot \mathbb E W)^2 
\end{equation}

The conditional variance will be small if we have 
\begin{equation*}
\mathbb E  
 \Bigl(Z[\Omega _{1\not\to0} \;:\;  \Omega _{1\to0}] ^2  \mid  
\prod _{\omega \in \Omega _{1\not\to0}} V (x _{1,2,3} ^{\omega })\Bigr)
\stackrel u = 
\mathbb E  
 \Bigl(Z[\Omega _{1\not\to0} \;:\;  \Omega _{1\to0}] \mid  
\prod _{\omega \in \Omega _{1\not\to0}} V (x _{1,2,3} ^{\omega })\Bigr) ^2 \,. 
\end{equation*}
But this is a recasting of \eqref{e.con5}. 
Namely, using the notation of \eqref{e.con5}, we can write the equation above as 
\begin{equation*}
\frac { \operatorname L (V \mid \Omega ^{1})} {\operatorname L (V\mid \Omega _{1\to0})}
\stackrel u =
\frac { \operatorname L (V \mid \Omega ) ^2 } {\operatorname L (V\mid \Omega _{1\to0}) ^2 }
\end{equation*}
which is \eqref{e.con5}. 

\end{proof}
%%%%%%%%%%%%%%%%%%%%%%%%%%%%%% PROOF PROOF PROOF

We are interested in refinements of 
the Gowers Box Norms, in which we estimate $ \operatorname L$ in terms 
of a Box Norm of one of its arguments, but do so in a  more efficient manner, just as in the 
proof of Lemma~\ref{l.3dvon}, which is presented in \S~\ref{s.von}.  For this Lemma, let us consider selections of $ f_\omega $ 
where $ f _{\omega } \in \{f, V\}$, and $ f$ is a fixed function supported on $ V$ and at most 
one in absolute value. In
application, $ f$ is  a balanced function.

In this Lemma, we will single out the first and second coordinates for a distinguished role,
which is done just for simplicity.  

%%%%%%%%%%%%%%%%%%%%%%%%%%%%%% LEMMA LEMMA LEMMA
\begin{lemma}\label{l.goodBox}   Let $ \lambda = 2 ,\dotsc, 6$.  
Suppose that $ V$ is $ (2\lambda, \vartheta  , 4)$-Uniform, where 
$ \vartheta < \mathbb P (V\mid T_4) ^{ 2 \cdot 3 ^{\lambda }}$.  Let $ \Omega \subset \Omega _{3\to \lambda } $, 
where the value of $ \lambda $ is half of the uniformity assumption imposed on $ V$. 
Let $ \{f _{\omega } \mid  \Omega \}$ be a selection of functions which are
either equal to $ V$ or a fixed function $ f$ which 
is supported on V and bounded by one in absolute value.  (In application, $ f$ will be a balanced function.)
\begin{enumerate}
\item   Suppose that  there is an $ \omega _0\in \Omega $ with $ f _{\omega _0}=f$, 
and $ \omega _0 (1)\neq \omega (1)$ for all other $ \omega \in \Omega $ with $ f _{\omega
}=f$. 
Then, we have the estimate 
\begin{equation}\label{e.1Box} 
\abs{ \operatorname L(f _{\omega } \mid  \Omega )} 
< 2 \operatorname L(V\mid  \Omega ) \cdot 
\Biggl[ O(  \vartheta )+\frac { \mathbb E _{x_2 , x_3 \in S _{2,3}} \norm f . \Box ^1 S_1 .  ^2  } 
 { \mathbb E _{x_2 , x_3 \in S _{2,3}} \norm V . \Box ^1 S_1 .  ^2  }  \Biggr] ^{1/2} \,. 
\end{equation}

\item   Suppose that  there is an $ \omega _0\in \Omega $ with $ f _{\omega _0}=f$, 
and $ (\omega _0 (1), \omega _0 (2))\neq (\omega (1), \omega (2))$ 
for all other $ \omega \in \Omega $ with $ f _{\omega
}=f$. 
Then, we have the estimate 
\begin{equation}\label{e.2Box} 
\abs{ \operatorname L(f _{\omega } \mid  \Omega )} 
< 4 \operatorname L(V\mid  \Omega ) \cdot 
\Biggl[ O(  \vartheta )+\frac { \mathbb E _{ x_3 \in S _{2,3}} \norm f . \Box ^{1,2} S_{1,2} .  ^{4} } 
 { \mathbb E _{x_3 \in S _{2,3}} \norm V . \Box ^{1,2} S_{1,2} . ^{4} }  \Biggr] ^{1/4}\,. 
\end{equation}

\item If there is at least one $ \omega_0 \in \Omega $ with $ f _{\omega_0 }=f$, we have 
\begin{equation}\label{e.3Box} 
\abs{ \operatorname L(f _{\omega } \mid  \Omega )} 
< 8 \operatorname L(V\mid  \Omega ) \cdot 
\Biggl[O(  \vartheta )+\frac { \mathbb E _{ x_3 \in S _{2,3}} \norm f . \Box ^{1,2,3} S_{1,2,3} .  ^{8} } 
 { \mathbb E _{x_3 \in S _{2,3}} \norm V . \Box ^{1,2,3} S_{1,2,3} . ^{8} }  \Biggr] ^{1/8}\,. 
\end{equation} 
\end{enumerate}

\end{lemma}
%%%%%%%%%%%%%%%%%%%%%%%%%%%%%% LEMMA LEMMA LEMMA

Of course the  estimate \eqref{e.3Box} applies in the first two cases of the Lemma.  But we will 
be in situations, in the proof of Lemma~\ref{l.Tbox}, 
where we do not wish to use the estimate \eqref{e.3Box}. 

We remark that one could read the proof of Lemma~\ref{l.3dvon} in \S~\ref{s.von} before the one below. 
This proof in \S~\ref{s.von} is independent of the proof below. It treats a more complicated 
situation, in that all the $ T _{j}$ have to be considered, but is only discussed in a single concrete 
instance.

%%%%%%%%%%%%%%%%%%%%%%%%%%%%%% PROOF PROOF PROOF
\begin{proof}
We can read off a good estimate for $ \operatorname L(V\mid  \Omega )
$ from \eqref{e.Uuniform}, in all cases $ (1)$---$(3)$ above. 
For each of the three cases, we assume that the choice of $ \omega _0$ specified in each of the 
three cases satisfies $ \omega _0\equiv 0$.

In case $ (1)$, we will apply the Cauchy-Schwartz inequality 
in all \emph{other } variables. To set notation for this,  let 
\begin{equation*}
 \Omega_{1\to0} = \{\omega \in \Omega \mid \omega (1)=0\}\,, 
 \qquad 
  \Omega_{1\not\to 0} = \{\omega \in \Omega \mid \omega (1)\neq 0 \}\,, 
\end{equation*}
and let $ \mathbf X' = \{x _{j} ^{\ell }
\mid 1 \le j \le 3\,,\, 0 \le \ell \le \lambda -1 \}- \{x _1 ^0\}$.  Then, we apply the Cauchy-Schwartz inequality to 
estimate 
\begin{gather}\label{e.U11}
\abs{ \operatorname L(f _{\omega } \mid \Omega )} 
\le 
\bigl[ \operatorname L(V \mid   \Omega _{1\not\to0} )  \cdot  W_1 \bigr] ^{1/2} 
\\ \label{e.U1=}
\begin{split}
W_1 &= \mathbb E _{x _{j} ^{\ell } \in \mathbf X' } 
\prod _{\omega'  \in \Omega _{1\not\to0}  } V (x _{1,2,3} ^{\omega'  }) 
\ABs{ \mathbb E _{x_1 ^{0}\in S_1}   
\prod _{\omega \in \Omega _{1\to 0}} f _{\omega } (x _{1,2,3} ^{\omega })
} ^2 
\end{split}
\end{gather}

We continue the analysis of $ W_1$.  It follows from the assumption in 
part (1) of the Lemma, that $ \omega _0\in \Omega _1$, and $ f _{\omega _0}=f$, but 
for all other choices of $ \omega \in \Omega _{1\to 0}$ we have $ f _{\omega }= V$. 
In order to expand the square of the expectation, using \eqref{e.elem}, let us define 
a new class of maps as follows.   For $ \omega \in \Omega _1$, define 
\begin{gather} \notag 
\overline \omega (j) = 
\begin{cases}
\omega (j) & j\neq 1 
\\
\lambda  & j =1 
\end{cases} 
\\ \label{e.OO4}
\begin{split}
&\Omega _{1\to \lambda } = \{\overline \omega \mid \omega \in \Omega _{1\to0}\}\,, 
\qquad 
\Omega ^1=  \Omega _{1\not\to0} \cup   \Omega _{1\to 0} \cup \Omega _{1\to \lambda }\,, 
\\
&\Omega _{\{1\}\to\{0,\lambda-1\}} = \{ \omega \in \Omega ^1 \,, \omega (1)=0\}\,. 
\end{split}
\end{gather}
Notice that $ \Omega _{\{1\}\to\{0,\lambda-1\}}= \{\omega _0 \,, \, \overline \omega _0\}$, by assumption on $ \Omega $ 
that holds in this case.

Here and below, we are expanding the set $ \Omega $.  We take $ f _{\omega }= V$ for all $ \omega \not \in \Omega $.

We can write 
\begin{align} \label{e.use}
W_1 &= \mathbb E _{x _{j} ^{\ell } \in \mathbf X' }  \mathbb E _{x_1^0,x_1^\lambda \in S_1}  
\prod _{\omega'  \in \Omega _{1\not\to0} \cup \Omega _{1\to 4} } V (x _{1,2,3} ^{\omega'  }) 
\prod _{\omega \in \Omega  _{\{1\}\to\{0,\lambda-1\}}  } f _{\omega } (x _{1,2,3} ^{\omega }) 
\\
& =  \label{e.donotuse}
 \mathbb E _{\substack{x_1^0,x_1^\lambda \in S_1\\  x_{2,3} ^{0,0}\in S_{2,3}}} 
 f (x _{1,2,3} ^{\omega _0})  f (x _{1,2,3} ^{\overline\omega _0 }) \cdot 
Z [  \Omega  _{\{1\}\to\{0,\lambda-1\}}  \;:\;   \Omega^1 - \Omega  _{\{1\}\to\{0,\lambda-1\}} ] \,,  
\end{align}
where the last term is defined in \eqref{e.Zdef}.

It follows from Lemma~\ref{l.Zvar} that $ Z [  \Omega  _{\{1\}\to\{0,\lambda-1\}}  \;:\;   \Omega^1 - \Omega  _{\{1\}\to\{0,\lambda-1\}}  ]$ 
is essentially constant on $ V (x _{1,2,3} ^{\omega _0}) V (x _ {1,2,3} ^{\overline \omega _0})$. Namely,  
\begin{equation} \label{e.mean}
 \begin{split}
 \mathbb E\bigl( Z [  \Omega  _{\{1\}\to\{0,\lambda-1\}}  \;:\;   \Omega^1 - \Omega  _{\{1\}\to\{0,\lambda-1\}}  ]
& \mid  
 V (x _{1,2,3} ^{\omega _0}) V (x _ {1,2,3} ^{\overline
\omega _0}\bigr)
%\\&
\stackrel {u} =  \frac { \operatorname L(V\mid    \Omega ^{1})}  
{ \operatorname L  (V,V \mid \Omega  _{\{1\}\to\{0,\lambda-1\}} )} \,. 
 \end{split}
\end{equation}
The implied $ \kappa $ in the `$\stackrel {u} =  $' is $ \kappa =\sqrt \vartheta $, see Definition~\ref{d.=}.  Similar 
comment applies to other uses of the the symbol `$ \stackrel {u} = $' below. 
And the variance of $ Z [  \Omega  _{\{1\}\to\{0,\lambda-1\}}  \;:\;  \Omega^1 - \Omega  _{\{1\}\to\{0,\lambda-1\}}  ]$ is very small.  
Note that $\operatorname L  (V,V \mid \Omega  _{\{1\}\to\{0,\lambda-1\}} )= \mathbb E _{x_2 , x_3 \in S _{2,3}} \norm V . \Box ^1 S_1 . 
^2 $,  we can estimate 
\begin{equation}\label{e.mmean}
W_1 \le 2 
{\operatorname L(V\mid    \Omega ^{1})} \Biggl[ O( \sqrt \vartheta )+ 
\frac   {  \mathbb E _{x_2 , x_3 \in S _{2,3}} \norm f  . \Box ^1 S_1 . ^2    }
{ \mathbb E _{x_2 , x_3 \in S _{2,3}} \norm V . \Box ^1 S_1 .  ^2 }\Biggr]\,. 
\end{equation}

We combine \eqref{e.U11}---\eqref{e.mmean}, to conclude that 
\begin{align*} 
\Abs{ \operatorname L(f _{\omega } \mid \Omega )} 
&\le 
2
\bigl[ \operatorname L(V \mid   \Omega _{1\not\to0} )  \cdot 
{ \operatorname L(V\mid   \Omega ^{1} )} 
\bigr] ^{1/2}  
\times 
\Biggl[  O( \sqrt \vartheta )+ 
\frac   {  \mathbb E _{x_2 , x_3 \in S _{2,3}} \norm f  . \Box ^1 S_1 .^2 } 
{ \mathbb E _{x_2 , x_3 \in S _{2,3}} \norm V . \Box ^1 S_1 . ^2   }\Biggr] ^{1/2} 
\,. 
\end{align*}
And so the proof of \eqref{e.1Box} will follow from the inequality 
\begin{equation}\label{e.1Box1}
\begin{split}
 \operatorname L(V \mid   \Omega _{1\not\to0} )  &\cdot 
\operatorname L(V\mid  \Omega ^{1})
\le 2  
\operatorname L _{\Omega } (V\mid \Omega ) ^2 \,. 
\end{split}
\end{equation}
This is Conservation of Densities Proposition, Proposition~\ref{p.con}. 

\medskip 

We turn to the proof of the second part, namely  \eqref{e.2Box}. 
The initial stage of the argument follows the lines of the argument above.  Namely, we use the estimate 
\eqref{e.U11} and \eqref{e.U1=}.  The term $ W_1$ is expanded as in \eqref{e.use}, with the same notation 
that we have in \eqref{e.OO4}.   But, under the assumptions on $ \Omega $ that hold in this case, 
$ \Omega  _{\{1\}\to\{0,\lambda-1\}} $ need not consist of just two maps $ \omega $.   

We apply the Cauchy-Schwartz inequality to $ W_1$.  To do this, we make these definitions, recalling that 
$ \Omega ^{1}$ is defined in \eqref{e.OO4}. 
\begin{gather*}
\Omega _{2\not\to 0} ^{1} = \{\omega \in \Omega ^{1}\mid \omega (2)\neq 0\}\,, 
\quad 
\Omega _{2\to 0} ^{1} = \{\omega \in \Omega ^{1}\mid \omega (2)= 0\}\,, 
\\
\mathbf X'' = \{x_1 ^{\ell } \mid 0 \le \ell \le \lambda \} \cup \{x_2 ^{\ell } \mid 1\le \ell \le \lambda -1 \} 
\cup \{x_3 ^{\ell } \mid 0\le \ell \le \lambda -1  \}\,. 
\end{gather*}
Here, the point is that the only variable omitted from $ \mathbf X''$ is $ x_2 ^{0}$.  Then, we can 
estimate 
\begin{gather}\label{e.oU1<}
W_1 \le \bigl[ \operatorname L (V\mid \Omega _{2\not\to 0} ^{1}) \cdot  W_2
\bigr] ^{1/2} 
\\ \label{e.oU2def} 
W_2 = \mathbb E _{x _{j} ^{\ell } \in \mathbf X '' } 
\prod _{\omega \in \Omega ^{1} _{2\not\to 0}} V (x _{1,2,3} ^{\omega }) 
\ABs{\mathbb E _{x_2 ^{0}\in S_2}  \prod _{\omega \in \Omega _{2\to 0}}  f _{\omega } (x _{1,2,3} ^{\omega }) } ^2 \,.  
\end{gather}

To expand the square in the definition of $ W_2$, we set 
\begin{gather} \notag 
\widetilde  \omega (j) = 
\begin{cases}
\omega (j) & j\neq 2 
\\
\lambda  & j =2 
\end{cases} 
\\ \label{e.OOO4}
\begin{split}
&\Omega ^1 _{2\to \lambda } = \{\overline \omega \mid \omega \in  \Omega ^1_{2\to0}\}\,, 
\qquad 
\Omega ^2=  \Omega^1 _{2\not\to0} \cup   \Omega^1 _{2\to 0} \cup \Omega ^1_{2\to \lambda }\,, 
\\
&\Omega_{\{1,2\}\to\{0,\lambda-1\}} = \bigl\{\omega  \in \Omega ^{2} \mid \omega (1), \omega (2) \in \{0,\lambda-1\}\bigr\}\,. 
\end{split}
\end{gather}
Observe that $ \Omega_{\{1,2\}\to\{0,\lambda-1\}} = \{\omega _0 \,,\, \overline \omega _{0} \,,\, \widetilde \omega _0 \,,\,, 
\overline {\widetilde \omega }_0
\}$. 
Then, we can write 
\begin{equation}\label{e.oU2=} 
W_2= 
\mathbb E _{x_j ^{\ell } \in \mathbf Y ''} \prod _{\omega \in \Omega_{\{1,2\}\to\{0,\lambda-1\}}} f (x _{1,2,3} ^{\omega } ) 
\times 
Z [\Omega_{\{1,2\}\to\{0,\lambda-1\}} \;:\;  \Omega ^2 - \Omega_{\{1,2\}\to\{0,\lambda-1\}}] \,. 
\end{equation}
where $ \mathbf Y''= \{x_1^0,x_1^\lambda ,x_2^0,x_2^\lambda ,x_3^{0}\}$, and $ Z [\Omega_{\{1,2\}\to\{0,\lambda-1\}} \;:\;  \Omega ^2 -
\Omega_{\{1,2\}\to\{0,\lambda-1\}}]$ is defined in \eqref{e.Zdef}.   (We assumed that $ \omega _0 \equiv 0$.) 

Using Lemma~\ref{l.Zvar}, and the the assumption of $ (2\lambda, \vartheta ,4)$-uniformity on $ V$, we can estimate 
\begin{align*}
\mathbb E _{x _{j} ^{\ell }\in \mathbf Y''} 
\bigl( Z [\Omega_{\{1,2\}\to\{0,\lambda-1\}} \;:\;  \Omega ^2 - \Omega_{\{1,2\}\to\{0,\lambda-1\}}] 
& 
\mid \prod _{\omega \in \Omega_{\{1,2\}\to\{0,\lambda-1\}}} V (x _{1,2,3} ^{\omega } ) 
\bigr)
\\&
\stackrel u = 
\frac { \operatorname L (V\mid \Omega ^2 )}  {\operatorname L (V\mid\Omega_{\{1,2\}\to\{0,\lambda-1\}} )}
\end{align*}
and the conditional variance of $ Z [\Omega_{\{1,2\}\to\{0,\lambda-1\}} \;:\; \Omega ^2 - \Omega_{\{1,2\}\to\{0,\lambda-1\}}]$  is 
very small.  Thus, we can estimate 
\begin{equation}\label{e.W2=}
W_2 = 2
 \operatorname L (V\mid \Omega ^2 ) \times  
\Bigl[ O( \sqrt \vartheta )+ 
\frac { \mathbb E _{x_3^0\in S_3} \norm f . \Box ^{1,2} S _{1,2}. ^{4}}
{ \mathbb E _{x_3^0\in S_3} \norm V . \Box ^{1,2} S _{1,2}. ^{4}} \Bigr]\,. 
\end{equation}

Combining \eqref{e.U11}, \eqref{e.U1=}, \eqref{e.oU1<}, \eqref{e.oU2def}, and \eqref{e.W2=}, we see that 
\begin{align*}
\Abs{\operatorname L (f_\omega \mid \Omega )} 
& \le 2 
 \operatorname L(V \mid   \Omega _{1\not\to0} ) ^{1/2} 
 \cdot 
  \operatorname L (V\mid \Omega _{2\not\to 0} ^{1}) ^{1/4} \cdot 
  \operatorname L (V\mid \Omega ^2 ) ^{1/4} 
  \\
  & \qquad   \times 
  \Biggl[ 
  \frac { \mathbb E _{x_3^0\in S_3} \norm f . \Box ^{1,2} S _{1,2}. ^{4}}
{ \mathbb E _{x_3^0\in S_3} \norm V . \Box ^{1,2} S _{1,2}. ^{4}} 
  \Biggr] ^{1/4} \,. 
\end{align*}
The last step in the proof of \eqref{e.2Box} is to verify that  
\begin{equation}\label{e.LL}
 \operatorname L(V \mid   \Omega _{1\not\to0} ) ^{1/2} 
 \cdot 
  \operatorname L (V\mid \Omega _{2\not\to 0} ^{1}) ^{1/4} \cdot 
  \operatorname L (V\mid \Omega ^2 ) ^{1/4} 
  \le 2
 \operatorname  L (V\mid \Omega )\,. 
\end{equation}
This is again the Conservation of Densities Proposition, Proposition~\ref{p.con}.

\medskip 

We turn to the third point of the Lemma, namely the inequality \eqref{e.3Box} is true. 
We can use earlier parts of the argument.  
Let us combine \eqref{e.U11}, \eqref{e.use}, \eqref{e.oU1<}, and   \eqref{e.oU2def}.  We have 
\begin{equation}\label{e.oU3}
\abs{\operatorname L (f _{\omega } \mid \omega \in \Omega ) } 
\le 
2 
\operatorname L (V\mid \Omega _{1\not\to0})  ^{1/2} 
\cdot 
\operatorname L (V\mid \Omega _{2\not\to0} ^{1})  ^{1/4} 
\cdot 
W_2 ^{1/4} \,, 
\end{equation}
where $ W_2$ is defined in \eqref{e.oU2=}.  

The strategy is to repeat an application of the Cauchy-Schwartz inequality in all variables except $ x_3 ^{0}$. 
To do this, we define 
\begin{gather*}
\Omega _{3\not\to 0} ^{2} = \{\omega \in \Omega ^{2}\mid \omega (3)\neq 0\}\,, 
\quad 
\Omega _{3\to 0} ^{2} = \{\omega \in \Omega ^{2}\mid \omega (3)= 0\}\,, 
\\
\mathbf X''' = \{x_j ^{\ell } \mid j=1,2\,,\  0 \le \ell \le \lambda \} 
\cup \{x_3 ^{\ell } \mid 1\le \ell \le \lambda -1  \}\,. 
\end{gather*}
Here, the point is that the only variable omitted from $ \mathbf X'''$ is $ x_3 ^{0}$.  Then, we can 
estimate 
\begin{gather}\label{e.oU3<}
W_2 \le \bigl[ \operatorname L (V\mid \Omega _{3\not\to 0} ^{2}) \cdot  W_3
\bigr] ^{1/2} 
\\ \label{e.oU3def} 
W_3 = \mathbb E _{x _{j} ^{\ell } \in \mathbf X ''' } 
\prod _{\omega \in \Omega ^{2} _{3\not\to 0}} V (x _{1,2,3} ^{\omega }) 
\ABs{\mathbb E _{x_3 ^{0}\in S_3}  \prod _{\omega \in \Omega _{3\to 0} ^2}  f _{\omega } (x _{1,2,3} ^{\omega }) } ^2 \,.  
\end{gather}
In the product over $ \Omega _{3\to0} ^2$, it is important to observe that if $ f _{\omega }=f$, it must follow that 
$ (\omega (1), \omega (2))\in \{0, \lambda \} ^{1,2}$.  For if this is not the case, an earlier step would have 
switched $ f _{\omega }$ to $ V$.  

To expand the square, we define 
\begin{gather*} \notag 
\underline \omega (j) = 
\begin{cases}
\omega (j) & j\neq 3 
\\
\lambda  & j =3 
\end{cases}   
\\ 
\begin{split}
&\Omega _{3\to \lambda } = \{\underline \omega \mid \omega \in \Omega ^2, \ \omega (3)=0\}\,, 
\qquad 
\Omega ^3=  \Omega  ^2 \cup \Omega _{3\to \lambda }\,, 
\\
&\Omega _{ \{1,2,3\} \to \{0, \lambda \}} =  \{0, \lambda \} ^{\{1,2,3\}} \,. 
\end{split}
\end{gather*}
Then, we can write  
\begin{align*}
W_3 & = 
\mathbb E _{ x^0 _{1,2,3}, x ^\lambda _{1,2,3} \in S _{1,2,3}} 
\prod _{\omega \in \Omega _{ \{1,2,3\} \to \{0, \lambda \}} } 
f (x ^{\omega } _{1,2,3}) \times 
Z[ \Omega _{ \{1,2,3\} \to \{0, \lambda \}} \;:\; \Omega ^3 - \Omega _{ \{1,2,3\} \to \{0, \lambda \}} ] \,. 
\end{align*}

Now, the term $ Z$ is nearly constant, by Lemma~\ref{l.Zvar}, and we have 
\begin{align*}
\mathbb E \Bigl( Z[ \Omega _{ \{1,2,3\} \to \{0, \lambda \}} \;:\; \Omega ^3 - \Omega _{ \{1,2,3\} \to \{0, \lambda \}} ] 
& \mid \prod _{\omega \in \Omega _{ \{1,2,3\} \to \{0, \lambda \}} } V \Bigr) 
= \frac { \operatorname L ( V \mid \Omega ^{3} )} {\operatorname L (V\mid \Omega _{ \{1,2,3\} \to \{0, \lambda \}} )} 
\end{align*}
Therefore, we can estimate 
\begin{equation}\label{e.W3}
W_3 = \Bigl[ O (\sqrt \vartheta ) + 
\frac {\norm f . \Box ^{1,2,3} S _{1,2,3}. ^{8}} {\norm V . \Box ^{1,2,3} S _{1,2,3}. ^{8}}
\Bigr] \times { \operatorname L ( V \mid \Omega ^{3} )} \,. 
\end{equation}

Combine \eqref{e.oU3}, \eqref{e.oU3<}, \eqref{e.oU3def}, and \eqref{e.W3} to conclude that 
\begin{equation}\label{e.W33}
\begin{split}
\abs{\operatorname L (f _{\omega } \mid \omega \in \Omega ) } 
&\le 
2  
\operatorname L (V\mid \Omega _{1\not\to0})  ^{1/2} 
\cdot 
\operatorname L (V\mid \Omega _{2\not\to0} ^{1})  ^{1/4} 
\cdot 
 \operatorname L (V\mid \Omega _{3\not\to 0} ^{2}) ^{1/8} 
\\ & \qquad \times  
 { \operatorname L ( V \mid \Omega ^{3} )} ^{1/8}
\cdot 
 \Biggl[ O (\sqrt \vartheta ) + 
\frac {\norm f . \Box ^{1,2,3} S _{1,2,3}. ^{8}} {\norm V . \Box ^{1,2,3} S _{1,2,3}. ^{8}}
\Biggr] ^{1/8}\,. 
\end{split}
\end{equation}
Therefore, it remains for us to check that 
\begin{equation}\label{e.33}
\operatorname L (V\mid \Omega _{1\not\to0})  ^{1/2} 
\cdot 
\operatorname L (V\mid \Omega _{2\not\to0} ^{1})  ^{1/4} 
\cdot 
 \operatorname L (V\mid \Omega _{3\not\to 0} ^{2}) ^{1/8} 
 \cdot 
 { \operatorname L ( V \mid \Omega ^{3} )} ^{1/8}
\le 2 \operatorname L (V \mid \Omega )\,. 
\end{equation}
This again follows from Proposition~\ref{p.con}. 
\end{proof}
\section{Linear Forms for the Analysis of Corners} \label{s.formsCorners}

In this section, we reprise the initial portion of the previous section, though 
our needs are not quite a significant.  
For the uses of this discussion, let us make the definition 
\begin{equation}\label{e.tildeTell}
\widetilde T _{\ell  } = \prod _{\substack{1\le j<k\le 4\\ j,k\neq \ell }} R _{j,k} \,. 
\end{equation}
This is the same definition as for  $ T _{\ell }$, but the set $ S _{\ell }$ is missing. 

For $ \Omega \subset \Omega _{4\to \lambda }$, where $ \lambda \le 3 $, and choices of functions 
$ F _{\omega } \in \{ T _{\ell } \,,\, \widetilde T _{\ell } \mid 1\le \ell \le 4 \}$, we have the 
linear form 
\begin{equation*}
\Lambda (F _{\omega } \mid \Omega )= 
\mathbb E  _{\substack{x _{1,2,3,4} ^{\lambda } \in S _{1,2,3,4}\\ 0\le \lambda \le 3 }}
\prod _{\omega \in \Omega } F _{\omega } (x _{1,2,3,4} ^{\omega })\,. 
\end{equation*}
Here, any $ S_j$ that occurs in this expectation is composed with $ \lambda _j$. 
Our first Lemma states that we can easily estimate the values of these forms. 

%%%%%%%%%%%%%%%%%%%%%%%%%%%%%% LEMMA LEMMA LEMMA
\begin{lemma}\label{l.==L} For $ \Omega $ and choices of $ F _{\omega  }$ as above we have 
\begin{gather*}
\Lambda (F _{\omega } \mid \Omega ) \stackrel u = 
\prod _{\ell =1} ^{4} \delta _{\ell } ^{\Phi (\ell )} \cdot \prod _{1\le j<k\le 4} \delta _{j,k} ^{\Psi (j,k)}
\\
\Phi (\ell ) = \lvert\{\omega \mid F_\omega = T _{\ell } \}\rvert\,,
\qquad 
\Psi (j,k) = \lvert \{ \omega \vert _{j,k} \mid \omega \in \Omega \} \rvert\,. 
\end{gather*}
In the last display we are counting the number of distinct maps there are when $ \omega $ is restricted to 
 the sets $ \{j,k\}$.
\end{lemma}
%%%%%%%%%%%%%%%%%%%%%%%%%%%%%% LEMMA LEMMA LEMMA

%%%%%%%%%%%%%%%%%%%%%%%%%%%%%% PROOF PROOF PROOF
\begin{proof}
We have  
\begin{equation*}
\prod _{\omega \in \Omega } F _{\omega } (x _{1,2,3,4} ^{\omega })
= 
\prod _{\ell =1} ^{4}
\prod _{\omega \in \phi (\ell )} S _{\ell } \circ \lambda (x _{1,2,3,4} ^{\omega }) 
\times 
\prod _{1\le j<k\le 4} 
\prod _{\omega \in \psi (j,k )} S _{j,k} \circ \lambda (x _{j,k} ^{\omega }) 
\end{equation*}
where $ \psi (\ell ) = \{\omega \mid F _{\omega }= T _{\ell }\}$, and $ \psi (j,k)=
\{ \omega\vert _{j,k} \mid \omega \in \Omega \}$.  
The Lemma then follows from the assumptions of admissibility, namely \eqref{e.Ad1Box} and \eqref{e.Ad2Box}, 
with application of \eqref{e.gcsi1}. 
\end{proof}
%%%%%%%%%%%%%%%%%%%%%%%%%%%%%% PROOF PROOF PROOF

We need an analog of the Conservation of Densities Lemma, Proposition~\ref{p.con}.  
Let $ \Omega \subset \Omega _{4 \to 3}$, and assume that for the set $ \Omega _{1\to0}$ below is not empty. 
\begin{equation*}
\Omega _{1\to0} = \{\omega \in \Omega \mid \omega (1)=0\,,\, F _\omega \not= \widetilde T_1\}\,, 
\qquad 
\Omega _{1\not\to} = \Omega-\Omega _{1\to0}\,. 
\end{equation*}
Here, we exclude $ \widetilde T_1$, as its expectation does not include any $ \delta _1$. 

Consider the estimate below obtained by applying the Cauchy-Schwartz inequality in all variables except $ x_1 ^0$.  
\begin{align}  \label{e.LFcon1}
\Lambda (F_\omega\mid \Omega ) &\le  \bigl[ \Lambda (\Omega _{1\not\to0} ) \cdot U_2 \bigr] ^{1/2} 
\\ \label{e.LFcon2}
U_2 & = \mathbb E \prod _{\omega \in \Omega _{1\not\to0}} F _{\omega }({x _{1,2,3,4} ^{\omega }}) 
\cdot \Abs{\mathbb E _{x_1 ^{0} \in S_1}  \prod _{\omega \in \Omega _{1\to0}} 
\prod _{\omega \in \Omega _{1\to0}} F _{\omega } ({x _{1,2,3,4} ^{\omega }}) } ^2 \,. 
\end{align}
Use \eqref{e.elem} to write the last term as $ U_2 = \Lambda (F_\omega \mid \Omega ^{1})$, where 
we define 
\begin{align} \label{e.LFcon3}
\overline \omega  (j) =
\begin{cases}
 \lambda & j=1 
 \\
 \omega (j) & j=2,3,4
\end{cases}
\\ \label{e.LFcon4}
\Omega ^{1} = \Omega _{1\not\to0} \cup \{\omega , \overline \omega \mid \omega \in \Omega _{1\to0}\}\,. 
\end{align}
And we define $ F _{\overline \omega }= F _{\omega }$. 

\begin{conserve2} \label{p.con2} If 
If  $ \Omega \subset \Omega _{3\to \lambda -1}$, 
with the notation in \eqref{e.LFcon1}---\eqref{e.LFcon4} we have the equality 
\begin{equation}\label{e.LFcon5}
\Lambda (F_\omega\mid \Omega ) \stackrel u = \Lambda (F_\omega\mid \Omega _{1\not\to0}) ^{1/2} \cdot 
\Lambda (F_\omega\mid \Omega ^1 ) ^{1/2} \,. 
\end{equation}
\end{conserve2}

%%%%%%%%%%%%%%%%%%%%%%%%%%%%%% PROOF PROOF PROOF
\begin{proof}
Each $ \omega \in \Omega $  be such that it contributes $ 1$ to the density $ \delta _{\ell }$, 
for $ 2\le \ell \le 4$ on the left-hand-side of \eqref{e.LFcon5}. 
Thus, $ \omega \in \Omega _{1\not\to0}$, and it contributes a $ 1/2$ to this same density 
in each of the two terms on the right-hand side.  Let $ \omega \in \Omega _{1\to0}$. Then, it 
contributes a $ 1$ to the density of $ \delta _1$ on the left-hand side, while on the right hand-side, 
there is no contribution from the first term, while the second term contributes a $ 2 \cdot 1/2=1$, since the 
there is a new variable $ x_1^4$.  

If one considers a density $ \delta _{j,k}$ where $ 2\le j<k\le 4$, it is accounted for much as the case of 
$ \delta _2$ above.  And a density $ \delta _{1,j}$, with $ j=2,3,4$, is accounted for as is $ \delta _1$ above. 
\end{proof}
%%%%%%%%%%%%%%%%%%%%%%%%%%%%%% PROOF PROOF PROOF

 This Conservation of Densities has an essentially equivalent 
formulation, also important to us, that we give here.  With the notation of \eqref{e.LFcon1}---\eqref{e.LFcon4}, 
set 
\begin{equation} \label{e.LFZdef}
Z [ \Omega _{1\not\to0} \;:\; \Omega _{1\to0} ] = \mathbb E _{x_1 ^{0} \in S_1} 
\prod _{\omega \in \Omega _{1\to0}} F_\omega  (x _{1,2,3,4} ^{\omega }) 
\end{equation}

%%%%%%%%%%%%%%%%%%%%%%%%%%%%%% LEMMA LEMMA LEMMA 
\begin{lemma}\label{l.LFZvar}  
For all choices of $ \Omega  \subset \Omega _{4\to 3 }$ as above, we have 
\begin{gather}\label{e.LFZvar}
\begin{split}
\operatorname {Var} _{x_j ^{\ell }\in \Omega } \Bigl(Z[\Omega _{1\not\to0} \;:\;  \Omega _{1\to0}] 
&\mid   \prod _{\omega \in \Omega _{1\not\to0}}  F _{\omega } (x _{1,2,3,4} ^{\omega })\Bigr)
\\&
\le  K\sqrt{\vartheta  } \cdot 
\Bigl[ \mathbb E  
 \Bigl(Z[\Omega _{1\not\to0} \;:\;  \Omega _{1\to0}] \mid  
\prod _{\omega \in \Omega _{1\not\to0}}  F _{\omega } (x _{1,2,3,4} ^{\omega })\Bigr) \Bigr] ^2 \,. 
\end{split}
\end{gather}
Here, $ K$ is an absolute constant.
\end{lemma}
%%%%%%%%%%%%%%%%%%%%%%%%%%%%%% LEMMA LEMMA LEMMA

Of course the conditional expectation of $ Z$ can be computed.  

%%%%%%%%%%%%%%%%%%%%%%%%%%%%%% PROOF PROOF PROOF
\begin{proof}
We use the standard formula for the variance of a random variable $ W$ supported on a set $
Y$ given in \eqref{e.elemVar}. 
The conditional variance will be small if we have 
\begin{equation*}
\mathbb E  
 \Bigl(Z[\Omega _{1\not\to0} \;:\;  \Omega _{1\to0}] ^2  \mid  
\prod _{\omega \in \Omega _{1\not\to0}} F_\omega  (x _{1,2,3,4} ^{\omega })\Bigr)
\stackrel u = 
\mathbb E  
 \Bigl(Z[\Omega _{1\not\to0} \;:\;  \Omega _{1\to0}] \mid  
\prod _{\omega \in \Omega _{1\not\to0}} F_\omega  (x _{1,2,3,4} ^{\omega })\Bigr) ^2 \,. 
\end{equation*}
But this is a recasting of \eqref{e.LFcon5}. 
\end{proof}
%%%%%%%%%%%%%%%%%%%%%%%%%%%%%% PROOF PROOF PROOF

There is a variant of the inequality \eqref{e.3Box} which holds.  Let us formulate it.  

%%%%%%%%%%%%%%%%%%%%%%%%%%%%%% LEMMA LEMMA LEMMA
\begin{lemma}\label{l.LF3box} Let $ \Omega \subset \Omega _{4\to 3} $, and let $ F _{\omega } \in \{T _{1} , 
T_2,T_3,T_4\}$.  Let $ f _{\omega }$ be a choice of function satisfying $ \lvert  f _\omega \rvert\le F _{\omega } $. 
Then, we have the following inequality.  Suppose, for the sake of simplicity that for $ \omega _0\in \Omega $ we have 
$ F _{\omega _0}= T_1$
\begin{equation}\label{e.LF3Box} 
\abs{\Lambda (f _{\omega } \mid \Omega )}
\le 2 \abs{\Lambda (F _{\omega } \mid \Omega )}
\times \biggl\{ \upsilon + \frac { \norm f _{\omega _0} . \Box ^{2,3,4} H _{2,3,4}. ^{8} } 
{\norm T_1 . \Box ^{2,3,4} H _{2,3,4}. ^{8}} 
\biggr\} ^{1/8}
\end{equation}
\end{lemma}
%%%%%%%%%%%%%%%%%%%%%%%%%%%%%% LEMMA LEMMA LEMMA

In view of the fact that we have the Second Conservation of Densities Proposition, Proposition~\ref{p.con2}, 
and the variance principle Lemma~\ref{l.LFZvar}, the proof of this inequality is just an iteration of the 
proof of \eqref{e.3Box} above, as well as the proof of Lemma~\ref{l.Q} below. Accordingly we omit it.

%%%%%%%%%%%%%%%%%%%%%% SUBSECTION SUBSECTION SUBSECTION
\section{Proof of the von Neumann Lemma}
\label{s.von}

This is a careful application of weighted Gowers-Cauchy-Schwartz inequality, 
which does not seem to follow from any standard inequality in the literature.
  The primary difference with the weighted inequalities of 
the work of Green and Tao, 
\cites{math.NT/0404188,math.NT/0606088}
is the absence of the von Mangoldt function with it's uniformity properties, a difference 
overcome by the enforced uniformity, an argument invented by Shkredov \cite{MR2266965}.

In our setting, the sets $ X_ a $ will most frequently be $ H$, the copy of the finite 
field.  The set $ U$ will for the most part be $ \{1,2,3,4\}$, though there are 
larger sets $ U$, as large as $ 24$ elements, that occurs in the analysis of different terms below.

We introduce the following 4-linear form. For four functions 
$ f _j  \;:\;  H \times H \times H \to \mathbb C $,  for $ 1\le j\le 4$, define 
\begin{equation}\label{e.Q} 
\begin{split}
{\operatorname Q}(f_1,f_2,f_3,f_4) \eqdef
\mathbb E_{\substack{y,x_j \in H\\  1 \le j \le 3 }} &f_4(x_1,x_2,x_3)f_3(x_1,x_2,x_3 + y)
\\& \qquad \times f_2(x_1, x_2 + y, x_3) f_1(x_1 + y, x_2, x_3)
\end{split}
\end{equation}
If $ A\subset H \times H \times H$, it follows that $ {\operatorname Q}(A,A,A,A)$ is the expected number 
of corners in $ A$.   It is an important remark that this is defined as an average over copies of $ H$, whereas 
earlier sections have been defined over e.\thinspace g.\thinspace $ S _{1,2,3,4}$.  This fact introduces 
extra factors of $ \delta _{\ell }$ below. 

We are deliberately choosing a definition that is slightly asymmetric with respect to
the subscripts on the $ f_j$ on the right above, to make the next display more symmetric.
 Using the change of variables $ y=x_4-(x_1+x_2+x_3)$, this is 
\begin{align*}
 {\operatorname Q}(f_1,f_2,f_3,f_4)
& = \mathbb E_{\substack{x_j \in H\\  1 \le j \le 4 }} \prod _{j=1} ^{4} 
 f _{j} \circ \lambda _{j}\,,
 \\
\lambda _{j} (x_1,x_2,x_3,x_4) &= 
\sum _{k \;:\; k\neq j} x_k \operatorname e_k \,, \qquad 1\le j \le 4 \,. 
 \end{align*}
The point which dominates the analysis below is that the functions $ f _{j} \circ \lambda _{j}$ is 
a function of $   \{x _{\ell }\mid 1\le \ell \neq j \le 4\}$, i.\thinspace e., is not a function of $ x_j$.  

We will write, by small abuse of notation, $ \lambda _{1} (x ^{\omega } _{1,2,3,4})= x ^{\omega }
_{2,3,4}$.  This is allowed, as $ \lambda _1 (x ^{\omega } _{1,2,3,4})$ is not a function of $ x _1 ^{\omega (1)}$. 
This will allow us reduce the complexity of some formulas below.

  We codify the result of the application of the proof of the Gowers-Cauchy-Schwartz Inequality 
for the operator $ \operatorname Q$ into the results of the following Lemma.  This technical result codifies 
the results that we need to understand about the set $ T$, and $ A$ to conclude Lemma~\ref{l.3dvon}.  

In this Lemma, we single out for a distinguished role the function that falls in the last place of $ \operatorname Q$, 
but there is a corresponding estimate for all the other three functions.

%%%%%%%%%%%%%%%%%%%%%%%%%%%%%% LEMMA LEMMA LEMMA
\begin{lemma}\label{l.Q} Let $ \overline  T_j$ either be identically $ T$, or $ \overline T_j=T_j$ for all 
$ 1\le j \le 4$.   Let $ f_j \;:\; \overline T_j \longrightarrow [-1,1] $ be functions.   We have the following 
estimate.  
\begin{gather}\label{e.Q1} 
\Abs{ \operatorname Q (f_1,f_2,f_3,f_4)} 
\le\operatorname U_1 ^{1/2} \cdot\operatorname U _2 ^{1/4} \cdot\operatorname U _3 ^{1/8}  \cdot\operatorname U _4 ^{1/8}\,, 
\\ \label{e.Q2}
\operatorname U_1=\operatorname U_1 (\overline T_1)=\mathbb E  _{x_2,x_3,x_4\in H} \overline T_1 (x_2,x_3,x_4)
\\ \label{e.Q3} 
\operatorname U_2 =\operatorname U_2 (\overline T_2)= 
\mathbb E _{\substack{x_3^0,x_4^0\in H\\ x_1^0, x^1_1\in H}} 
\prod _{\omega \in \{0,1\}^{ \{1\}} \times \{0\} ^{ \{3,4\}}} 
\overline T _{2}  ( x
_{\{1,3,4\} } ^{\omega })\,, 
\\ \label{e.Q4} 
\operatorname U_3 =\operatorname U_3 (\overline T_3) =  
\mathbb E _{\substack{x_4^0\in H\\ x _{ \{1,2\} }^0,x _{ \{1,2\} }^1 \in H_{\{1,2\}} }}
\prod _{\omega \in \{0\} ^{ \{1,2\}} \times \{0\} ^{ \{4\}}} 
\overline  T _3    ( x_{\{1,2,4\}} ^{\omega })\,, 
\\ \label{e.Q5}
\operatorname U_4 = \operatorname U _4 (f_4, \overline T_1, \overline T_2, \overline T_3)
= \mathbb E _{\substack{ x_{\{1,2,3\}}^0,x_{\{1,2,3\}}^1\in H_{\{1,2,3\}}}}   \; 
\operatorname Z \cdot 
\prod _{\omega \in \{0\} ^{ \{1,2,3\}} \times \{0\} ^{ \{4\}} }
 f _{4} ( x_{\{1,2,3\}} ^{\omega }) 
\\ \label{e.QZ} 
\operatorname Z= \operatorname Z(\overline T_1,\overline T_2,\overline T_3) =  
 \mathbb E _{x_4 ^{0} \in H}
\prod _{\omega \in\{0,1\}^{\{1,2,3\}} 
\times\{0\}^{\{4\}}}  \prod _{j=1} ^{3}
  \overline  T_j \circ \lambda _j (x _{1,2,3,4} ^{\omega }) 
\end{gather}
\end{lemma}
%%%%%%%%%%%%%%%%%%%%%%%%%%%%%% LEMMA LEMMA LEMMA

This Lemma makes it clear that we need to understand the linear forms  $ \operatorname U_1,\operatorname U_2,\operatorname U_3$, 
and $ \operatorname Z$ for 
both the $ T_j$ and for $ T$.   

%%%%%%%%%%%%%%%%%%%%%%%%%%%%%% REMARK REMARK REMARK
\begin{remark}\label{r.Z} 
The presence of the term $ Z$ in \eqref{e.U32=} can be seen in the argument of 
\cite{MR2289954}, but it is not needed in Shkredov's approach \cite{MR2266965}.  
However, this term is much more subtle in the three dimensional case.   Similar terms will arise in 
\S~\ref{s.Tbox}, are dealt with systematically in Lemma~\ref{l.Zvar}.
\end{remark}
%%%%%%%%%%%%%%%%%%%%%%%%%%%%%% REMARK REMARK REMARK

%%%%%%%%%%%%%%%%%%%%%%%%%%%%%% PROOF PROOF PROOF
\begin{proof} % [Proof of Lemma~\ref{l.3dvon}.]

The method of proof is to follow the proof of the Gowers-Cauchy-Schwartz inequality, 
especially in the case of \eqref{e.gcsi2}, but keeping track of the 
additional information that follows from terms that are neglected in the usual proofs of this inequality.  
All earlier applications of the Gowers-Cauchy-Schwartz inequality has 
in some sense `lost units of density.'  In the present argument, we recover these 
lost units by the mechanism of the various functions of $ T$ that appear in the 
definitions of $ U _{1}$, $ U _{2}$ and $ U _{3}$ above.

Estimate the left-hand side of \eqref{e.Q1} by 
\begin{gather}\label{e..U1}
\lvert  \operatorname Q (f _{ 1}, f _{ 2} , f _{ 3}, f _{4})
\rvert 
\le \bigl[ U _{1} \cdot U _{1} \bigr] ^{1/2} \,
\\ \label{e..U11}
U _{1}= 
\mathbb E _{x_2,x_3,x_4\in H} 
\lvert   f _{1} \circ \lambda _1\rvert ^2  \le \mathbb E  _{x_2,x_3,x_4\in H} \overline T_1(x_2,x_3,x_4)\,, 
\\ \label{e..U12}
U _{1,2} = \mathbb E _{x_2,x_3,x_4\in H}  \overline T_1 (x _{ \{2,3,4\}})
\ABs{\mathbb E _{x_1}  \prod _{j=1} ^{3} f _{\epsilon (j)} \circ \lambda _{j} x
_{\{1,2,3,4}\} } ^2 
\end{gather}
We use the Cauchy-Schwartz inequality in the variables $ x_2,x_3,x_4$.  The term in \eqref{e..U1} proves \eqref{e.Q2}. 
In the last line, we are using the notation of the general Gowers-Cauchy-Schwartz Inequalities, 
so that $ x_ { \{1,2,3,4\}}= (x_1,x_2,x_3,x_4)$.  This will be helpful 
in the steps below.

For $ U _{1,2}$, we use the elementary fact that 
\begin{equation}\label{e.elem}
\mathbb E _{x\in X}  g (x)\Abs{\mathbb E _{y \in Y} f (x,y)} ^2 
= \mathbb E _{\substack{x\in X\\ y^0, y^1\in Y }} g (x) \prod _{\epsilon = 0} 
f (x, y ^{\epsilon })\,. 
\end{equation}
This is in fact crucial to the proof of the Gowers-Cauchy-Schwartz inequality. 
In particular, it is essential that we insert the $ \overline T_1 (x _{ \{2,3,4\}}) $ 
on the right in \eqref{e..U12}.  Thus, 
\begin{equation}\label{e..U12=}
U _{1,2}
= 
\mathbb E _{\substack{x_2^0,x_3^0,x_4^0\in H\\ x_1^0, x^1_1\in H}}
\overline T_1 (x _{ \{2,3,4\}})
\prod _{\omega \in \{0,1\}^{ \{1\}} \times \{0\} ^{ \{2,3,4\}}} 
 \prod _{j=2} ^{4} f _{\epsilon (j)} \circ \lambda _{j}( x
_{\{1,2,3,4\} } ^{\omega })\,. 
\end{equation}
We refer to this identity as `passing $ x_1$ through the square.'
With this notation, it is clear that the variables $ x_2,x_3,x_4$ will also 
need to `pass through the square'.  

Thus, we write as below, using the Cauchy-Schwartz inequality in the 
variables $ x _1 ^{0}, x_1 ^{1}, x_3^0$, and $ x_4 ^{0}$.  
\begin{gather}\label{e.U2}
U _{1,2}\le \bigl[ U _{2} \cdot U _{2,2} \bigr] ^{1/2} 
\\ \label{e.U21}
U _{2} \le 
\mathbb E _{\substack{x_3^0,x_4^0\in H\\ x_1^0, x^1_1\in H}}
\prod _{\omega \in \{0,1\}^{ \{1\}} \times \{0\} ^{ \{3,4\}}} 
\overline T _{2} \circ \lambda _{2} ( x
_{\{1,2,3,4\} } ^{\omega })
\\ \label{e.U22}
\begin{split}
U _{2,2}&= 
\mathbb E _{\substack{x_3^0,x_4^0\in H\\ x_1^0, x^1_1\in H}} 
\prod _{\omega \in \{0,1\}^{ \{1\}} \times \{0\} ^{ \{3,4\}}} 
\overline  T_2 ( x _{ \{1,3,4\}} ^{\omega } )
\\ & \qquad \times  
\ABs{ \mathbb E _{x_2\in H}  
\overline T_1(x _{ \{2,3,4\}})
\prod _{\omega \in \{0,1\}^{ \{1\}} \times \{0\} ^{ \{2,3,4\}}} 
\prod _{j=3} ^{4} f _{\epsilon (j)} \circ \lambda _{j}( x
_{\{1,2,3,4\} } ^{\omega })
} ^2 
\end{split}
\end{gather}
The term in \eqref{e.U21} is \eqref{e.Q3}.

For the  term \eqref{e.U22}, we write 
\begin{equation}\label{e.U22=}
\begin{split}
U _{2,2}= 
\mathbb E _{\substack{x_3^0,x_4^0\in H\\ x _{ \{1,2\} }^0,x _{ \{1,2\} }^0 \in H_{\{1,2\}} }}
\prod _{\omega \in \{0\} ^{ \{1,2\}} \times \{0\} ^{ \{3,4\}}}  &
\Biggl[\overline T_2 (x _{ \{1,3,4\}} ^{\omega }) \overline T_1 (x _{ \{2,3,4\}} ^{\omega })
\\ & \qquad \times 
\prod _{j=2} ^{4} f _{\epsilon (j)} \circ \lambda _{j}( x
_{\{1,2,3,4\}}^{\omega }) \Biggr]
\end{split}
\end{equation}
We estimate using the Cauchy-Schwartz inequality in the variables 
$ x _{1,2} ^0, x _{1,2} ^{1}$ and $ x_4 ^0$.  
\begin{gather}\label{e.U3}
U _{2,2}\le \bigl[ U _{3} \cdot U _{3,2} \bigr] ^{1/2} \,, 
\\ \label{e.U31}
\begin{split}
U _{3} &= 
\mathbb E _{\substack{x_4^0\in H\\ x _{ \{1,2\} }^0,x _{ \{1,2\} }^1 \in H_{\{1,2\}} }}
\prod _{\omega \in \{0\} ^{ \{1,2\}} \times \{0\} ^{ \{4\}}} 
\overline  T _3  ( x_{\{1,2,4\}} ^{\omega })
\end{split}
\\
\begin{split}
U _{3,2}= 
\mathbb E _{\substack{ x _{ \{1,2\} }^0,x _{ \{1,2\} }^1 \in H_{\{1,2\}} \\ x_4\in H}}
& \ABs{ \mathbb E _{x_3}
\prod _{\omega \in \{0\} ^{ \{1,2\}} \times \{0\} ^{ \{3\}}} \Bigl[ 
\overline T_2( x _{ \{1,3,4\}} ^{\omega } )\overline T_1 (x _{ \{2,3,4\}} ^{\omega })
\\ & \quad \times 
\overline T_3 ( x_{\{1,2,4\}} ^{\omega })
 f _{4} \circ \lambda _{4}( x_{\{1,2,3\}} ^{\omega }) \Bigr]
} ^2  
\end{split}
\end{gather}
The term $ U _{3}$ is \eqref{e.Q4}.

We write $ U _{3,2}$ as follows, after application of \eqref{e.elem}, and 
recalling the definition of $ Z$ in \eqref{e.QZ}. 
\begin{gather}\label{e.U32=}
U _{3,2} = 
\mathbb E _{\substack{ x_{\{1,2,3\}}^0,x_{\{1,2,3\}}^1\in H_{\{1,2,3\}}}}   \; 
Z \cdot 
\prod _{\omega \in \{0\} ^{ \{1,2,3\}} \times \{0\} ^{ \{4\}} }
 f _{4} \circ \lambda _{4}( x_{\{1,2,3,4\}} ^{\omega }) 
\end{gather}
This completes the proof. 

\end{proof}
%%%%%%%%%%%%%%%%%%%%%%%%%%%%%% PROOF PROOF PROOF

We now provide the estimates that the previous Lemma calls for, in the case of the sets $ T_j$.

%%%%%%%%%%%%%%%%%%%%%%%%%%%%%% LEMMA LEMMA LEMMA
\begin{lemma}\label{l.QTj}  For the terms $ U_1, U_2, U_3$ and $ Z$ as defined in \eqref{e.Q2}---\eqref{e.Q4} 
and \eqref{e.QZ}, and $ \overline T_j=T_j$ we have these estimates. 
\begin{gather}\label{e.QTj}
\operatorname Q (T_1,T_2,T_3,T_4) \stackrel u = 
\operatorname U _1 (T_1) ^{1/2}  \cdot  \operatorname U_2 (T_2) ^{1/4} 
\operatorname U _3 (T_3) ^{1/8}  \cdot  \operatorname U_4 (T_4, T_3,T_2, T_1) ^{1/8} \,.  
\end{gather}
The constant $ \vartheta $ in the definition of $ \stackrel u =$, see Definition~\ref{d.U}, 
can be taken to be $ \vartheta = \mathbb P (T\mid H \times H \times H) ^{C}$, 
where $ C$ is a large constant, depending only on $ C _{\textup{admiss}}$ in Definition~\ref{d.admissible}.  
And for $ \operatorname Z (T_1,T_2,T_3)$, we have this  inequalities on conditional variance. 
\begin{gather}
\label{e.ZTjvar} 
\operatorname {Var} 
\Bigl( \operatorname Z(T_1,T_2,T_3) \mid 
\prod _{\omega \in \{0\} ^{ \{1,2,3\}} \times \{0\} ^{ \{4\}} }
 T _{4} ( x_{\{1,2,3\}} ^{\omega }) 
\Bigr) \le   \vartheta  \mathbb P (A\mid H \times H \times H) ^{C}  \,. 
\end{gather}
\end{lemma}
%%%%%%%%%%%%%%%%%%%%%%%%%%%%%% LEMMA LEMMA LEMMA

%%%%%%%%%%%%%%%%%%%%%%%%%%%%%% PROOF PROOF PROOF
\begin{proof}
The first claim \eqref{e.QTj} follows from (an iteration of) the Second Proposition on Conservation 
of Densities, Proposition~\ref{p.con2}.  The second from Lemma~\ref{l.LFZvar}. 
\end{proof}
%%%%%%%%%%%%%%%%%%%%%%%%%%%%%% PROOF PROOF PROOF

The content of the next Lemma is that in the case where $ A\subset T$ has full probability, that 
$ A$ has the expected number of corners.  

%%%%%%%%%%%%%%%%%%%%%%%%%%%%%% LEMMA LEMMA LEMMA
\begin{lemma}\label{l.QTTTT}  Let $ \mathcal A$ be an admissible corner system.  Then, we have 
\begin{equation}
\label{e.QTTTT} 
\operatorname Q (T,T,T,T) \stackrel u =
\prod _{\ell =1} ^{4} \delta _{T\mid \ell } \times 
\operatorname Q (T_1,T_2,T_3,T_4)  \,. 
\end{equation}
Here, the constant $ \vartheta $ implicit in the $ \stackrel u =$ can be taken to be 
$ \vartheta = \kappa ' \epsilon $, where these two constants are determined by 
$ \kappa _{\textup{admiss}}$ and $ \epsilon _{\textup{admiss}}$ in Definition~\ref{d.admissible}, 
and can be made arbitrarily small.  
\end{lemma}
%%%%%%%%%%%%%%%%%%%%%%%%%%%%%% LEMMA LEMMA LEMMA

%%%%%%%%%%%%%%%%%%%%%%%%%%%%%% PROOF PROOF PROOF
\begin{proof}
One  considers the expression in \eqref{e.QTTTT} is a $ 4$-linear form, and expand 
$ T$ as $ T=f _{j,1}+ f _{j,0}$, where $ f _{j,1}= \delta _{T\mid j} T_j$.  This 
leads to an expansion of $ \operatorname Q (T,T,T,T)$ into $ 2 ^{4}$ terms, of which 
the leading term is 
\begin{align*}
\operatorname Q (f _{1,1},f _{2,1},f _{3,1},f _{4,1})
&= \prod _{j=1} ^{4} \delta _{T\mid j} 
\cdot \operatorname Q (T_1,T_2,T_3,T_4) \,. 
\end{align*}

The remaining $ 2 ^{4}-1$ terms all have at least one $ f _{j,0}$.  We can show that 
all of these terms is at most a small constant times the expression above
by appealing to \eqref{e.Ad3Box} and 
\eqref{e.gcsi2}.   In particular, we show that we can estimate 
\begin{align} \label{e.Qv}
\Abs{\operatorname Q (f _{1,\epsilon (1)},f _{2,\epsilon(2)},f _{3,\epsilon(3)},f _{4,0})}
& \le   2 
\operatorname Q (T_1, T_2, T_3, T_4) \cdot 
\Biggl[ \upsilon + \frac { \norm f _{4,0} . \Box ^{1,2,3} S _{1,2,3}.  ^{8}} 
{ \norm T_4 . \Box ^ {1,2,3} S _{1,2,3}.  ^{8}}  \Biggr] ^{1/8}\,. 
\end{align}
By \eqref{e.Ad3Box}, this proves that this term is very small.  This inequality singles out the fourth 
coordinate for a special role, but the proof, presented in full in this case, holds in full generality, so 
completes this case.  

Apply Lemma~\ref{l.Q}, with $ \overline T_j=T_j$ and $ f _{j}= f _{j,\epsilon (j)}$ as above. 
The estimate we get from this Lemma is \eqref{e.Q1}, with the terms in \eqref{e.Q2}---\eqref{e.QZ} 
estimated in Lemma~\ref{l.QTj}.  The particular point to observe is that the function $ Z$ has 
 a small conditional variance \eqref{e.ZTjvar}.  
These conditional estimates hold on the support of the product that occurs in \eqref{e.Q5}.  Hence, 
we can estimate 
\begin{equation}\label{e.Qvv}
\begin{split}
\Abs{\operatorname Q (f _{1,\epsilon (1)},f _{2,\epsilon(2)},f _{3,\epsilon(3)},f _{4,0})}
& \le 
\operatorname U_1 (T_1) ^{1/2}  \cdot 
\operatorname U_2 (T_2)  ^{1/4}\cdot 
\operatorname U_3 (T_3) ^{1/8}  \cdot 
\operatorname U_4 (T_1,T_2,T_3,f_{4,0}) ^{1/8}  
\\  
&=\operatorname U_1 (T_1) ^{1/2}  \cdot 
\operatorname U_2 (T_2)  ^{1/4}\cdot 
\operatorname U_3 (T_3) ^{1/8}  \cdot 
\\ \notag
& \qquad \times  
\mathbb E \Biggl(\operatorname  Z (T_1,T_2,T_3) \mid  
\prod _{\omega \in \{0\} ^{ \{1,2,3\}} \times \{0\} ^{ \{4\}} }
 T _{4} ( x_{\{1,2,3\}} ^{\omega }) ^{1/8}
\Biggr) 
\\ \notag
& \qquad \times  { \norm T_4 . \Box ^ {1,2,3} H _{1,2,3}.  } \cdot 
\Biggl[ \upsilon + \frac { \norm f _{4,0} . \Box ^{1,2,3} H _{1,2,3}.  } { \norm T_4 . \Box ^ {1,2,3} H _{1,2,3}.  }    \Biggr]
\end{split}
\end{equation}
In the last line, $ \upsilon $ is a small quantity arising from the conditional variance estimate 
\eqref{e.Zvar}.

The key identity is  \eqref{e.QTj}.  In it, observe that 
\begin{align*}
\operatorname U_4 (T_4, T_3,T_2, T_1) 
&\stackrel u =  \norm T_4 . \Box ^{1,2,3} H _{1,2,3} . ^{8} 
\cdot  \mathbb E \Biggl(\operatorname  Z (T_1,T_2,T_3) \mid  
\prod _{\omega \in \{0\} ^{ \{1,2,3\}} \times \{0\} ^{ \{4\}} }
 T _{4} ( x_{\{1,2,3\}} ^{\omega }) \Biggr)\,. 
\end{align*}
Therefore, we have 
\begin{align*}
\operatorname Q (T_1,T_2,T_3,T_4) &\stackrel u = 
\operatorname U_1 (T_1) ^{1/2}  \cdot 
\operatorname U_2 (T_2)  ^{1/4}\cdot 
\operatorname U_3 (T_3) ^{1/8}  
\\ & \quad \times 
\mathbb E \Biggl(\operatorname  Z (T_1,T_2,T_3) \mid  
\prod _{\omega \in \{0\} ^{ \{1,2,3\}} \times \{0\} ^{ \{4\}} }
 T _{4} ( x_{\{1,2,3\}} ^{\omega }) \Biggr)^{1/8}
\\& \quad 
\times  \bigl\{ \upsilon + \norm T_4 . \Box ^ {1,2,3} H _{1,2,3}.   \bigr\}
\end{align*}
And this completes the proof of \eqref{e.Qv} and hence the Lemma.
\end{proof}
%%%%%%%%%%%%%%%%%%%%%%%%%%%%%% PROOF PROOF PROOF

To apply Lemma~\ref{l.Q} to prove Lemma~\ref{l.3dvon}, we will need estimates for the terms in 
\eqref{e.Q2}---\eqref{e.Q5}.  We turn to this next, discussing the estimates for the 
terms $ \operatorname U_j$.  The estimates for $ \operatorname Z(T,T,T,T)$ as defined in \eqref{e.QZ} 
we discuss in the next Lemma.

%%%%%%%%%%%%%%%%%%%%%%%%%%%%%% LEMMA LEMMA LEMMA
\begin{lemma}\label{l.QT} We have the estimates below for the forms $ \operatorname U_j$ defined in
\eqref{e.Q2}---\eqref{e.Q5}.  
\begin{align}
\label{e.QT}  
\operatorname U_1 (T) &\stackrel u = \delta _{T\mid 1} 
\operatorname U _1 (T_1)\,, 
\\
\label{e.QTT} 
\operatorname U_2 (T)
&\stackrel u = \delta _{T\mid 2} ^2  
\operatorname U _2 (T_2)\,, 
\\ \label{e.QTTT} 
\operatorname U_3 (T)
&\stackrel u =  \delta _{T \mid 3} ^{4}
\operatorname U _3 (T_3)\,, 
\\ \label{e.Tbox}
\norm T . \Box \{1,2,3\}. ^{8} 
&\stackrel u =
\delta _{T\mid 4} ^{8}  \cdot 
\norm T_4 . \Box \{1,2,3\}. ^{8} 
\end{align} 
The implied constant $ \vartheta $ in the definition of $ \stackrel u =$ can be taken to be 
$ \mathbb P (T\mid H \times H \times H)$ to some large power. 
\end{lemma}
%%%%%%%%%%%%%%%%%%%%%%%%%%%%%% LEMMA LEMMA LEMMA

%%%%%%%%%%%%%%%%%%%%%%%%%%%%%% PROOF PROOF PROOF
\begin{proof}
The equality \eqref{e.Tbox} is a corollary to part 2 of Lemma~\ref{l.Uuniform}, and Definition~\ref{d.U}. 
The other parts of the Lemma are also corollaries to the same fact, but not as stated, but with the 
role of $ T_4$ in Definition~\ref{d.U} replaced by that of $ T_2$ for \eqref{e.QTT}, and $ T_3$ for \eqref{e.QTTT}. 
\end{proof}
%%%%%%%%%%%%%%%%%%%%%%%%%%%%%% PROOF PROOF PROOF

We turn to the analysis of the term $ \operatorname Z (T,T,T)$ as defined in \eqref{e.QZ}.  

%%%%%%%%%%%%%%%%%%%%%%%%%%%%%% LEMMA LEMMA LEMMA
\begin{lemma}\label{l.Z} 
   We have the estimates below where $ Z= \operatorname Z (T,T,T)$. 
\begin{gather}\label{e..Zmean}
\mathbb E _{x ^0 _{ \{1,2,3\}} , x ^1 _{ \{1,2,3\}} \in H _{ \{1,2,3\}}} 
(Z \mid  U )
\stackrel u =  \prod _{j=1} ^{3} \delta _{T\mid j} ^{4} 
\times \mathbb E _{x ^0 _{ \{1,2,3\}} , x ^1 _{ \{1,2,3\}} \in H _{ \{1,2,3\}}} 
(Z (T_1,T_2,T_3) \mid  U )
 \,,
\\
\label{e..Zvar} 
\operatorname {Var} _{x ^0 _{ \{1,2,3\}} , x ^1 _{ \{1,2,3\}} \in H _{ \{1,2,3\}}} 
(Z \mid U)
\le \delta _{A\mid T} ^{12}\,, 
\\ \label{e.Udef}
\textup{where } \quad 
U=\prod _{\omega \in\{0,1\}^{\{1,2,3\}} }\prod _{1\le j <k\le 3} R_{j,k} (x ^{\omega
} _{j,k})\,.
\end{gather}  
The implied constant in $ \stackrel u =$ can be taken as in Lemma~\ref{l.QTTTT}. 
\end{lemma}
%%%%%%%%%%%%%%%%%%%%%%%%%%%%%% LEMMA LEMMA LEMMA

Here, note that we are using the conditional expectation notation.  As the random 
variable $ Z$ is supported on the event $ U\subset H _{ \{1,2,3\}} ^{0} \times H _{ \{1,2,3\}} ^{1} $,  
we have 
\begin{gather}\label{e.condExp1}
\mathbb E _{x ^0 _{ \{1,2,3\}} , x ^1 _{ \{1,2,3\}} \in H _{ \{1,2,3\}}} 
(Z \mid U)
= 
\frac {\mathbb E _{x ^0 _{ \{1,2,3\}} , x ^1 _{ \{1,2,3\}} \in H _{ \{1,2,3\}}}  Z}
{ \mathbb E _{x ^0 _{ \{1,2,3\}} , x ^1 _{ \{1,2,3\}} \in H _{ \{1,2,3\}}} U}
\\ \label{e.condVar1}
\operatorname {Var} (Z\mid U) 
= 
\frac {
\mathbb E _{x ^0 _{ \{1,2,3\}} , x ^1 _{ \{1,2,3\}} \in H _{ \{1,2,3\}}}  Z ^2 -
\Bigl(  \mathbb E _{x ^0 _{ \{1,2,3\}} , x ^1 _{ \{1,2,3\}} \in H _{ \{1,2,3\}}}  Z  \Bigr) ^2 
\Bigl(\mathbb E _{x ^0 _{ \{1,2,3\}} , x ^1 _{ \{1,2,3\}} \in H _{ \{1,2,3\}}} U \Bigr) ^{-1} }
{ \mathbb E _{x ^0 _{ \{1,2,3\}} , x ^1 _{ \{1,2,3\}} \in H _{ \{1,2,3\}}} U}
\end{gather}
And the point of the Lemma is that the random variable $ Z$ is nearly constant on 
the set $ U$, and we can compute that constant.

%%%%%%%%%%%%%%%%%%%%%%%%%%%%%% PROOF PROOF PROOF
\begin{proof}
We first calculate the denominator in \eqref{e.condExp1} and \eqref{e.condVar1}.  
This is relatively simple as the sets $ R_{j,k}$ are uniform in $ S_j \times S_k$, 
so that we can estimate 
\begin{equation}\label{e.U=}
\mathbb E _{x ^0 _{ \{1,2,3\}} , x ^1 _{ \{1,2,3\}} \in H _{ \{1,2,3\}}} U
\stackrel u = \prod _{j=1} ^{3} \delta  _{j} ^2 \prod _{1\le j<k\le 3} \delta _{j,k} ^{4} \,. 
\end{equation}

We now turn to the numerator in \eqref{e.condExp1}.  
The expectation of $ Z$ in \eqref{e.condExp1} is thought of as a $ 12$-linear form.  
Set 
\begin{equation*}
\Omega _{\neq j}  =  \{0,1\}^{\{ 1\le k \neq j \le 3\}} \times \{0\} ^{4}\,, \qquad 1\le j\le 3\,. 
\end{equation*}
Set $ \Omega _{\neq }= \bigcup _{j=1} ^{3} \Omega _{\neq j}$.  
For functions $ \{f _{\omega }\mid \omega \in \Omega _{\neq } \}$ 
define 
\begin{equation*}
\operatorname L (f _{\omega } \mid \Omega _{\neq} )= \mathbb E _{ \substack{x _{1,2,3} \in H _{1,2,3}\\ x_4\in H }} 
\prod _{\omega \in \Omega_1} f _{\omega } \,. 
\end{equation*}
We are to prove the estimate 
\begin{equation}\label{e.L1}
\operatorname L ( T \mid \Omega _{\neq}) \stackrel u =  
\prod _{j=1} ^{3} \delta _{T\mid j} ^{4} \cdot  
\operatorname L ( T _j \mid \Omega _{\neq j}\,,\, 1\le j\le 3)\,.  
\end{equation}

Expand $ T \circ \lambda _j=f _{j,1}-f _{j,0}$, where $ f _{j,1}= \delta _{T\mid j} T_j$. 
The leading term is then when $ f _{j,1}$ occurs in all twelve positions. 
But, then we have the Second Conservation of Densities Proposition at our disposal, so that 
\eqref{e.L1} follows from Proposition~\ref{p.con2}. 

The ratio of \eqref{e.L1} and \eqref{e.U=} proves \eqref{e..Zmean}, provided the other 
terms arising from the expansion of the $ 12$-linear form are all sufficiently small.  
That is, we should see that 
for all $ 2 ^{12}-1$ selections of $ f _{j,\epsilon (\omega )} \in \{ f _{j,0}\,,\, f _{j,1} 
\}$ for $ \omega \in \Omega _{\neq j}$, $ 1\le j \le 3$, with at least one $ f _{j,\epsilon (\omega )}= f _{j,0}$ 
we have 
\begin{equation}\label{e.L10}
\Abs{ \operatorname L ( f _{j,\epsilon (\omega )} \mid \Omega _{\neq} )}
\le \kappa  \operatorname  L (T\mid \Omega _{\neq })\,, 
\end{equation}
for a suitably small constant $ \kappa $.  

If we use the same  line of reasoning that we have before, this 
would lead to a (yet) longer multi-linear form.  We therefore present the following variant 
of the argument used thus far. We prove \eqref{e.L10} under the following assumptions.  
For some $ \omega \in \Omega _{\neq 1}$, we have $ f _{1,\epsilon (\omega )}= f _{1,0}= T - \delta _{T\mid 1} T_1$. 
Moreover, this happens for $ \omega \equiv 0$, which we can assume after a change of variables.  
Finally,  let $ J _{\textup{small}} = \{ j=2,3 \mid \delta _{T\mid j} <  \delta _{T\mid 1}\}$.  We 
assume that $ f _{j, \epsilon (\omega )}= \delta _{T\mid j} T_j$ for all $ j\in J _{\textup{small}} $. 
This can also be assumed, after a permutation of the coordinates.  We now prove the inequality 
\begin{equation}\label{e.L100}
\Abs{ \operatorname L ( f _{j,\epsilon (\omega )} \mid \Omega _{\neq} )} 
\le   \prod _{j\in J _{\textup{small}}} \delta _{T\mid j} ^{4} \cdot 
\operatorname L ( T _j \mid \Omega _{\neq j}\,,\, 1\le j\le 3)
\cdot 
\Biggl[ \upsilon  + \frac { \norm f _{1,0} . \Box \{2,3,4\}. ^{8} } { \norm T_1 . \Box \{2,3,4\}. ^{8} } \Biggr]
^{1/8}\,. 
\end{equation}
Here, $ \upsilon  $ will be a very small positive constant.  Our assumption \eqref{e.Ad3Box}, together with 
the assumption about $ J _{\textup{small}}$ permits us to conclude \eqref{e.L10} from this inequality.  
In particular, we can accumulate a large number of powers of $ \delta _{T\mid 1}$ from \eqref{e.Ad3Box}.  
The essential point, is that we accumulate the correct power on the densities $ \delta _{T\mid j}$ for 
$ j\in J _{\textup{small}}$, as there is no  \emph{a priori} reason that the different densities $ \delta _{T\mid j}$ 
need be comparable.

But, \eqref{e.L100} follows from application of the inequality \eqref{e.LF3Box}, and so our proof of the Lemma 
is complete. 

\end{proof}
%%%%%%%%%%%%%%%%%%%%%%%%%%%%%% PROOF PROOF PROOF

%%%%%%%%%%%%%%%%%%%%%%%%%%%%%% PROOF PROOF PROOF
\begin{proof}[Proof of Lemma~\ref{l.3dvon}.]

 Write 
$ A=f_{0}+f _{1}$  where $ f _{1}= \delta _{A\mid T} T$. 
We expand 
\begin{equation} \label{e.;q15}
\operatorname Q (A,A,A,A)=
\sum _{\epsilon \in M_4}
\operatorname Q (f _{ \epsilon (1)}, f _{ \epsilon (2)}, f _{ \epsilon (3)}, f _{
\epsilon (4)})\,. 
\end{equation}
The leading term is for the function $ \epsilon\equiv 1 $.
It is $ \delta _{A\mid T} ^{4} \operatorname Q (T,T,T,T)$, with the 
latter expression estimated in \eqref{e.QTTTT}.  

All other choices of $ \epsilon $ have at least one choice choice of $ 1\le j\le 4$ for 
which we have $ \epsilon (j)=0$.  We claim that for all of these we have the estimate 
\begin{equation}\label{e.;q0} 
\lvert  
\operatorname Q (f _{ \epsilon (1)}, f _{ \epsilon (2)}, f _{ \epsilon (3)}, f _{
\epsilon (4)})\rvert  
\le \kappa \delta _{A\mid T} ^{4} \operatorname Q (T,T,T,T) \,. 
\end{equation}
This depends upon the assumption \eqref{e.UniformEnough}. 
For $ \kappa < 2 ^{-32} $, this will show that 
$ \operatorname Q (A,A,A,A) \ge \tfrac 14 \delta _{A\mid T} ^{4} \operatorname Q (T,T,T,T) $.
From this, we conclude that the number of corners in $ A$ is at least 
\begin{equation*}
 \operatorname Q (A,A,A,A) \lvert  H\rvert ^{4}- \lvert  A\rvert\ge   
\tfrac 14 \delta _{A\mid T} ^{4} \operatorname Q (T,T,T,T) \lvert  H\rvert ^{4} - 
\lvert  A\rvert >0  
\end{equation*}
Here, we subtract off $ \lvert  A\rvert $, as the average $  \operatorname Q (A,A,A,A)$ 
includes the `trivial corners' where all four points in the corner are the same.;  
The inequality holds by  \eqref{e.BigEnough}, and this completes the proof. 

\medskip

We prove \eqref{e.;q0} for $ \epsilon (4)=0$, with the other cases following by symmetry.  
Apply Lemma~\ref{l.Q}, with $ \overline T_j=T$, and $ f_4=f _{0}$.  This gives us 
the inequality 
\begin{equation*}
\lvert  
\operatorname Q (f _{ \epsilon (1)}, f _{ \epsilon (2)}, f _{ \epsilon (3)}, f _{0})\rvert
\le \operatorname U_1 (T) ^{1/2} 
\cdot\operatorname U _2 (T) ^{1/4} \cdot\operatorname U _3  (T)^{1/8} \cdot\operatorname U _4  (f_0,T,T,T)^{1/8}\,. 
\end{equation*}
The terms $ \operatorname U_j (T)$ for $ j=1,2,3$ are estimated in Lemma~\ref{l.QT}.  
The definition of $ \operatorname U_4 (f_0, T,T,T)$ in \eqref{e.Q5} depends upon $ \operatorname Z$, 
which has its properties listed in Lemma~\ref{l.Z}.  This leads us to the estimate 
\begin{align*}
\operatorname Q (f _{ \epsilon (1)}, f _{ \epsilon (2)}, f _{ \epsilon (3)}, f _{0})\rvert
& \le 
 \operatorname U_1 (T) ^{1/2} 
\cdot\operatorname U _2 (T) ^{1/4} \cdot\operatorname U _3  (T)^{1/8}  \cdot 
\mathbb E (Z\mid U)  ^{1/8}
\\ & \qquad \times \norm T. \Box \{1,2,3\}.  \cdot 
\Biggl[ \upsilon + \frac {\norm f_0. \Box \{1,2,3\}.} {\norm T. \Box \{1,2,3\}. } \Biggr]
\\& \le 
\prod _{\ell =1} ^{4} \delta _{T\mid \ell } \times 
 \operatorname U_1 (T_1) ^{1/2} 
\cdot\operatorname U _2 (T_2) ^{1/4} \cdot\operatorname U _3  (T_3)^{1/8}  
\\& \quad  \qquad \times 
\mathbb E (Z (T_1,T_2,T_3)\mid U)  ^{1/8}
\\ & \qquad \times \norm T_4. \Box \{1,2,3\}.  \cdot 
\Biggl[ \upsilon + \frac {\norm f_0. \Box \{1,2,3\}.} {\norm T. \Box \{1,2,3\}. } \Biggr]
\\
& \le \operatorname Q (T,T,T,T) 
\Biggl[ \upsilon + \frac {\norm f_0. \Box \{1,2,3\}.} {\norm T. \Box \{1,2,3\}. } \Biggr]\,. 
\end{align*}
Our proof is complete. 
\end{proof}
%%%%%%%%%%%%%%%%%% PROOF PROOF PROOF

%%%%%%%%%%%%%%%%%%%%%%%%%%%%%% SECTION  SECTION SECTION
%%%%%%%%%%%%%%%%%%%%%%%%%%%%%% SECTION  SECTION SECTION 
\section{The Paley-Zygmund Inequality for the Box Norm and the  set $ T$} \label{s.Tbox}

Let us recall the following classical result.

\begin{paleyZygmund} \label{p.pz} There is a $ 0<c<1$ so that for all random variables 
$ -1<Z<1$ with $ \mathbb E Z=0$ we have 
$
\mathbb P ( Z> c \mathbb E Z ^2 )\ge c \mathbb E Z ^2 \,. 
$
\end{paleyZygmund}

Our central purpose in this section is to provide extensions of this result 
to the case where the assumption on the standard deviation of the random variable is 
replaced by an assumption on the Box Norm.  Extensions are provided into 
two different settings, an `unweighted'  and a `weighted' one. 
Indeed, in the unweighted case, we will only require the two dimensional version of this inequality.

\begin{BoxPaleyZygmund} \label{l.BPZ} There is a constant $ c (2)$, 
and $ t (2)>1$ so that the following holds.   
For all finite sets $ X _{t}$, $ 1\le t\le 2$, and subsets  
$ A \subset   X_ {\{1 ,2\}}$,  set  $ \delta = \mathbb P (A)$ and 
$
\sigma = \norm A - \mathbb P (A) . \Box ^{ \{1 ,2 \}}   X_ {\{1 ,2 \}} . 
$.
There are subsets 
\begin{gather} \label{e.bpz1} 
X' _{ i} \subset X_i\,, \qquad  i= 1,2\,,  
\\ \label{e.bpz3}
\mathbb P (  {X'_i} \Bigr)
\ge  c (2)(\sigma \delta)  ^{t(2)}
\,, 
\\ \label{e.bpz4}
\mathbb P ( A \mid  X' _{1,2}) 
\ge  \delta + c (2) (\delta \sigma ) ^{t(2)}\,. 
\end{gather}
\end{BoxPaleyZygmund}

We refer the reader to \cite{MR2187732}*{Proposition 5.7} or \cite{MR2289954}*{Lemma 3.4} for a proof of this Lemma.

We need a more general version of the 
 Paley-Zygmund Inequality for the Box Norm,  
 is  based upon the properties of the sets $ A\subset T\subset T_j$.   We need two 
 Lemmas, with very similar proofs, accordingly we state one  Lemma.  Our Lemmas 
should be coordinate-free, but to ease the burden of notation, we state them distinguishing  
the coordinate $ x_4$ for a special role.

 %%%%%%%%%%%%%%%%%%%%%%%%%%%%%% LEMMA LEMMA LEMMA
\begin{lemma}\label{l.Tbox}  There are constants $ c>0$ and $ C, p>1$ so that the following 
holds. Suppose that $ \mathcal T$ is a $ T$-system as in 
\eqref{e.Tsystem}, which satisfies \eqref{e.Ad1Box} and \eqref{e.Ad2Box}.  
Let $ U\subset V\subset T_4$.  Assume that  $ V\in \{T_4,T\}$. 
\begin{equation}\label{e.Tbox0}
\frac{\norm U- \mathbb P (U\mid V) V . \Box ^{  \{1,2,3\}} S_{ \{1,2,3\}}. }
 { \norm V . \Box ^{  \{1,2,3\}} S_{ \{1,2,3\}}.  }
\ge \tau 
\end{equation}
and that $ V$ is $ (4,\vartheta , 4)$-uniform, (Recall Definition~\ref{d.U}.)  where 
\begin{equation}\label{e.zvoDef}
\vartheta = ( \tau \mathbb P (U\mid V )) ^{C}\,. 
\end{equation}
Then, there is a $ T$-system
\begin{equation}\label{e.T'system}
\mathcal T'=\{H\,,\, S'_k\,,\, R ' _{k, \ell } \,,\, T' \mid  1\le k, \ell \le 4\,,\ k< \ell  \}
\end{equation}
and a set $ V'\subset T'_4$, which satisfy  
\begin{gather} \label{e.V'}
\begin{cases}
V'=T'_4  &   V=T_4
\\
V'\subset V  &  V=T
\end{cases}
\\
\label{e.T'Uniform}
\begin{cases}
\mathbb P (T'_4 \mid T_4) \ge (\tau \mathbb P (U\mid T_4)) ^{p}  &  V=T_4 
\\
\mathbb P (T'\mid T)\ge (\tau \mathbb P (U\mid T)) ^{p}  & V=T 
\end{cases}
\\
\label{e.Tbox6}
\mathbb P (U \mid T'\cap V) \ge \mathbb P (U\mid  V)+ c(\tau \cdot \mathbb P (U\mid V)   ) ^{p}\,.
\end{gather}
\end{lemma}
%%%%%%%%%%%%%%%%%%%%%%%%%%%%%% LEMMA LEMMA LEMMA

The point of these estimates is that we have a little information about the new data, in \eqref{e.V'}. 
There are some lower bounds on the probabilities 
of the elements of the new $ T$-system given by  the estimate \eqref{e.T'Uniform}.
And in \eqref{e.Tbox6}, we have that 
$ U$ has a slightly larger probability in $ T'\cap V$.  Note that we certainly do not 
assume that the new $ T$-system $ \mathcal T'$ satisfies the uniformity assumptions 
in the definition of admissibility, Definition~\ref{d.admissible}.

%%%%%%%%%%%%%%%%%%%%%%%%%%%%%% PROOF PROOF PROOF
\begin{proof}[Proof of Lemma~\ref{l.dinc}.] %\label{s.}

To prove Lemma~\ref{l.dinc}, apply Lemma~\ref{l.Tbox} with $ V=T$, $ U=A$, and 
$ \tau = \kappa \delta _{A\mid T} ^{4}$, where $ \kappa $ is as in \eqref{e.UniformEnough}.  
The conclusions of Lemma~\ref{l.Tbox} then imply those of Lemma~\ref{l.dinc}. 
\end{proof}
%%%%%%%%%%%%%%%%%%%%%%%%%%%%%% PROOF PROOF PROOF

%%%%%%%%%%%%%%%%%%%%%%%%%%%%%% SUBSECTION SUBSECTION SUBSECTION SUBSECTION
 %%%%%%%%%%%%%%%%%%%%%%%%%%%%%% SUBSECTION SUBSECTION SUBSECTION SUBSECTION 
\subsection{One-Dimensional Obstructions}%\label{ss.}
We carry out the proof of Lemma~\ref{l.Tbox}. Throughout, we use the expansion 
$ U= f_1+f_0$ where $ f_1= \delta _{U\mid V} V$ where $ \delta _{U\mid V}= \mathbb P (U\mid
V)$.    We will also use the notation $ \delta _{V\mid 4}= \mathbb P (V\mid T_4)$. 
The key assumption \eqref{e.Tbox0}, which 
could hold due to lower-dimensional obstructions, and so there are two initial stages in 
which we address these obstructions.

We begin by considering
 the possibility that \eqref{e.Tbox0} holds for some one-dimensional reason.  Namely, 
 let us assume that, for instance,  we have 
 \begin{equation}\label{e.Tbox-oneD}
\begin{split}
\mathbb E _{x _{2,3} \in S_{2,3} }  \Abs{ \mathbb E _{x_1\in S_1} f _0 (x_1,x_2,x_3)} ^2 
& \ge [c_1 (\delta _{U\mid V} \tau ) ^{t_1}  ] ^2 
\mathbb E _{x _{2,3} \in S_{2,3} }  \Abs{ \mathbb E _{x_1\in S_1} V(x_1,x_2,x_3)} ^2 
\\
& \ge \tfrac 12 [ c_1 (\delta _{U\mid V} \tau ) ^{t_1} ] ^2 \cdot     
\delta _4 ^2 \cdot \delta _{V\mid 4} ^2 \cdot  \delta _{1,2} ^2  \cdot \delta _{1,3} ^2 \cdot  \delta _{2,3}\,. 
\end{split}
\end{equation}
Note that the last expectation is estimated by virtue of our assumption on 
$ (4,\vartheta,4)$-uniformity, recall \eqref{e.Uuniform}. 
Here, $ c_1>0$ and $ t_1>1$ are constants that we will specify below, 
based upon considerations in the next two stages of our argument.

Let us rephrase \eqref{e.Tbox-oneD} as 
\begin{equation} \label{e.Tbox-one'}
\mathbb E _{x _{2,3} \in R_{2,3} }  \Abs{ \mathbb E _{x_1\in S_1} f _0 (x_1,x_2,x_3)} ^2 
\ge \tfrac 12 c_1 (\delta _{U\mid V} \tau ) ^{t_1}   \cdot   
\delta _4 ^2 \cdot  \delta _{V\mid 4} ^2 \cdot  \delta _{1,2} ^2  \cdot  \delta _{1,3} ^2
\end{equation}
where we have replaced the expectation over  $ S _{2,3}= S_2 \times S_3$ by expectation over the 
smaller set $ R _{2,3}$. Of course, we have 
$ \abs{\mathbb E _{x_1\in S_1} f _0 (x_1,x_2,x_3)}\le 
\mathbb E _{x_1\in S_1} V (x_1,x_2,x_3) $.  But, the 
variance of this last random variable over $ R _{2,3}$ is nearly constant.  Namely, 
\begin{equation} \label{e.Tbox-oneVar}
\operatorname {Var} _{ x _{2,3}\in R _{2,3}} \Bigl( \mathbb E _{x_1\in S_1} V (x_1,x_2,x_3)\Bigr)
\le K \tau ^{C} 
\Bigl[  \mathbb E _{\substack{x_1\in S_1\\ x _{2,3}\in R _{2,3}}} V (x_1,x_2,x_3)\Bigr] ^2\,. 
\end{equation}
This is a corollary to Lemma~\ref{l.Zvar}. 

We are in a situation where we can apply the Paley-Zygmund inequality,  Proposition~\ref{p.pz}
Note that the random variable  $ \mathbb E  _{x_1\in S_1} f _0 (x_1,x_2,x_3)$ 
is dominated in absolute value by $ \mathbb E  _{x_1\in S_1} V (x_1,x_2,x_3)$, which 
 has average value (on $ R_{2,3}$) given by 
 \begin{equation} \label{e.4mean}
\mathbb E _{\substack{x_1\in S_1\\ x _{2,3} \in R _{2,3} }} V (x_1,x_2,x_3)
\stackrel u = \delta _{V\mid 4} \cdot \delta _{1,2} \cdot  \delta _{1,3} \cdot \delta _{4} \,. 
\end{equation}
 This follows from assumption and  \eqref{e.Uuniform}.
 Moreover, by \eqref{e.Tbox-oneVar}, 
 the random variable $ \mathbb E  _{x_1\in S_1} V (x_1,x_2,x_3)$ has very small variance on $ R _{2,3}$, 
 so that except for a negligible probability, it is dominated by, say, twice its expectation.
 The key point here, is that in applying the Paley-Zygmund inequality, we can use the normalized 
 variance given by the ratio \eqref{e.Tbox-one'} and \eqref{e.4mean}: 
 \begin{align*}
\frac
{\mathbb E _{x _{2,3} \in S_{2,3} }  \Abs{ \mathbb E _{x_1\in S_1} f _0 (x_1,x_2,x_3)} ^2} 
{ [\mathbb E _{\substack{x_1\in S_1\\ x _{2,3} \in R _{2,3} }} V (x_1,x_2,x_3) ] ^2  }
& \ge 
\frac 
{\tfrac 12 c_1 (\delta _{U\mid V} \tau ) ^{t_1}    
\delta _4 ^2 \delta _{V\mid 4} ^2  \delta _{1,2} ^2  \delta _{1,3} ^2}
{\delta_{ V\mid 4} ^2 \cdot \delta _4 ^2 \cdot \delta _{1,2} ^2 } 
\\
& = \tfrac 12 c_1 (\delta _{U\mid V} \tau ) ^{t_1}    \,. 
\end{align*}
 Thus, we can estimate  
 \begin{gather} \notag
R' _{2,3} = \Bigr\{ x _{2,3} \in R_{2,3} \mid \mathbb E _{x_1\in S_1} f _0 (x_1,x_2,x_3) 
\ge \tfrac {c_1} {20} 
 c_1 (\delta _{U\mid V} \tau ) ^{t_1}  
 \mathbb E _{\substack{x_1\in S_1\\ x _{2,3} \in R _{2,3} }} V (x_1,x_2,x_3)
\Bigr\}\,,
\\ \label{e.new23}
\mathbb P (R' _{2,3} \mid R_{2,3}) 
\ge 
\tfrac 1{10} c_1 (\delta _{U\mid V} \tau ) ^{t_1}    \,. 
\end{gather}

We conclude the Lemma by taking the set $R' _{2,3} $ in \eqref{e.T'system} as above,
$ T'= T\cap \overline R _{2,3}' $, 
and the other data is unchanged.  If $ V=T_4$, the new set $ V' = V \cdot \overline R _{2,3}'$, so 
that \eqref{e.V'} holds. 
That \eqref{e.T'Uniform} holds follows from \eqref{e.new23}, and several 
applications of \eqref{e.gcsi2}. 
And that \eqref{e.Tbox6} holds follows from construction of $ R' _{2,3}$.

%%%%%%%%%%%%%%%%%%%%%%%%%%%%%% SUBSECTION SUBSECTION SUBSECTION SUBSECTION
 %%%%%%%%%%%%%%%%%%%%%%%%%%%%%% SUBSECTION SUBSECTION SUBSECTION SUBSECTION 
\subsection{Two-Dimensional Obstructions}%\label{ss.}
We continue the proof assuming that \eqref{e.Tbox-oneD} fails as written, and also fails 
 under any permutation of the variables $ x_1,x_2$, and $ x_3$.   The 
potential  lower dimensional obstruction are now two-dimensional in nature.  We could have for instance 
\begin{equation}\label{e.Tbox-twoD} 
\mathbb E _{x_1 \in S_1 }  
\norm f_0 . \Box ^{2,3} S _{2,3}. ^{4}
\ge 
c_2 (\delta _{U\mid V} \tau ) ^{t_2} 
\mathbb E _{x_1 \in S_1 }  \norm V . \Box ^{2,3} S _{2,3}. ^{4}\,. 
\end{equation}
Here, $ t_2, c_2>0$ are constants that are to be specified, based upon considerations in 
the next stage of the argument.  
The last expectation can be computed exactly, and is  
\begin{equation}\label{e.T44}
\begin{split}
\mathbb E _{x_1 \in S_1 }  \norm V . \Box ^{2,3} S _{2,3}. ^{4}
&= \mathbb E _{\substack{x_1\in S_1\\ x _{2,3} ^0, x _{2,3}^1\in S _{2,3} }} 
\prod _{ \omega \in \{0,1\}^{\{2,3\}}}
V (x_1,x _{2,3} ^{\omega }) 
\\
& \stackrel u = [ \delta _{4} \delta _{V\mid 4} ] ^{4}  \prod _{1\le j<k\le 3}\delta _{j,k} ^{2}\,. 
\end{split}
\end{equation}

% While we don't specifically appeal to this fact, it is useful to note that
% the random variable (on $ S_1$) 
% $ \norm V . \Box ^{2,3} S _{2,3}. ^{4}$ is essentially constant.  Namely, 
% \begin{equation}\label{e.T45}
% \operatorname {Var} _{x_1\in S_1} \bigl( \norm V. \Box ^{2,3} S _{2,3}. ^{4} \bigr)
% \le 2 \vartheta ' 
% \Bigl[\mathbb E _{x_1 \in S_1 }  \norm V . \Box ^{2,3} S _{2,3}. ^{4} \Bigr] ^2 \,. 
% \end{equation}

Of course we have $ \norm f_0 . \Box ^{2,3} S _{2,3}. ^{4} \le \norm V . \Box ^{2,3} S
_{2,3}. ^{4}$.   Still, the deduction of the Lemma in this case doesn't follow from a 
a straight forward application of Lemma~\ref{l.BPZ} in two dimensions, as we are in the 
weighted case. This 
argument is the one that relates the constants  $ c_1, t_1$ and constants $ c_2, t_2$. 

Following notation used in the proof of Lemma~\ref{l.BPZ}, we define a four linear 
term which arises from \eqref{e.Tbox-twoD}. 
\begin{equation} \label{e.UboxB4} 
\operatorname B _4 (f _{0,0}, f _{0,1}, f _{1,0}, f _{1,1})
= \mathbb E _{\substack{x _{1}\in S_1, 
\\ x_{2,3} ^0, x _{2,3}^1\in S _{2,3}} }
\prod _{ \epsilon \in \{0,1\} ^2 } 
f _{\epsilon } (x_1, x_{2,3} ^{\epsilon }) \,. 
\end{equation}
Note that  the left-hand-side of \eqref{e.Tbox-twoD} is 
$
\operatorname B _4 (f _{0}, f _{0}, f _{0}, f _{0}) 
$, and that 
$
{\mathbb E _{x_1 \in S_1 }  \norm V . \Box ^{2,3} S _{2,3}. ^{4}}= 
\operatorname B_4 (V, V, V, V) 
$, which is given in \eqref{e.T44}.

Our central claims are these inequalities, which hold for $ c_1, t_1$ sufficiently 
large, in terms of $ c_2,t_2$. 
\begin{gather}\label{e.B4U}
\frac {\operatorname B _4 (U, U, U, U)} 
{\operatorname B_4 (V, V, V, V)}
\ge  \delta _{U\mid V} ^{4} + \tfrac 1 4 c_2 (\delta _{U\mid V} \tau ) ^{t_2} \,, 
\\
\label{e.B4U4} 
\ABS{\delta _{U\mid V} ^{3} -  
\frac {\operatorname B _4 (U,U,U,V)} 
{\operatorname B_4 (V, V, V, V)} }
\le  8 c_1 (\delta _{U\mid V} \tau ) ^{t_1} \,,
\\
Z_V \coloneqq   \mathbb E _{\substack{x_1\in S_1\\ x _{2,3}^1 \in S _{2,3} }}  V(x_1, x_2^0,x_3^0) 
 V(x_1, x_2^0,x_3^1) V(x_1, x_2^1,x_3^0) V(x_1, x_2^1,x_3^1) \,,
\\ \label{e.B4Vratio}
\mathbb E _{ x _{2,3} ^{0} \in S _{2,3} } (Z_V)=
 { \operatorname B_4 (V,V,V,V)} \,, 
\,,
 \\
\label{e.B4Vvar} 
\operatorname {Var} _{ x _{2,3} ^{0} \in S _{2,3}} ( Z_V )
\le \sqrt \vartheta  \cdot  { \operatorname B_4 (V,V,V,V)} ^2 
\\ \notag 
Z_U \coloneqq  \mathbb E _{\substack{x_1\in S_1\\ x _{2,3}^1 \in S _{2,3} }}  U (x_1, x_2^0,x_3^0) 
 U (x_1, x_2^0,x_3^1) U (x_1, x_2^1,x_3^0) V(x_1, x_2^1,x_3^1) \,,
\\ \label{e.B4ratio}
\mathbb E _{ x _{2,3} ^{0} \in S _{2,3} } (Z_U)=
 { \operatorname B_4 (U,U,U,V)} \,. 
\,,
 \\  
\label{e.B4var} 
\operatorname {Var} _{ x _{2,3} ^{0} \in S _{2,3}} ( Z_U )
\le 32 c_1 (\delta _{U\mid V} \tau ) ^{t_1} { \operatorname B_4 (V,V,V,V)} ^2  \,. 
\end{gather}
Notice that the constant $ t_1$ of \eqref{e.Tbox-oneD} appears in the 
estimates \eqref{e.B4U4} and \eqref{e.B4var}.  We take  $ t_1>2t_2+3$.  
In \eqref{e.B4var}, note that we have three occurrences of $ U$ and one of $ V$.  
The expectation of $ Z$ is the term in \eqref{e.B4U4}.

%%%%%%%%%%%%%%%%%%%%%%%%%%%%%% PROOF PROOF PROOF
\begin{proof}[Proof of \eqref{e.B4U}.] 
  The denominator on the left-hand-side is 
estimated in \eqref{e.T44}.  So we estimate the numerator.  We use the expansion
 $ U= f_1+f_0$ four times to write  $ \operatorname B _4 (U,U,U,U)$ as 
 a sum of sixteen terms. 
 \begin{equation*}
\operatorname B _4 (U,U,U,U)
=\sum _{\epsilon \in M_4} \operatorname B _4 (f _{\epsilon (0,0)},f _{\epsilon (0,1)},
f _{\epsilon (1,0)},f _{\epsilon (1,1)}) 
\end{equation*}
where $ M_4$ denotes the collection of sixteen maps from $ \{0,1\} ^2 $ into $ \{0,1\}$. 
The two significant terms are associated to the maps $ \epsilon \equiv 0$ and $ \epsilon
\equiv1$. 
\begin{align*}
\operatorname B _4 (f_1, f_1, f_1, f_1) 
&= \delta _{U\mid V} ^{4} \operatorname B_4 (V,V,V,V)
\\
\operatorname B _4 (f_0, f_0, f_0, f_0) 
&\ge   c_2 (\delta _{U\mid V} \tau ) ^{t_2} \operatorname B_4  (V,V,V,V)
\end{align*}
The first is by definition of $ f_1= \delta _{U\mid V} V$, while the second is 
by assumption \eqref{e.Tbox-twoD}.   We should argue that the sum of the remaining 
fourteen choices of $ \epsilon $ are small.   But this follows from 
the fact that \eqref{e.Tbox-one'} fails, and the inequality \eqref{e.1Box}.  For any 
choice of   $ \epsilon \not\equiv 0,1$, the central hypothesis leading to that  inequality holds.  
Of course, it is important to use the fact that the one-dimensional obstructions are not in place 
at this point. 

\end{proof}
%%%%%%%%%%%%%%%%%%%%%%%%%%%%%% PROOF PROOF PROOF

%%%%%%%%%%%%%%%%%%%%%%%%%%%%%% PROOF PROOF PROOF
\begin{proof}[Proof of \eqref{e.B4U4}.]  
In $ \operatorname B _4 (U,U,U,V)$, expand each $ U$ as $ f_1+f_0$.  The leading term is 
when each $ U$ is replaced by $ f_0$, giving us
\begin{equation*}
 \operatorname B _4 (f_1,f_1,f_1,V)= \delta _{U\mid V} ^{3} 
 \operatorname B _4 (V,V,V,V) \,. 
\end{equation*}
The remaining seven terms are of the form 
$
\operatorname B _4 (f _{\epsilon (0,0)},f _{\epsilon (0,1)},
f _{\epsilon (1,0)},V ) 
$, where $ \epsilon \not\equiv 1$.  But then, the estimate \eqref{e.1Box} applies, 
so this proof is finished. 

\end{proof}
%%%%%%%%%%%%%%%%%%%%%%%%%%%%%% PROOF PROOF PROOF

%%%%%%%%%%%%%%%%%%%%%%%%%%%%%% PROOF PROOF PROOF
\begin{proof}[Proof of \eqref{e.B4Vratio} and \eqref{e.B4Vvar}.] 
The equation \eqref{e.B4Vratio} is by definition, and \eqref{e.B4Vvar} is a consequence of 
assumption on $ V$ and Lemma~\ref{l.Zvar}. 
\end{proof}
%%%%%%%%%%%%%%%%%%%%%%%%%%%%%% PROOF PROOF PROOF

%%%%%%%%%%%%%%%%%%%%%%%%%%%%%% PROOF PROOF PROOF
\begin{proof}[Proof of  \eqref{e.B4ratio} and \eqref{e.B4var}.]    
The equation \eqref{e.B4ratio} is by definition of $ Z_U$. 
The inequality \eqref{e.B4var} is very similar in spirit to Lemma~\ref{l.Zvar}, but does 
not explicitly follow from that Lemma. 

To compute the variance of $ Z_U$, we need the following $ 8$-linear form. 
\begin{align*}
\operatorname L_8 (g_1,g_2,g_3,&g_4,g_5,g_6,g_7,g_8) 
\\&=
\mathbb E _{\substack{ x _{1,2,3} ^0,x _{1,2,3} ^1,x _{1,2,3} ^2\in S _{1,2,3} \\ }}
g_1 (x_1^0,x_2^0,x_3^0)
g_2 (x_1^0,x_2^0,x_3^1)
g_3 (x_1^0,x_2^1,x_3^0)
g_4 (x_1^0,x_2^1,x_3^1)
\\& \qquad \times g_5 (x_1^1,x_2^0,x_3^0)
g_6 (x_1^0,x_2^0,x_3^2)
g_7 (x_1^1,x_2^2,x_3^0)
g_8 (x_1^1,x_2^2,x_3^2)
\end{align*}
The point of this definition is that $ \mathbb E _{x_{2,3}\in S_{2,3}} Z_U ^2
= \operatorname L_8 (U,U,U,V,U,U,U,V)$, and 
we want to establish the estimate 
\begin{equation}\label{e.L8<}
\mathbb E _{x_{2,3}\in S_{2,3}} Z_U ^2- \bigl(\mathbb E _{x _{2,3}\in S _{2,3}} Z_U \bigr) ^2 
\le 20 c_1 (\delta _{U\mid V} \tau ) ^{t_1} \bigl(\mathbb E _{x _{2,3}\in S _{2,3}} Z_U \bigr) ^2 \,. 
\end{equation}

We already have \eqref{e.B4U4}, which gives us an estimate of $ \mathbb E _{x _{2,3}\in S _{2,3}} Z_U $. 
It follows from $ V$ being $ (4, \vartheta ,4)$-uniform that we have 
\begin{equation*}
\delta _{U\mid V} ^{6}\operatorname L_8 (V,V,V,V,V,V,V,V) 
\stackrel u = 
[\delta _{U\mid V} ^{3} \cdot \operatorname B_4 (V,V,V,V)] ^2 
\end{equation*}
And so, we should verify that 
\begin{equation}\label{e.L8<<}
\begin{split}
\Abs{ \operatorname L_8 (U,U,U,V,U,U,U,V) - &\delta _{U\mid V} ^{6}\operatorname L_8 (V,V,V,V,V,V,V,V) } 
\\
&\le 20  c_1 (\delta _{U\mid V} \tau ) ^{t_1}\operatorname L_8 (V,V,V,V,V,V,V,V) \,. 
\end{split}
\end{equation}

The key assumption is that \eqref{e.Tbox-oneD} fails, which in turn suggests that 
we appeal to the inequality  \eqref{e.1Box}.  But, in the definition of $ \operatorname L_8$, 
no single variable occurs in just one function, the key hypothesis needed to apply \eqref{e.1Box}. 
This  fact  brings us to the observation that,  
for instance,  in the definition of $ \operatorname L_8$, only $ g_7$ and $ g_8$ are
  functions of $ x_2^2$.  Moreover, we are interested in the case where $ g_8=V$, 
a `highly uniform' function, and $ g_7=U= f_1+f_0 $.  Thus, our strategy is to selectively 
replace occurrences of $ U$ in $  \operatorname L_8 (U,U,U,V,U,U,U,V)$ in such a way that 
at each stage, there is single occurrence of $ f_0$, and that there is a variable in $ f_0$ which 
is only occurs in instances of $ V$.  

Specifically, we write 
\begin{gather*}
 \operatorname L_8 (U,U,U,V,U,U,U,V) - \delta _{U\mid V} ^{6}\operatorname L_8 (V,V,V,V,V,V,V,V)
 = \sum _{m=1} ^{6} D_m \,, 
 \\
 D_1 =  \operatorname L_8 (U,U,U,V,U,U,f_0,V)\,, \qquad  
 D_2 = \delta _{U\mid V}  \operatorname L_8 (U,U,U,V,U,f_0,V,V) \,, 
\\
 D_3 = \delta _{U\mid V} ^2   \operatorname L_8 (U,U,U,V,f_0,V,V,V) \,, \qquad 
  D_4 = \delta _{U\mid V}  ^{3} \operatorname L_8 (U,U,f_0,V,V,V,V,V) \,, 
  \\
  D_5 = \delta _{U\mid V} ^4   \operatorname L_8 (U,f_0,V,V,V,V,V,V) \,, \qquad 
  D_6 = \delta _{U\mid V}  ^{5} \operatorname L_8 (f_0,V,V,V,V,V,V,V) \,. 
\end{gather*}
Then, \eqref{e.L8<<} will follow from the estimate 
\begin{equation}\label{e.D<}
\lvert  D_m\rvert \le 3    c_1 (\delta _{U\mid V} \tau ) ^{t_1}  \operatorname L_8 (V,V,V,V,V,V,V,V)\,, 
\qquad 1\le m\le 6 \,. 
\end{equation}

Each of the six inequalities in \eqref{e.D<} follow from the same principle, and so we will 
only explicitly discuss the estimate for $ D_1$.  Write 
\begin{align*}
D_1= 
\mathbb E _{\substack{ x _{1,2,3} ^0,x _{1,2,3} ^1\in S _{1,2,3} \\  x _{1,3} ^2 \in S _{1,3}}}&
U (x_1^0,x_2^0,x_3^0)
U (x_1^0,x_2^0,x_3^1)
U (x_1^0,x_2^1,x_3^0)
V (x_1^0,x_2^1,x_3^1)
\\& \quad 
\times U (x_1^1,x_2^0,x_3^0)
U (x_1^0,x_2^0,x_3^2)
\cdot \mathbb E _{x_2 ^{2} \in S_2} 
f_0 (x_1^1,x_2^2,x_3^0)
V (x_1^1,x_2^2,x_3^2) \,. 
\end{align*}
Apply the Cauchy-Schwartz inequality in all variables except $ x_2^2\in S_2$.  In so doing, apply the 
First Proposition on Conservation of Densities, Proposition~\ref{p.con}, and the assumption of $ V$ 
being $ (4,\vartheta ,4)$-uniform to conclude that 
\begin{gather} \label{e.D11}
\lvert  D_1\rvert
\le 
\operatorname L_8 (V,V,V,V,V,V,V,V) 
\Biggl\{ \sqrt \vartheta + \frac  { \operatorname L_4 (f_0,f_0,V,V)} { \operatorname L_4 (V,V,V,V)} \Biggr\} 
^{1/2} 
\\ \notag 
\operatorname L_4 (g_1,g_2,g_3,g_4) 
= 
\mathbb E _{\substack{x_1^1\in S_1\\ x_2^2,x_2^3 \in S_2  \\ x_3^0,x_3^2\in S_3}}
g_1 (x_1^1,x_2^2,x_3^0)
g_2 (x_1^1,x_2^3,x_3^0)
g_3 (x_1^1,x_2^2,x_3^2)
g_1 (x_1^1,x_2^3,x_3^2)\,. 
\end{gather}

In the right-hand-side of \eqref{e.D11}, observe that we can write 
\begin{gather*}
{ \operatorname L_4 (f_0,f_0,V,V)}
= 
\mathbb E _{\substack{x_1^1\in S_1\\ x_2^2,x_2^3 \in S_2  \\ x_3^0\in S_3}}
f_0 (x_1^1,x_2^2,x_3^0)
f_0 (x_1^1,x_2^3,x_3^0) \cdot Y 
\\
Y= Y (x_1^1,x_2^2,x_2^3) 
= \mathbb E _{x_3^2\in S_3}
V (x_1^1,x_2^2,x_3^2)
V (x_1^1,x_2^3,x_3^2)\ ,. 
\end{gather*}
It follows from Lemma~\ref{l.Zvar} and assumption on $ V$, that $ Y$ is a random variable with 
non-zero mean and very small variance on the event $ V (x_1^1,x_2^2,x_3^0)
V (x_1^1,x_2^3,x_3^0)$.  Hence, 
\begin{align*}
\frac  { \operatorname L_4 (f_0,f_0,V,V)} { \operatorname L_4 (V,V,V,V)}
 \le 
 \sqrt \vartheta + 
 \frac  { \operatorname L_4 (f_0,f_0,1,1)} { \operatorname L_4 (V,V,1,1)}
\end{align*}
But the last ratio is controlled by the failure of \eqref{e.Tbox-oneD}, so our proof of 
\eqref{e.D<}, and hence \eqref{e.B4var} is complete. 
\end{proof}
%%%%%%%%%%%%%%%%%%%%%%%%%%%%%% PROOF PROOF PROOF

We need to conclude the proof of the Lemma, assuming the inequalities \eqref{e.B4U}---\eqref{e.B4var}. 
Select a point $ x _{2,3} ^{0}\in S _{2,3}$ at random, and define the data in \eqref{e.T'system} 
as follows. 
\begin{align*}
S_1'(x _{2,3} ^{0}) &= \{ x_1 \mid (x_1, x _{2} ^{0}, x_3 ^0)\in U\}\,,
\\
S _{1,2} ' (x _{2,3} ^{0})&=  \{ (x_1,x_2^1) \mid  (x_1, x_2^0,x_3^0),(x_1, x_2^1,x_3^0)\in U\}\,, 
\\
S _{1,3} '(x _{2,3} ^{0}) &=  \{ (x_1,x_3^1) \mid  (x_1, x_2^0,x_3^0),(x_1, x_2^0,x_3^1)\in U\}\,, 
\\
T' (x _{2,3} ^{0})&= \{ (x_1,x_2^1,x_3^1) \mid (x_1, x_2^0,x_3^0),(x_1, x_2^0,x_3^1)\in U\,,\, (x_1, x_2^1,x_3^1) \in V\}\,. 
\end{align*}
With this definition, it is clear that \eqref{e.V'} holds, namely if $ V=T_4$, we have $ V'=T_4'=T' (x_{2,3}^0)$. 
No change is made to the data not listed here, namely $ S_2, S_3$ and $ S _{2,3}$.  The point of these 
definitions is that we have 
\begin{equation*}
\mathbb E _{ \substack{x_1\in S_1\\ x _{2,3} ^{0}, x _{2,3} ^{1}\in S _{2,3} }} T' (x _{2,3} ^{0})
= \operatorname B_4 (U,U,U,V)\,,
\end{equation*}
and  $ \mathbb P _{\substack{x_1\in S_1\\ x _{2,3} ^{1}\in S _{2,3} }} 
(T' (x _{2,3} ^{0})) = Z_U(x _{2,3} ^{0})= Z_U$, in the notation of \eqref{e.B4ratio} and \eqref{e.B4var}.  

Define the event 
\begin{align*}
\widetilde S _{2,3}= 
\bigl\{ x ^{0} _{2,3} \in S _{2,3} & \mid 
\abs{ Z_U- \operatorname B_4 (U,U,U,V) } < [ c_2 (\delta _{U\mid V} t)] ^{t_1/2}  \operatorname B_4 (V,V,V,V)
\\
&\qquad \abs{ Z_V- \operatorname B_4 (V,V,V,V) } < [ c_2 (\delta _{U\mid V} t)] ^{t_1/2} \operatorname B_4 (V,V,V,V) 
\bigr\}\,. 
\end{align*}
It follows from \eqref{e.B4Vratio}---\eqref{e.B4var} that we have 
\begin{equation*}
\mathbb P (S _{2,3} - \widetilde S _{2,3}) < 32 
[ c_2 (\delta _{U\mid V} t)] ^{t_1/2} \,. 
\end{equation*}

Moreover, for $ t_1>4 t_2$, notice that we would have inequalities that look quite similar to 
\eqref{e.B4U} and \eqref{e.B4U4}.  In particular, we will have 
\begin{equation*}
\Abs{  \mathbb E _{x _{2,3} ^{0} \in S _{2,3} } Z_U- \operatorname B_4 (U,U,U,V) } 
\le [ c_2 (\delta _{U\mid V} t)] ^{t_1/2}  \operatorname B_4 (V,V,V,V) \,, 
\end{equation*}
with a similar inequality for $ Z_V$.  Hence, 
we can conclude the proof of the Lemma, by noting that 
\begin{align*}
\sup _{x _{2,3} ^{0} \in \widetilde S _{2,3}} \frac { Z_U} {Z_V} 
& \ge 
\frac 
{ \mathbb E _{x _{2,3} ^{0} \in \widetilde S _{2,3}}  { Z_U}   }
{ \mathbb E _{x _{2,3} ^{0} \in \widetilde S _{2,3}}  { Z_V}   }
 \ge \delta _{U\mid V} + \tfrac 14 (\delta _{U\mid V} \tau ) ^{t_2} \,.  
\end{align*}

%%%%%%%%%%%%%%%%%%%%%%%%%%%%%% SUBSECTION SUBSECTION SUBSECTION SUBSECTION
 %%%%%%%%%%%%%%%%%%%%%%%%%%%%%% SUBSECTION SUBSECTION SUBSECTION SUBSECTION 
\subsection{Three-Dimensional Obstructions}%\label{ss.}
We proceed under the assumption that that both \eqref{e.Tbox-oneD} and \eqref{e.Tbox-twoD} 
fail, as written and under all permutations of coordinates.  We have specified $ c_1,t_1$ 
as functions of $ c_2,t_2$, and this argument will specify these last two constants.  

We need the $ 8$-linear form, the analog of \eqref{e.UboxB4} given by 
\begin{equation}\label{e.U8}
\operatorname B _{8} ( f _{\epsilon }\mid \epsilon \in \{0,1\} ^{ \{1,2,3\}}) 
= \mathbb E _{x _{1,2,3} \in S _{1,2,3}} \prod _{\epsilon \in \{0,1\} ^{\{1,2,3\}}} 
f _{\epsilon } (x _{1,2,3} ^{\epsilon })\,. 
\end{equation}
The relevant facts we need about this form concern these values.  Set 
\begin{gather*}
\operatorname B_8 [W]= \operatorname B_8 (W\mid \epsilon \in \{0,1\} ^{ \{1,2,3\}})\,, \qquad W= U,V
\\
\operatorname B_8[U,V]= \operatorname B_8(U ,\dotsc, U,V\mid \epsilon \in \{0,1\} ^{ \{1,2,3\}})\,, 
\end{gather*}
where the lone $ V$ occurs in the $ \{1\} ^{1,2,3}$ position.
Indeed, note that $ \operatorname B_8[U]= \norm U. \Box ^{1,2,3} S _{1,2,3}. ^{8}$.

The facts we need are these. 
\begin{gather}\label{e.B8U}
\frac {\operatorname B_8[U]} {\operatorname B_8[V]}\ge \delta _{U\mid V} ^{8}+ \tfrac 12 \tau ^{8}\,, 
\\ 
\label{e.B8UV} 
\ABs{ \delta _{U\mid V} ^{7}- \frac {\operatorname B_8 [U,V]} {\operatorname B_8[V]}} 
\le \tfrac 1 {20} (\delta_{U\mid V} \tau )^{30}\,, 
\\
\notag
Z= \mathbb E _{x^1 _{1,2,3}\in S _{1,2,3}}  V (x ^1 _{1,2,3} )\prod _{ \substack{\epsilon \in \{0,1\} ^{ \{1,2,3\}}\\ \epsilon 
\not\equiv 0,1} } U (x_ {1,2,3} ^{\epsilon }) \,, 
\\
\label{e.B8mean} 
\mathbb E (Z\mid U) = \frac { \operatorname B_8 [ U,V]} {\mathbb P (U)} \,,
\\
\label{e.B8var} 
\operatorname {Var} _{x _{1,2,3} ^{0} \in S _{1,2,3}}  (Z\mid U)
\le \tfrac 1 {20} (\delta _{U\mid V} \tau ) ^{30} \operatorname B_8[V] ^2  \,.
\end{gather}

%%%%%%%%%%%%%%%%%%%%%%%%%%%%%% PROOF PROOF PROOF
\begin{proof}[Proof of \eqref{e.B8U}.]  
Consider $ \operatorname B_8[U]$.  Expand each occurrence of $ U$ as $ f_1+f_0$, where $ f_1= \delta _{U\mid V} V$.  
This leads to 
\begin{equation} \label{e.B8z}
\operatorname B_8[U]= \sum _{ \rho \in M_8}
\operatorname B_8 (f _{\rho (\epsilon) } \mid \epsilon \in \{0,1\} ^{ \{1,2,3\}})
\end{equation}
where $ M_8$ is the class of maps from $ \{0,1\} ^{ \{1,2,3\}} $ into $ \{0,1\}$.  The leading term is 
$ \rho \equiv 1$, which is 
\begin{equation} \label{e.rho==1}
 \delta _{U\mid V} ^{8}\operatorname B_8 [V]=  \delta _{U\mid V} ^{8} \norm V. \Box ^{1,2,3} S _{1,2,3}. ^{8}\,. 
\end{equation}
The other significant term is $ \rho \equiv 0$, which is 
\begin{equation*}
\operatorname B_8 (f_0 \mid \epsilon \in \{0,1\} ^{ \{1,2,3\}})= \norm f_0. \Box ^{1,2,3} S _{1,2,3}. ^{8} 
\ge \tau ^{8} \norm V. \Box ^{1,2,3} S _{1,2,3}. ^{8}\,.  
\end{equation*}
The last inequality follows from \eqref{e.Tbox0}. 

That leaves $2 ^{8}-2  $ additional terms in $ M_8$ to consider.  For each $ \rho \in M_8$ which is not 
equivalent to $ 0$ or $ 1$, the assumption for the inequality \eqref{e.2Box} holds.  Namely, there is a 
choice of $ \epsilon \in \{0,1\} ^{ \{1,2,3\}}$, and choice of distinct $ j,k\in \{1,2,3\}$ so that 
$ \rho (\epsilon )=0$, and for every other $ \epsilon ' $, we have either $ \epsilon (j)\neq \epsilon ' (j)$ 
or $ \epsilon (k)\neq \epsilon ' (k)$.  Therefore, the inequality \eqref{e.2Box} holds.  Combining this inequality 
with our assumption that \eqref{e.Tbox-twoD} fails, we see that this holds.  
\begin{align} \label{e.28-2}
\Abs{\operatorname B_8 (f _{\rho (\epsilon )} \mid \epsilon \in \{0,1\} ^{ \{1,2,3\}})} 
& \le  c_2 (\delta _{U\mid V} \tau ) ^{t_2} 
\times \norm V. \Box ^{1,2,3} S _{1,2,3} . ^{8} \,.   
\end{align}
For $ c_2$ sufficiently small, and $ t_2\ge 8$, this completes the proof of \eqref{e.B8U}. 

\end{proof}
%%%%%%%%%%%%%%%%%%%%%%%%%%%%%% PROOF PROOF PROOF

%%%%%%%%%%%%%%%%%%%%%%%%%%%%%% PROOF PROOF PROOF
\begin{proof}[Proof of \eqref{e.B8UV}.] 
Keeping the notation of \eqref{e.B8z}, we have 
\begin{equation*}
\operatorname B_8[U, V]= \delta _{U\mid V} ^{-1} \sum _{ \rho \in M_8'}
\operatorname B_8 (f _{\rho (\epsilon) } \mid \epsilon \in \{0,1\} ^{ \{1,2,3\}})
\end{equation*}
where $ M_8'$ is the class of maps $ \rho \in M_8$ such that $ \rho ( 1 ^{ \{1,2,3\}})=1$.  
The leading term is again $ \rho \equiv 1$, which is \eqref{e.rho==1} above.  
The remaining $ 2 ^{8}-1$ terms all admit the bound \eqref{e.28-2}.  Therefore, 
\begin{equation*}
\Abs{\operatorname B_8[U, V]- \delta _{U\mid V} - \delta _{U\mid V} ^{7} 
\norm V. \Box ^{1,2,3} S _{1,2,3} . ^{8} } 
\le 2 ^{8} (\delta _{U\mid V} \tau ) ^{t_2-1} 
\times \norm V. \Box ^{1,2,3} S _{1,2,3} . ^{8} \,. 
\end{equation*}
This proves \eqref{e.B8UV} for $ c_2$ sufficiently small, and $ t_2\ge 31$. 
\end{proof}
%%%%%%%%%%%%%%%%%%%%%%%%%%%%%% PROOF PROOF PROOF

%%%%%%%%%%%%%%%%%%%%%%%%%%%%%% PROOF PROOF PROOF
\begin{proof}[Proof of \eqref{e.B8mean} and \eqref{e.B8var}.] 
The equation \eqref{e.B8mean} is just the definition of conditional expectation.  
Note that as $ V$ is $ (4, \vartheta ,4)$-uniform, we have 
\begin{align} \notag 
\mathbb E _{x _{1,2,3} ^{0}, x _{1,2,3} ^{1} \in S _{1,2,3}} 
Z \cdot U& = \operatorname B_8[U,V]
\\& \notag 
= 
\delta _{U\mid V} ^{7}\norm V . \Box ^{1,2,3} S _{1,2,3}. ^{8} + \epsilon \,, 
\\ \label{e.EZU} 
& = \delta _{U\mid V} ^{7} \delta _{V\mid 4} ^{8} \prod _{1\le j<k\le 3} \delta _{j,k} ^{4} + \epsilon \,, 
\\
\label{e.EZerror}
\lvert  \epsilon \rvert & \le 
\tfrac 1 {20} (\delta_{U\mid V} \tau )^{30} \operatorname B_8 [V]\,, 
\end{align}
by \eqref{e.B8UV}, and \eqref{e.Uuniform}.

The 
inequality \eqref{e.B8var} is clearly a relative of Lemma~\ref{l.Zvar}, but does not follow from 
any principal like that which we have stated.  Indeed, we will see that \eqref{e.Tbox-twoD} is 
instrumental to this inequality, as it has been to the prior inequalities. 
Recalling \eqref{e.elemVar}, we see that we 
need to estimate $ \mathbb E Z ^2 \cdot U $.  This is a linear form on $ U$ and $ V$, which we now specify. 
Take $ \Omega \subset \{0,1,2\} ^{1,2,3}$ be set of maps $ \epsilon \;:\; \{1,2,3\} \to \{0,1,2\} $ such that 
the range of $ \epsilon $ does not include both $ 1$ \emph{and } $ 2$.  Then, 
\begin{equation}\label{e.ZUV=}
\mathbb E _{x _{1,2,3} ^{0} \in S _{1,2,3}}  Z ^2 \cdot U 
= \mathbb E _{\substack{x _{1,2,3} ^{j} \in S _{1,2,3}\\ j=1,2,3 }} 
V (x ^{1} _{1,2,3})V (x ^{2} _{1,2,3}) \prod _{\substack{\epsilon \in \Omega \\ \epsilon \not \equiv 1,2 }} 
U (x _{1,2,3} ^{\epsilon }) \,. 
\end{equation}

There are $  13$ occurrences of $ U$ in this expression. (Of the $ 7$ occurrences of $ U$ in $ \operatorname B_8[U,V]$, 
all but one get `doubled' in the expression above.)   Each occurrence is expanded as 
 as $ f_1+f_0$, where $ f_1= \delta _{U\mid V} V$.  The leading term is when each occurrence of $ U$ 
 is replaced by $ f_1$. This leads to 
 \begin{align}\notag 
\delta _{U\mid V} ^{13} 
\mathbb E _{\substack{x _{1,2,3} ^{j} \in S _{1,2,3}\\ j=1,2,3 }}  
\prod _{\substack{\epsilon \in \Omega }}
V (x _{1,2,3} ^{\epsilon })
& \stackrel u = \delta _{U\mid V} ^{13}   \delta _{4} ^{\lvert  \Omega \rvert }
\prod _{1\le j<k\le 3} \delta _{j,k} ^{ \lvert  \{\omega \vert _{j,k} \mid \omega \in \Omega  \}\rvert }
\\ \label{e.ZUleading}
& \stackrel u = \delta _{U\mid V} ^{13}   \delta _{4} ^{15} 
\prod _{1\le j<k\le 3} \delta _{j,k} ^{ 7} =\delta _{U\mid V} ^{13} \cdot  L_V\,. 
\end{align}
Recall that this last expectation can be estimated by assumption that $ V$ is $ (4, \vartheta ,4)$-uniform, 
see \eqref{e.Uuniform}.  

In each of the $ 2 ^{13}-1$ remaining terms, there is at least one occurrence of $ U$ which is replaced 
by $ f _{0}$.  As in the previous two proofs, we are again in a situation in which \eqref{e.2Box} applies. 
Therefore, as \eqref{e.Tbox-twoD} fails,  each of these terms is at most 
\begin{equation} \label{e.ZUc}
2 L_V \bigl\{ \vartheta' +  c_2 (\delta _{U\mid V} \tau ) ^{t_2}  \bigr\}\,. 
\end{equation}
Therefore, for $ c_2$ sufficiently small, and $ t_2$ sufficiently large, we can combine \eqref{e.ZUc},
\eqref{e.ZUleading} and \eqref{e.ZUV=} to conclude that 
\begin{align}\label{e.ZUV==}
\mathbb E _{x _{1,2,3} ^{0} \in S _{1,2,3}}  Z ^2 \cdot U 
 &= 
\delta _{U\mid V} ^{13}   L_V + \epsilon '
\\ \label{e.ZUVerror} 
\lvert  \epsilon '\rvert & \le c_2' L_V   (\delta _{U\mid V} \tau ) ^{t_2} \,. 
\end{align}
Here,  the implied constant in `$\stackrel u = $' depends upon the failure of the inequality \eqref{e.Tbox-twoD}, 
and $ L_V$ is defined in \eqref{e.ZUleading}. 

Now observe that combining \eqref{e.EZU} and \eqref{e.ZUV=} and \eqref{e.ZUleading}, we have 
\begin{align} \notag 
\mathbb P (U \mid T_4) \cdot \mathbb E Z ^2 \cdot U
& =    \delta _{U\mid V} ^{14}   \delta _{4} ^{16} 
\prod _{1\le j<k\le 3} \delta _{j,k} ^{ 8} + \epsilon'  \cdot \mathbb P (U \mid T_4) 
\\   \label{e.RV} 
& = \bigl(\mathbb E Z \cdot U \bigr) ^2 + \epsilon '' 
\\ \label{e.e''}
\lvert  \epsilon '' \rvert & \le  c_2' \operatorname B_8[V] ^2 [ 
 (\delta _{U\mid V} \tau ) ^{t_2} + 
\tfrac 1 {20} (\delta_{U\mid V} \tau )^{30} ] ^2  \,. 
\end{align}
In the last line, we have used  \eqref{e.EZerror} and 
\eqref{e.ZUVerror}.   Dividing \eqref{e.RV} by $ \mathbb P (U\mid T_4) ^2 $, and using the estimate 
in \eqref{e.e''}  completes the proof of \eqref{e.B8var}. 

\end{proof}
%%%%%%%%%%%%%%%%%%%%%%%%%%%%%% PROOF PROOF PROOF

We can complete the proof of Lemma~\ref{l.Tbox}, assuming the inequalities \eqref{e.B8U}---\eqref{e.B8var}. 
For a suitably generic point $ x ^{0} _{1,2,3} \in U$, we define the new data in \eqref{e.T'system} to be 
\begin{equation*}
S_1' (x^0 _{1,2,3}) = \{x_1 ^{1} \mid  x _{1,2,3} ^{1,0,0} \in U\}\,, 
\end{equation*}
with a corresponding definition for $ S_2'(x^0 _{1,2,3})$ and $ S_3'(x^0 _{1,2,3})$.   
The set $ S' _{1,2} (x^0 _{1,2,3})$ is defined as 
\begin{equation*}
S' _{1,2} (x^0 _{1,2,3}) 
= \{ x _{1,2} ^{1} \in S_1' (x^0 _{1,2,3}) \times S_2' (x^0 _{1,2,3}) \mid 
x ^{1,1,0} _{1,2,3} \in U  \} \,, 
\end{equation*}
with a corresponding definition for $ S ' _{1,3} (x^0 _{1,2,3})$ and $ S' _{2,3} (x^0 _{1,2,3})$.  
Last of all, the set $ T' (x^0 _{1,2,3})$ is taken to be 
\begin{equation*}
T' (x^0 _{1,2,3}) 
=  \{ x ^{1} _{1,2,3} \in V \mid 
x ^{1,1,0} \in S ' _{1,2} (x^0 _{1,2,3}) \,, \, 
x ^{1,0,1} \in S ' _{1,3} (x^0 _{1,2,3}) \,, \, 
x ^{0,1,1} \in S ' _{2,3} (x^0 _{1,2,3})
\} \,. 
\end{equation*}
With these definitions, note that \eqref{e.V'} holds, that is if $ V=T_4$, then $ V'=T' (x^0 _{1,2,3}) = 
T'_4$ in the new $ \mathcal T$-system.  
The point of this definition is that 
\begin{equation}\label{e.T'point}
\mathbb E _{x ^{0} _{1,2,3}, x ^{1} _{1,2,3} \in S _{1,2,3}} 
U (x^0 _{1,2,3}) T' (x^0 _{1,2,3}) = \operatorname B_8 [U,V]\,, 
\end{equation}
with the last expression found in \eqref{e.B8UV}.

Now, set 
\begin{equation*}% \label{e.U'}
U'= \bigl\{ x ^{0} _{1,2,3} \in U \mid    
\mathbb P _{x^1 _{1,2,3} S _{1,2,3} }  (T' (x^0 _{1,2,3}))  \ge \tfrac 1 {4} \delta _{U\mid V} ^{7} 
\operatorname B_8 [ V]
\bigr\}\,. 
\end{equation*}
It follows from \eqref{e.B8mean} and 
\eqref{e.B8var} that we have 
\begin{align*}
\mathbb P _{x^0 _{1,2,3} \in S _{1,2,3}   }  (U-U') 
& \le \mathbb P (U)  \cdot \bigl( (\tau \delta _{U\mid V}) ^{7} \operatorname B_8 [V]\bigr) ^2 
\operatorname {Var} (Z\mid U) 
\\
& \le \mathbb P (U) (\tau  \delta _{U\mid V} ) ^{14} \,. 
\end{align*}

Now, it will follow from the $ (4, \vartheta , 3)$-uniformity of $ V$, and Lemma~\ref{l.Zvar} that 
we have 
\begin{equation*}
\operatorname {Var} _{x^0 _{1,2,3}} \Bigl( \mathbb E _{x^1 _{1,2,3} \in S _{1,2,3}} 
\prod _{\epsilon \in \{0,1\} ^{1,2,3}} 
V (x ^{\epsilon } _{1,2,3}) \mid V (x^0 _{1,2,3}) \Bigr) 
\le \vartheta \operatorname B_8[V] ^2  \,. 
\end{equation*}
Here, $ \vartheta $ is as in \eqref{e.zvoDef}. Therefore, it will follow that in the formula \eqref{e.B8U}, 
 we can change the leading $ U (x^0 _{1,2,3})$ by $ U' (x^0 _{1,2,3})$.  Namely, we have 
\begin{align} \notag 
 \operatorname B_8 [ U-U', U ,\dotsc, U] 
& \le \operatorname B_8 [ U-U', V ,\dotsc, V] 
\\  \label{e.U'}
& \le 2  (\tau  \delta _{U\mid V} ) ^{14}  \operatorname B_8[V] \,. 
\end{align}

We can conclude this proof by estimating as follows: For element $ x^0 _{1,2,3} \in U'$, we have 
\begin{align*}
\sup _{x^0 _{1,2,3} \in U'} \mathbb P (U\mid T) & 
= 
\frac { \mathbb E _{x^1 _{1,2,3} \in S _{1,2,3}}   \prod _{ \substack{\epsilon \in \{0,1\} ^{1,2,3} \\ \epsilon \not\equiv 0 }}  
U (x _{1,2,3} ^{\epsilon }) }
{  \mathbb E _{x^1 _{1,2,3}\in S _{1,2,3}}  V (x^1 _{1,2,3}) 
\prod _{ \substack{\epsilon \in \{0,1\} ^{1,2,3} \\ \epsilon \not\equiv 0,1 }}  
U (x _{1,2,3} ^{\epsilon })  }
\\
& \ge 
\frac { \mathbb E _{x^0 _{1,2,3},x^1 _{1,2,3}\in S _{1,2,3}}  
U' (x^0 _{1,2,3})
\prod _{ \substack{\epsilon \in \{0,1\} ^{1,2,3} \\ \epsilon \not\equiv 0 }}  
U (x _{1,2,3} ^{\epsilon }) }
{  \mathbb E _{x^0 _{1,2,3},x^1 _{1,2,3}\in S _{1,2,3}} U' (x^0 _{1,2,3}) V (x^1 _{1,2,3}) 
\prod _{ \substack{\epsilon \in \{0,1\} ^{1,2,3} \\ \epsilon \not\equiv 0,1 }}  
U (x _{1,2,3} ^{\epsilon })  }
\\
& \ge \delta _{U\mid V} + \tfrac 14 \tau ^{8}\,. 
\end{align*}
The last line follows by combining \eqref{e.B8U}, \eqref{e.B8UV}, and \eqref{e.U'}, with this last inequality 
showing that modifications of \eqref{e.B8U} and \eqref{e.B8UV} hold, with the leading $ U (x^0 _{1,2,3})$ 
replaced by $ U' (x^0 _{1,2,3})$.

%%%%%%%%%%%%%%%%%%%%%%%%%%%%%% SECTION  SECTION SECTION
%%%%%%%%%%%%%%%%%%%%%%%%%%%%%% SECTION  SECTION SECTION 
\section{Proof of Uniformizing Lemma} \label{s.uniform}

We marshal several facts, and set some notations, before beginning 
the main lines of the proof of the Information Lemma~\ref{l.uni}.

%%%%%%%%%%%%%%%%%%%%%%%%%%%%%% SUBSECTION SUBSECTION SUBSECTION SUBSECTION
 %%%%%%%%%%%%%%%%%%%%%%%%%%%%%% SUBSECTION SUBSECTION SUBSECTION SUBSECTION 
\subsection{Martingales}%\label{ss.}
We will use basic facts about martingales.  
Let $ Z$ be a  real-valued random variable on a probability space $ \Omega $, 
bounded by one.   And let $\mathsf  P$ be a finite partition of $ \Omega $. 
Elements of the partition we refer to as \emph{atoms}. 
The 
\emph{conditional expectation  of $ Z$ relative to $\mathsf P$} 
is 
\begin{equation*}
\mathbb E (Z\mid \mathsf P )
 \coloneqq 
\sum _{A\in \mathsf P } A \cdot \mathbb P (A) ^{-1} \mathbb E (Z \cdot A)\,. 
\end{equation*}

Partition $ \mathsf P $ \emph{refines} $ \mathsf Q$ iff each element of $\mathsf  Q$ is 
a finite union of elements of $ \mathsf P$.
In our application, all partitions will be a finite collection of sets. 
Let $ \mathsf P _n$ be a sequence of refining partitions of $ \Omega $, 
that is, $ \mathsf P _{n}$ is a \emph{refining} 
sequence of partitions means that $ \mathsf P _{n+1}$ refines $ \mathsf P _{n}$ for all integers $ n$. 
We will take $ \mathsf P _0$ to be the trivial partition, namely $ \mathsf P _0= \{\Omega \}$. 

The sequence of random variables $ \mathbb E (Z \mid \mathsf P _n)$ is an example of 
a \emph{martingale}.  The sequence of random variables $ \Delta Z_n =\mathbb E (Z \mid \mathsf P _n)
- \mathbb E (Z \mid \mathsf P _ {n-1})$ for $ n\ge 1$ is a \emph{martingale difference sequence.} 
Then, the sum below is telescoping
\begin{equation*}
\mathbb E (Z \mid \mathsf P _n)= \mathbb E (Z \mid \mathsf P _0)+\sum _{m=1} ^{n} \Delta Z _m\,. 
\end{equation*}

Observe that the martingale difference sequence is a sequence of pairwise orthogonal 
random variables.  That is, for $ m<n$, 
\begin{equation}\label{e.mds}
\mathbb E \Delta Z_m \cdot \Delta Z _n=0\,. 
\end{equation}
Indeed, as the partitions $ \mathsf P _n$ are refining, and $ m<n$, for each 
element $ E\in \mathsf P _m$, the random variable $ \Delta Z_m $ is constant on $ E$, while 
$ \mathbb E \Delta Z _n \cdot E =0$.   This leads us to: 

%%%%%%%%%%%%%%%%%%%%%%%%%%%%%% PROPOSITION PROPOSITION PROPOSITION
\begin{proposition}\label{p.Stopping} Let $ 0<u<1$. 
Suppose that $ Z$ is a random variable bounded by $ 1$, and that 
$ \mathsf P _n$ is the sequence of refining partitions such that 
for an increasing sequence of integers $ t_m$ we have 
\begin{equation*}
\mathbb E [ \mathbb E (Z\mid \mathsf P _{t_m-1} )] ^2 +u 
\le 
\mathbb E [ \mathbb E (Z\mid \mathsf P _ {t_{m}} )] ^2  
\,, \qquad 1\le m<M\,. 
\end{equation*}
Then,  $M\le u  $. 
\end{proposition}
%%%%%%%%%%%%%%%%%%%%%%%%%%%%%% PROPOSITIOM PROPOSITIOM PROPOSITIOM

%%%%%%%%%%%%%%%%%%%%%%%%%%%%%% REMARK REMARK REMARK
\begin{remark}\label{r.stoppingtimes} Below, we will refer to an increasing sequence of 
integers as `stopping times.'  An extension of this definition, to make the stopping 
times certain sequences of measurable functions, is an essential tool in martingale 
theory. 
\end{remark}
%%%%%%%%%%%%%%%%%%%%%%%%%%%%%% REMARK REMARK REMARK

%%%%%%%%%%%%%%%%%%%%%%%%%%%%%% PROOF PROOF PROOF
\begin{proof}
Notice that the assumption tells us that $  \mathbb E (\Delta Z _{t _{m}}) ^2 \ge u$.  
Indeed, since $ \mathbb E (Z\mid \mathsf P _ {t_{m}})= \mathbb E (Z\mid \mathsf P _ {t_{m}-1}) 
+ \Delta Z _{t_m}$, and orthogonality of martingale difference sequences, 
\begin{align*}
\mathbb E (\Delta Z _{t _{m}}) ^2
& =
\mathbb E [ \mathbb E (Z\mid \mathsf P _ {t_{m}}) ] ^2   
-2 \mathbb E [ \mathbb E (Z\mid \mathsf P _ {t_{m}}) \cdot \mathbb E (Z\mid \mathsf P _ {t_{m}-1})] 
+ \mathbb E [\mathbb E (Z\mid \mathsf P _ {t_{m}-1}) ^2 ]
\\
&
=\mathbb E [ \mathbb E (Z\mid \mathsf P _ {t_{m}}) ] ^2   
- \mathbb E [\mathbb E (Z\mid \mathsf P _ {t_{m}-1}) ^2 ]
\\ & \ge u \,. 
\end{align*}

We then have 
\begin{align*}
1 \ge \mathbb E Z ^2 \ge \sum _{m=1} ^{M} \mathbb E [\mathbb E (\Delta Z _{t _{m}}) ^2] \ge N u\,. 
\end{align*}
\end{proof}
%%%%%%%%%%%%%%%%%%%%%%%%%%%%%% PROOF PROOF PROOF

We will use the extension of the previous proposition. 

%%%%%%%%%%%%%%%%%%%%%%%%%%%%%% COROLLARY COROLLARY COROLLARY
\begin{corollary}\label{c.Stopping}  
Suppose that $ \Omega '\subset \Omega $, where $ (\Omega, \mathbb P )  $ is a probability 
space.  Let $\mathsf P$ be a partition of $ \Omega '$ into a finite number of sets.  
Let $ \mathsf P _m$ be a sequence of refining partitions of $  p  $, and $ t _{m} ( p  )$, for $  p  \in P$, 
be a set of stopping times so that for all $ 1\le m \le M (p)$ we have 
\begin{equation*}
\mathbb E [ \mathbb E ( p  \mid \mathsf P _{t_m ( p  )-1} ) ^2 ]+u 
\le 
\mathbb E [ \mathbb E (Z\mid \mathsf P _ {t_{m} ( p  )} ) ^2 ] 
\,, \qquad   p  \in P\,,\ 1\le m<M ( p  )\,. 
\end{equation*}
Then, 
\begin{equation}\label{e.cStopping}
\sum _{ p  \in P} M ( p  )\le u ^{-1} \,. 
\end{equation}
\end{corollary}
%%%%%%%%%%%%%%%%%%%%%%%%%%%%%%  COROLLARY COROLLARY COROLLARY

%%%%%%%%%%%%%%%%%%%%%%%%%%%%%% PROOF PROOF PROOF
\begin{proof}
We have 
\begin{align*}
1 & \ge \sum _{ p  \in P } \mathbb P ( p  )
 \ge \sum _{ p  in P} \sum _{m=1} ^{M ( p  )} 
\mathbb E [\mathbb E (\Delta   p  _{t_m}) ^2 ]
 \ge \sum _{ p  in P} \sum _{m=1} ^{M ( p  )} u\,. 
\end{align*}
And this proves our Corollary. 
\end{proof}
%%%%%%%%%%%%%%%%%%%%%%%%%%%%%% PROOF PROOF PROOF

Here is an extension of the previous propositions, where the conditional 
variance increment is permitted to be much smaller.

%%%%%%%%%%%%%%%%%%%%%%%%%%%%%% PROPOSITION PROPOSITION PROPOSITION
\begin{proposition}\label{p.STopping}
 Let $ 0<u, \tau <1$, and $ C\ge 1$. 
Suppose that $0\le  Z\le 1 $ is a random variable, and that 
$ \mathsf P _m$ is the sequence of refining partitions, and that $ t_m$ is a sequence 
of stopping times  such that for all $ 1\le m \le M$,  
\begin{gather*}
\mathbb E  [Z \cdot E _{m}] \ge \tau 
\\
E_m \coloneqq 
\bigl\{  p \in \mathsf P _{t_m-1} \mid  \mathbb E [\mathbb E (Z \cdot  p \mid \mathsf P _{t_m} )] ^2 
\ge \mathbb E (Z\mid   p  ) ^2 + u   \mathbb E (Z\mid   p  ) ^C  \bigr\}
\end{gather*}
Then,  $M\le u ^{-2} \tau ^{-C} $. 
\end{proposition}
%%%%%%%%%%%%%%%%%%%%%%%%%%%%%% PROPOSITION PROPOSITION PROPOSITION

%%%%%%%%%%%%%%%%%%%%%%%%%%%%%% PROOF PROOF PROOF
\begin{proof}
Observe that for 
$
\Delta _{m} \coloneqq \mathbb E (Z\mid  \mathsf P _ {t_m}) - \mathbb E (Z\mid \mathsf P _{t_m-1})
$ we have the estimate 
\begin{equation*}
\mathbb E [\Delta _m ^2 \cdot E_m ]\ge u ^2  \mathbb E [\mathbb E (Z\mid   \mathsf P _{t _m-1}  ) ^C E_m] \,. 
\end{equation*}
Therefore, using Jensen's inequality, available to us as  $C\ge 1 $, 
\begin{align*}
1 & \ge \sum _{m=1} ^{M} \mathbb E \Delta _m ^2  
 \ge \sum _{m=1} ^{M} \mathbb E \Delta _m ^2 E_m 
\ge \sum _{m=1} ^{M}  u ^2  \mathbb E [ \mathbb E (Z\mid   \mathsf P _{t _m-1}  ) ^C E_m ] 
\\&\ge \sum _{m=1} ^{M}  u ^2  \mathbb E [ \mathbb E (Z\mid   \mathsf P _{t _m-1}  )  E_m ] ^{C}
\ge M  u ^2  \tau ^{C} \,. 
\end{align*}
This proves the Proposition. 
\end{proof}
%%%%%%%%%%%%%%%%%%%%%%%%%%%%%% PROOF PROOF PROOF

% 
% Below we will use the notation 
% \begin{equation}\label{e.condVar}
% \operatorname {Var} (Z; \mathsf P _n ; \Omega )= \mathbb E ( \mathbb E (Z \mid \mathsf P _n)
% - \mathbb E (Z \mid \mathsf P _{n-1}))  ^2 \,. 
% \end{equation}
% The underlying probability space $ \Omega $ will change, thus 
% the addition of it to the notation above. 

%%%%%%%%%%%%%%%%%%%%%%%%%%%%%% SUBSECTION SUBSECTION SUBSECTION SUBSECTION
 %%%%%%%%%%%%%%%%%%%%%%%%%%%%%% SUBSECTION SUBSECTION SUBSECTION SUBSECTION 
\subsection{Partitions}%\label{ss.}

We need several partitions, which `fit together' in an appropriate way. 

Let $ \Omega $  be a set with partition $\mathsf  P$.  Let $ \Omega '\subset \Omega $ 
have partition $ \mathsf P '$.  Say that $ \mathsf P '$ is \emph{subordinate} to $\mathsf P$ iff 
each atom $ p'\in \mathsf P '$ is contained in some atom $ p\in \mathsf P$.  We do not insist that 
every atom of $\mathsf P$ be a union of atoms from $ \mathsf P '$, that is, we do not require that 
$ \mathsf P '$ refine $ \mathsf P$. 

The \emph{minimum} of two partitions $ \mathsf P $ and $ \mathsf P '$ of the same set $ \Omega $ is 
\begin{equation*}
\mathsf P \wedge \mathsf P ' = \{ A\cap B \mid A\in  P \,,\, B\in \mathsf P ' \}. 
\end{equation*}
If $ \mathsf P '$ is a partition of a subset $ \Omega '\subset \Omega $, we use the 
same notation $ \mathsf P\wedge \mathsf P '$ for a (maximal) partition of $ \Omega '$ subordinate 
to both $\mathsf P$ and $ \mathsf P '$.

Suppose that $\mathsf P$ is a partition in $ \Omega $, and that $ \mathsf P '$ is a partition 
of $ \Omega '\subset \Omega $, that is subordinate to $\mathsf P$.  We define 
\begin{equation}\label{e.multi}
\operatorname {multi} (\mathsf P '\mid \mathsf  P)
=\sup _{ p  \in \mathsf P} \sharp  \{  p  '\in \mathsf P '\mid  p  '\subset  p  \}\,. 
\end{equation}

%%%%%%%%%%%%%%%%%%%%%%%%%%%%%% SUBSECTION SUBSECTION SUBSECTION SUBSECTION
 %%%%%%%%%%%%%%%%%%%%%%%%%%%%%% SUBSECTION SUBSECTION SUBSECTION SUBSECTION 
\subsection{Useful Propositions}%\label{ss.}

This general proposition provides the motivation for the overall 
approach we take. 

%%%%%%%%%%%%%%%%%%%%%%%%%%%%%% PROPOSITION PROPOSITION PROPOSITION
\begin{proposition}\label{p.partitions} 
Let $ 0<v<\delta <1$.  Let $ A\subset T\subset X$ be finite sets 
with $ \mathbb P (A\mid T)\ge \delta + v$. 
Let $\mathsf P$ be a partition of $ X$, and let $ \mathsf P '\subset \mathsf P $ be any subset 
of $\mathsf P$ for which 
\begin{equation}\label{e.piSmall}
\mathbb P \Bigl( \bigcup _{p\in \mathsf P '} p \Bigr) 
\le v/4\,. 
\end{equation}
Then, there is some element $ p\in \mathsf P - \mathsf P '$ with 
\begin{equation}\label{e.piConclude}
\mathbb P (T\mid p)\ge \tfrac v 4 \mathbb P (T\mid X)\,, 
\qquad 
\mathbb P (A\mid T\cap p) \ge \delta +\tfrac v2\,. 
\end{equation}
\end{proposition}
%%%%%%%%%%%%%%%%%%%%%%%%%%%%%% PROPOSITION PROPOSITION PROPOSITION

%%%%%%%%%%%%%%%%%%%%%%%%%%%%%% PROOF PROOF PROOF
\begin{proof}
Take $ \mathsf P ''$ to be all those elements $p\in  \mathsf P $ which are in $ \mathsf P '$ or 
$   \mathbb P  (T\mid \mathsf  P)\le \tfrac v 4 \mathbb P (T\mid X)$.  It is clear that we have 
\begin{equation*}
\mathbb P \Bigl( A\cap \bigcup _{p\in \mathsf P ''} p \mid T \Bigr)\le \tfrac v2 \,. 
\end{equation*}
Applying the pigeonhole principle to those elements of $ \mathsf P - \mathsf P ''$ proves the 
Proposition.  
\end{proof}
%%%%%%%%%%%%%%%%%%%%%%%%%%%%%% PROOF PROOF PROOF

The `energy increment' steps we take are governed by these two general propositions. 

%%%%%%%%%%%%%%%%%%%%%%%%%%%%%% PROPOSITION PROPOSITION PROPOSITION
\begin{proposition}\label{p.simpleEnergy} Let $ A$ be a subset of a probability 
space $ (\Omega , \mathbb P )$.  Suppose that the there is a subset $ B\subset \Omega $ 
for which we have 
\begin{equation*}
\mathbb P (A\mid B)=\mathbb P (A)+ \nu > \mathbb P (A)\,.
\end{equation*}
Then, for the partition $ \mathsf P _B$ of $ \Omega $ generated by $ B$, we have 
\begin{equation} \label{e.simpleEnergy}
\mathbb E [\mathbb E (A\mid \mathsf P _B)] ^2 \ge \mathbb P (A) ^2 + \mathbb P  (B) \cdot \nu  ^{2}\,. 
\end{equation}
\end{proposition}
%%%%%%%%%%%%%%%%%%%%%%%%%%%%%% PROPOSITION PROPOSITION PROPOSITION

In application, we will have $ \nu \,,\, \mathbb P (B)\ge \mathbb P (A) ^{C}$, for 
an absolute constant $ C$.  Thus, we have 
\begin{equation} \label{e.SimpleEnergy}
\mathbb E [\mathbb E (A\mid \mathsf P _B)] ^2 \ge \mathbb P (A) ^2 + \mathbb P  (A) ^{3C}\,. 
\end{equation}

%%%%%%%%%%%%%%%%%%%%%%%%%%%%%% PROOF PROOF PROOF
\begin{proof}
Let us set $ \alpha = \mathbb P (A)$, $ \mathbb P (B)= \beta $ so that 
\begin{equation*}
\mathbb P (A \cap B)= (\alpha + \nu ) \beta \,, \qquad 
\mathbb P (A \cap B ^{c} )= (1- \beta ) \alpha - \nu \beta \,. 
\end{equation*}

We can calculate the left-hand side of \eqref{e.simpleEnergy} directly. 
\begin{align*}
\mathbb E [\mathbb E (A\mid \mathsf P _B)] ^2
&= \mathbb P (B) [ \mathbb P (A\mid B)] ^2 + 
(1- \mathbb P (B)) [ \mathbb P (A\mid B ^{c})] ^2 
\\
&= \mathbb P (A \cap B) \cdot  \mathbb P (A\mid B)  
+ \mathbb P (A \cap B ^{c}) \mathbb P (A \mid B ^{c})
\\
&=( \alpha + \nu ) ^2 \beta + (1- \beta ) ^{-1} [ (1- \beta )\alpha + \nu \beta ] ^2 
\\
&= \alpha ^2 + (1- \beta ) ^{-1}  \nu  ^2 \beta 
\\
&\ge \alpha ^2 + \nu  ^2 \beta \,. 
\end{align*}
And this proves the proposition. 
\end{proof}
%%%%%%%%%%%%%%%%%%%%%%%%%%%%%% PROOF PROOF PROOF

This trivial extension of the previous proposition is the one that we use.

%%%%%%%%%%%%%%%%%%%%%%%%%%%%%% PROPOSITION PROPOSITION PROPOSITION
\begin{proposition}\label{p.Energy}  Let $ A$ be a subset of a probability space 
$ (\Omega , \mathbb P )$, and let $\mathsf P$ be a finite partition of $ \Omega $ so that 
this condition holds.  For a subset $ Q\subset \mathsf P$, suppose the following holds.  
For each element $ p \in \mathsf P$, there is a further subset $ p'$ so that 
\begin{gather*}
\mathbb P (A \mid  p  ') \ge \mathbb P (A\mid  p  )+ \nu 
\,, \qquad  p  \in Q\,. 
\\
\mathbb P \Bigl( \bigcup _{ p  \in \mathsf P}  p '  \Bigr)\ge \tau \,. 
\end{gather*}
Then, for the partition $ \mathsf P '$ which refines both $\mathsf P$ and $ \{ p  '\mid  p  \in Q\}$,
we have the estimate 
\begin{equation*}
\mathbb E [\mathbb E (A\mid \mathsf P ') ]^{2}\ge 
\mathbb E [ \mathbb E (A\mid \mathsf P) ] ^2 + \tau \nu ^2 \,. 
\end{equation*}
\end{proposition}
%%%%%%%%%%%%%%%%%%%%%%%%%%%%%% PROPOSITION PROPOSITION PROPOSITION

\bigskip 

We will appeal to a simple bound for the tower notation given by 
\begin{equation}\label{e.tower}
2 \uparrow n \coloneqq 2 ^{n}\,, \qquad 2 \uparrow\uparrow n \coloneqq 2 \uparrow( 2
\uparrow\uparrow n-1)\,. 
\end{equation}
In the function $ 2 \uparrow\uparrow n$ is  called the Ackerman function, and its 
inverse is 
\begin{equation}\label{e.log-star}
\log _{\ast } N=\min  \{n \mid N\le 2 \uparrow\uparrow n\}\,. 
\end{equation}

%%%%%%%%%%%%%%%%%%%%%%%%%%%%%% PROPOSITION PROPOSITION PROPOSITION
\begin{proposition}\label{p.simpleBound} For integers $ \ell , u ,v \ge 2$  
define 
\begin{equation*}
\psi (0, u ,v)= u \cdot  {v}\,, \qquad 
\psi (\ell +1, u, v)= 2 \uparrow ( u \cdot  \psi (\ell ,u,v))
\end{equation*}
We have the estimate 
\begin{equation*}
\psi (\ell ,u,v)\le 2 \uparrow\uparrow [ \ell + \log _{\ast } 2 u v ] \,. 
\end{equation*}
\end{proposition}
%%%%%%%%%%%%%%%%%%%%%%%%%%%%%% PROPOSITION PROPOSITION PROPOSITION

%%%%%%%%%%%%%%%%%%%%%%%%%%%%%% PROOF PROOF PROOF
\begin{proof}
Define 
\begin{align*}
\epsilon _{\ell }= \frac { \log _2 u} {u \psi (\ell -1)}\,, 
\qquad 
\epsilon _{k-1 }= \frac {\log _2 u (1+ \epsilon _{k})} {u \psi (k-1) }\,. 
\end{align*}
It is elementary to see that $ \epsilon _1\le 1$.

The point of these definitions is that we have 
\begin{align*}
\psi (\ell , u, v)&=2 \uparrow [ (1+ \epsilon _ \ell ) u \psi (\ell -1)] 
\\
&= 2 \uparrow [ 2 \uparrow [ (1+ \epsilon  _{\ell -1}) \psi (\ell -2)] 
\\
&\;\,\vdots  
\\
&=\stackrel { \textup{$ \ell $ times}} 
{\overbrace{2 \uparrow [ 2 \uparrow [ \cdots 2  \uparrow [ (1+ \epsilon _1) u v] \cdots ]]}}
\\
& \le 2 \uparrow\uparrow [ \ell + \log _{\ast } 2 u v ] \,. 
\end{align*}
\end{proof}
%%%%%%%%%%%%%%%%%%%%%%%%%%%%%% PROOF PROOF PROOF

The following definition is used to make a quicker appeal to Lemma~\ref{l.BPZ}, 
and its relative Lemma~\ref{l.Tbox}.

%%%%%%%%%%%%%%%%%%%%%%%%%%%%%%  DEFINITION DEFINITION DEFINITION
\begin{definition}\label{d.good} 
Consider a subset $ S $ of a set $ X $, a partition $\mathsf P$, 
and a positive parameter $ \Delta $.  Say that $ \mathsf P ' $ is 
\emph{$ (S, \Delta ,\mathsf P ) $-good} iff $ \mathsf P ' $ refines $\mathsf P$ and 
\begin{equation}  \label{e.good}
\mathbb E (\mathbb E (S\mid \mathsf P ') ^2 )\ge \mathbb E (\mathbb E (S\mid \mathsf P ) ^2 )+\Delta \,. 
\end{equation}
\end{definition}
%%%%%%%%%%%%%%%%%%%%%%%%%%%%%%  DEFINITION DEFINITION DEFINITION

%%%%%%%%%%%%%%%%%%%%%%%%%%%%%% SUBSECTION SUBSECTION SUBSECTION SUBSECTION
 %%%%%%%%%%%%%%%%%%%%%%%%%%%%%% SUBSECTION SUBSECTION SUBSECTION SUBSECTION 
\subsection{The $ U (3)$ Norm}%\label{ss.}

In this section we discuss the Lemmas needed to obtain sets that are uniform with respect to the 
Gowers $ U (3)$ norm. 

%%%%%%%%%%%%%%%%%%%%%%%%%%%%%%  DEFINITION DEFINITION DEFINITION
\begin{definition}\label{d.affine}
We call a partition  of $ H \times H \times H$ \emph{affine} iff all atoms of the partition 
are of the form $ V_1 \times V_2 \times V_3$, where $ V_i$ are 
all translates of the \emph{same} subspace $ V\le H$.  
This is an essential definition for us, as an  affine partition, in 
say the basis $ (\operatorname e_1, \operatorname e_2, \operatorname e_3) $ 
is also affine in any choice of basis formed from these three vectors.  
Each atom of an affine partition is, after translation, a copy of 
$ H \times H \times H$ with a lower dimension.  
\end{definition}
%%%%%%%%%%%%%%%%%%%%%%%%%%%%%%  DEFINITION DEFINITION DEFINITION

In particular, given $ S_j$, $ 1\le j\le 4$, and an affine partition $ \mathsf P $, 
for each atom $ \alpha \in \mathsf P $, it makes sense to compute the Gowers uniformity 
norm of $ S_j$ relative to the atom $ \alpha $.  That is, the atom $ \alpha $ determines 
an affine  subspace $ V_j$ in the coordinate $ \operatorname e_j$.  After translation, 
we could assume that $ V_j$ is actually a subspace, in which we can unambiguously 
compute the Gowers $ U (3)$ norm.  This is what we mean by 
\begin{equation}\label{e.U3A}
\norm S_j  - \mathbb P (S_j\mid \alpha )  . U (3), \alpha . 
\end{equation}

The \emph{codimemsion} of an affine partition, written as $ \operatorname {codim} (\mathsf P )$ 
is the maximum codimension 
of $ V_1$ in $ H$, for all $ V_1 \times V_2 \times V_3 \in \mathsf P $.  
Clearly, we have 
\begin{equation}\label{e.index>}
\abs{ \mathsf P }\le 5 ^{\operatorname {codim} (\mathsf P )}\,. 
\end{equation}

We need the following version of the  Inverse Theorem for the $ U (3)$ Norm, in 
a 

\begin{inverse}  \label{p.uni0} 
There are  constant $0<c< C<\infty $ so that the 
following holds.  
Let $ S\subset H$ and assume that $ \operatorname {dim} (H)> 10C  u^{-C} $  and 
\begin{equation*}
\norm S - \mathbb P (S\mid H). U (3). >u
\end{equation*}
Then, there is an affine subspace $ H'$ of $ S$ so that 
$ \operatorname {dim} (H')\ge \operatorname {dim} (H)- Cu ^{-C}$ 
and 
\begin{equation*}
\mathbb P (S\mid H')\ge \mathbb P (S\mid H)+ c u ^{C}\,. 
\end{equation*}
\end{inverse}

We emphasize that the exact value of the estimates on the co-dimensions above are 
important in the study of four-term progressions, but the exact form of these 
estimates are not important to the proof of our Main Theorem, Theorem~\ref{t.main}. 
For this result, see \cite{math.NT/0503014}*{p. 27---28}.

We will use this elementary observation:  If $ \mathsf P ,\mathsf P '$ are 
affine partitions, then 
\begin{equation}\label{e.index+}
\operatorname {codim} (\mathsf P \wedge \mathsf P ') 
\le 
\operatorname {codim} (\mathsf P )+\operatorname {codim} (\mathsf P ')\,.
\end{equation}

%%%%%%%%%%%%%%%%%%%%%%%%%%%%%% PROPOSITION PROPOSITION PROPOSITION
\begin{proposition}\label{p.uni1} There is a constant $ C$ so that 
the following holds for all $ 0<u, \tau <1$ the following holds. 
Let $ S_j$, $ 1\le j\le 4$ be 
sets in the $ j$th coordinate.   Then there is 
an affine partition $ \mathsf P $ of $ H \times H \times H $, satisfying 
$ \operatorname {codim} (\mathsf P ) \lesssim  ( C /u \tau ) ^{C} $,  so that 
\begin{equation*}
\mathbb P (A\in \mathsf P  \mid \sup _{j} \norm S_j. U (3), A. >u)<\tau \,. 
\end{equation*}
\end{proposition}
%%%%%%%%%%%%%%%%%%%%%%%%%%%%%% PROPOSITION PROPOSITION PROPOSITION

%%%%%%%%%%%%%%%%%%%%%%%%%%%%%% PROOF PROOF PROOF
\begin{proof}

Here is an important point in the proof.  For an affine partition $ \mathsf P $, 
suppose there is an atom $ A\in \mathsf P $ such that 
\begin{equation*}
\norm S_j - \mathbb P (S_j \mid A). U (3), A. > u 
\end{equation*}
Let $ A_j$ denote  the affine subspace for coordinate $ \operatorname e_j$.  
Then, there is a partition $ \mathsf P _A$ of $ A_j$ into affine subspaces of 
codimension $  \le C u ^{-C}$, for which we have 
\begin{equation*}
\mathbb E _{A_j}  ( \mathbb E (S_j\cap A_j \mid \mathsf P _A) ^2 ) 
\ge 
\mathbb E _ {A_j} (S_j\cap A_j) ^2 + c u ^{C}\,. 
\end{equation*}
A moments thought shows that there is then an affine refinement 
$ \mathsf P '$ of $ \mathsf P $, in which only the atom $ A$ is further refined, for which 
we have 
\begin{equation} \label{e.A+}
\mathbb E ( \mathbb E (S_j \mid \mathsf P ') ^2 )
\ge 
\mathbb E ( \mathbb E (S_j \mid \mathsf P ) ^2 )
+ c u ^{C} \mathbb P (A). 
\end{equation}
Indeed, since the atom $ A$ is the product of translates of the \emph{same} subspace 
$ A_j$, we impose an appropriate translate of the partition $ \mathsf P _A$ on the 
two choices of the remaining coordinates.  The codimension of the refining partition 
has increased by only $  C u ^{-C}$.

\medskip

Here is the principal line of the argument.  We construct a sequence of 
refining affine partitions $ \mathsf P_n $, and a sequence of stopping times $ \tau _{j,k}$,
for $ 1\le j\le 4$ and $ k\ge 1$, 
which are used to running time of the recursive procedure below.

Let $ \mathsf P $ be an affine partition. 
Notice that there is some $ C>0$ so that the following is sufficient condition for the 
existence of a $ (S_j, u^C \tau , \mathsf P ) $-good partition $ \mathsf P '$: 
\begin{equation}  \label{e.goodSuff}
\mathbb P (A\in \mathsf P \mid  \norm S_j. U (3), A.>u)\ge \tau /4 
\end{equation}
In addition, $ \mathsf P '$ can be taken to be affine and 
$ \operatorname {codim} (\mathsf P ')\le \operatorname {codim} (\mathsf P )+C u ^{-C}$.
This is a consequence of the discussion at the beginning of the proof. 
The notion of a \emph{good} partition is defined in Definition~\ref{d.good}.

Initialize variables 
\begin{align*}
\mathsf P _0 \leftarrow \{H \times H \times H\}
\,, \qquad 
n\leftarrow 0\,,
\qquad 
\tau _{j,0}=0\,, \qquad k_j\leftarrow 0
\end{align*}
Likewise set $ \tau _{j,0}=0$
\textsf{WHILE} for some  
$ 1\le j\le 4$, there is an affine $ (S_j, u^C \tau  /4, \mathsf P _n)$-good partition $ \mathsf P '$, 
with $ \operatorname {codim} (\mathsf P _{n+1})\le \operatorname {codim} (\mathsf P _{n})+C u ^{-C}$, 
increment 
\begin{equation*}
 n\leftarrow n+1\,, \quad  k_j\leftarrow k_j+1\,.
\end{equation*}
Define $ \tau _{j, k_j}=n$, and $ \mathsf P _{n+1}=\mathsf P '$.  

As the underlying space is finite dimensional, this \textsf{WHILE} loop must stop. 
The sequence of stopping times 
$ \tau _{j,1},\dotsc, \tau _{j,K}$ cannot exceed $ (\tau u) ^{-C}$.  
Indeed, the hypotheses of Proposition~\ref{p.Stopping} hold, proving this claim 
immediately.  The conclusions of the Lemma are then immediate from the recursion, 
and the observation (\ref{e.goodSuff}).

\end{proof}
%%%%%%%%%%%%%%%%%%%%%%%%%%%%%% PROOF PROOF PROOF

In fact, we will rely upon the following variant of the the previous 
result. 

%%%%%%%%%%%%%%%%%%%%%%%%%%%%%% PROPOSITION PROPOSITION PROPOSITION
\begin{lemma}\label{l.uni2} There is a constant $ C$ so that 
the following holds for all $ 0<u, \tau <1$ the following holds. 
Let $ \mathcal S_j$, $ 1\le j\le 4$ be a collection of sets 
 in the $ j$th coordinate.   Then there is 
an affine partition $ \mathsf P $ of $ H \times H \times H $ of 
\begin{align*}
\operatorname {codim} (\mathsf P ) \lesssim  \Bigl[
(u \tau ) ^{-1}  \prod _{j=1} ^{4} \abs{ \mathcal S_j} \Bigr] ^{C}
\quad \textup{and} \quad 
\mathbb P (A\in \mathsf P \mid \sup _{j} \norm S_j. U (3), A. >u) &<\tau \,. 
\end{align*}
\end{lemma}
%%%%%%%%%%%%%%%%%%%%%%%%%%%%%% PROPOSITION PROPOSITION PROPOSITION

This proof is a simple variant of the previous proof.  Note that the codimension of the the 
partition admits a substantially worse bound.  This is because we have to 
keep track of a running time for each possible set $ S \in \bigcup _{j} \mathcal S_j$. 

%%%%%%%%%%%%%%%%%%%%%%%%%%%%%% SUBSECTION SUBSECTION SUBSECTION SUBSECTION
 %%%%%%%%%%%%%%%%%%%%%%%%%%%%%% SUBSECTION SUBSECTION SUBSECTION SUBSECTION 
\subsection{The Box Norm in Two Variables}%\label{ss.}

The goal of this subsection is Lemma~\ref{l.2BU}, which combines the fact about the $ U (3)$ norm in 
Lemma~\ref{l.uni2}, with some facts about the Box Norm.   
We begin with some generalities on the Box Norm in two variables. 
Recall the definition of $ \mathsf P '$ being $ (S, \delta ,\mathsf P)$-good given in \eqref{e.good} above. 

%%%%%%%%%%%%%%%%%%%%%%%%%%%%%% PROPOSITION PROPOSITION PROPOSITION
\begin{proposition}\label{p.GG} There is a $ C_2$ so that for all $ 0<u , \tau <1$ the 
following holds. 
Let $ Z\subset X \times Y$, and let $ \mathsf P _X$, $ \mathsf P _Y$ be partitions of $ X$ and $ Y$.
Suppose that the following condition holds. 
\begin{gather*}
\mathbb P ( E \mid X \times Y)\ge \tau \,, \qquad \textup{where}
\\
E= \{ (p_x,p_y)\in \mathsf P _X \times \mathsf P _Y \mid  \norm Z - \mathbb P (Z\mid p_x \times p_y) . \Box ^{x,y} 
p_x \times p_y . \ge u\}\,. 
\end{gather*}
Then, 
there are  partitions $ \mathsf P '_X$ and $ \mathsf P '_Y$ so that 
%%  ITEMIZE
\begin{gather}
\label{e.2good1} 
 \textup{ $ \mathsf P '_X \times \mathsf P '_Y$ is $ (Z, \tau u^{C_2}, \mathsf P _X \times \mathsf P _Y)$-good. }
\\
\label{e.2good2} 
 \textup{ $ \operatorname {multi} (\mathsf P '_X\mid \mathsf P _X) \le 2 \uparrow \sharp \mathsf P _Y$\,, and likewise for $ \mathsf P '_Y$.  }
\end{gather}
%% GATHER
Here, $ C_2$ could be taken to be $ 4$.  
\end{proposition}
%%%%%%%%%%%%%%%%%%%%%%%%%%%%%% PROPOSITION PROPOSITION PROPOSITION

Note that the estimate \eqref{e.2good2}, recursively applied, leads to tower power style
bounds.

%%%%%%%%%%%%%%%%%%%%%%%%%%%%%% PROOF PROOF PROOF
\begin{proof}
For each $ (p_x,p_y)\in E$, Lemma~\ref{l.BPZ} assures us the existence of 
a partition $ \mathsf P _x (y)$ of $ p_x$ into two elements, and a partition $\mathsf P _y (x)$ of $ p_y$ into 
two elements so that $ \mathsf P _x (y) \times \mathsf P _y (x)$ is $ (Z \cap p_x \times  p_y,  u ^{C_2}, 
p_x \times  p_y)$-good.  (There is no $ \tau $ in this last assertion.)

We take 
\begin{equation*}
\mathsf P '_X = \mathsf P _X \wedge \bigwedge _{y\in \mathsf P _Y} \mathsf P _{x} (y)\,, 
\end{equation*}
and likewise for $ \mathsf P '_Y$.  It is clear that \eqref{e.2good2} holds. 
By the assumption that $ \mathbb P (E)> \tau $, and the martingale property \eqref{e.mds}, 
it follows that \eqref{e.2good1} holds. 
\end{proof}
%%%%%%%%%%%%%%%%%%%%%%%%%%%%%% PROOF PROOF PROOF

%%%%%%%%%%%%%%%%%%%%%%%%%%%%%% PROPOSITION PROPOSITION PROPOSITION
\begin{proposition}\label{p.BoxTower} There is a $ C_2>0$ so that for all $0<u, \tau <1 $ 
the following holds. 
Let $ Z\subset X \times Y$, and let $ \mathsf P _X$, $ \mathsf P _Y$ be partitions of $ X$ and $ Y$. 
Let $ \mathsf P _Z $ be a partition of $ Z$ that is subordinate to $ \mathsf P _X \times \mathsf P _Y$. 
Suppose that the following condition holds. 
\begin{gather*}
\mathbb P ( E \mid Z)\ge \tau \,,
\\
E= \{ z\in \mathsf P _Z \mid  \norm z - \mathbb P (z\mid X_z \times Y_z) . \Box ^{x,y} X_z
\times Y_z . \ge u\}\,. 
\end{gather*}
Here, $ z\subset X_z \times Y_z$, and $ X_z\in \mathsf P _X$ and $ Y_z\in \mathsf P _Y$.  
$ X_z,Y_z$ must exist as $ \mathsf P _Z$ is subordinate to $ \mathsf P _X \times \mathsf P _Y$.  
Then, there is a partition $ \mathsf P '_X$ and $ \mathsf P '_Y$ so that 
%%  ITEMIZE
\begin{gather}
\label{e.22good1} 
 \textup{ $ \mathsf P '_X \times \mathsf P '_Y$ is $ (\mathsf P _Z, \tau u^{C_2}, \mathsf P _X \times \mathsf P _Y)$-good. }
\\
\label{e.22good2} 
 \operatorname {multi} (\mathsf P '_X\mid \mathsf P _X) \le 2 \uparrow [(\sharp \mathsf P _Y) \cdot 
 \operatorname {multi} (\mathsf P _Z \mid \mathsf P _X \times \mathsf P _Y)] \,,
\ \textup{ and likewise for $ \mathsf P '_Y$.  }
\end{gather}
%% GATHER
Here, $ C_2$ could be taken to be $ 4$.  
\end{proposition}
%%%%%%%%%%%%%%%%%%%%%%%%%%%%%% PROPOSITION PROPOSITION PROPOSITION

Note in particular the form of the tower in \eqref{e.22good2}, with the notation 
as in \eqref{e.tower}

%%%%%%%%%%%%%%%%%%%%%%%%%%%%%% PROOF PROOF PROOF
\begin{proof}
For each $ z\in E$, there is a partition $ \mathsf P '_{X_z}$ into two elements, and 
likewise for $ \mathsf P ' _{Y_z}$ so that $ \mathsf P ' _{X_z} \times \mathsf P ' _{Y_z} $ is 
$ (z, u ^{C_2}, \{X_z\} \times \{Y_z\})$-good.  
This follows from \eqref{e.2good1} and \eqref{e.2good2}.  

Define the partition $ \mathsf P '_X$ to be 
\begin{equation*} 
\mathsf P '_X = \mathsf P _X \wedge  \bigwedge_{z\in E} \mathsf P '_{X_z}\,. 
\end{equation*}
Observe that \eqref{e.22good2} follows.   Indeed, for each $ x\in \mathsf P _X$, we 
could have up to $ (\sharp \mathsf P _Y) \cdot 
 \operatorname {multi} (\mathsf P _Z \mid \mathsf P _X \times \mathsf P _Y)$ many sets to form the 
 minimum partition over, leading to \eqref{e.22good2}.  
 
Use the basic fact about martingales, \eqref{e.mds}, and the assumption that 
$ \mathbb P (E)\ge \tau $ to conclude that \eqref{e.22good1} holds. 
\end{proof}
%%%%%%%%%%%%%%%%%%%%%%%%%%%%%% PROOF PROOF PROOF

We make a definition that we use in this section, and the next.

%%%%%%%%%%%%%%%%%%%%%%%%%%%%%%  DEFINITION DEFINITION DEFINITION
\begin{definition}\label{d.3system}  We say that the 
data 
\begin{equation} \label{e.3sys}
\mathcal S= \{H \times H \times H\,,\,\mathsf P _H\,,\,  S_i\,,\,  \mathsf P _i \,,\, R_{j,k}\,,\,   \mathsf P _{j,k}
\,,\, T\,,\, \mathsf P _T 
\mid  1\le i\le 4\,,\,  1\le j<k\le 4 \}
\end{equation}
is a $ \emph{partition-system}$ iff  
%%  ITEMIZE
\begin{itemize}
\item $ \mathsf P _H$ is an affine partition of $ H \times H \times H$. 
\item $ S_i\subset H$,  and $ \mathsf P _i$ is a partition of $ \overline  S_i$ that is subordinate to 
$ \mathsf P _H$, $  1\le i\le 4$. 
\item $ R_{j,k}\subset S_j \times S_k$, 
and $ \mathsf P _{j,k}$ is a partition of $ \overline  R_{j,k}$ that is subordinate to $  \mathsf P _j \wedge \overline S_k $ 
and $\overline  S_j \times \mathsf P _k$, $  1\le j<k\le 4$. 
\item $ T\subset H\times H\times H$ is such that 
$
T\subset   \overline  R_{j,k} 
$, $  1\le j<k\le 4$. 
\item  $ \mathsf P _T = \bigwedge _{ 1\le j<k\le 4}   \mathsf P _{j,k}$. 
\end{itemize}
We stress that all partitions are collections of subsets of $ H \times H \times H$.   Set 
\begin{gather}
\label{e.dptl}
\mathsf P _{T, \ell } \coloneqq  \mathsf P _{\ell } \wedge \bigwedge _{\substack{ 1\le j<k\le 4\\ 
j,k\neq \ell }}  \mathsf P _{j,k}\,, \qquad 1\le \ell \le 4\,,
\\ \label{e.boldP1}
\mathbf  P _1 (\mathcal S) = \sum _{i=1} ^{4}  \operatorname {multi} (\mathsf P _i \mid \mathsf P_H)\,,
\\ \label{e.boldP2} 
\mathbf  P _2 (\mathcal S) = \sum _{ 1\le j<k\le 4} \operatorname {multi}(\mathsf P_i \mid   \mathsf P _{j,k}) \,, 
\\ \label{e.boldPT}
\mathbf  P _T (\mathcal S) =    \operatorname {multi}(\mathsf P_T \mid   \mathsf P _{H}) \,, 
\end{gather}
%% ITEMIZE
These last quantities are some counting functions that we will need to keep track of. 
\end{definition}
%%%%%%%%%%%%%%%%%%%%%%%%%%%%%%  DEFINITION DEFINITION DEFINITION

A \emph{trivial partition-system} is a partition-system in which each of the partitions 
are trivial.  
For each $ t\in \mathsf P _T$, we take 
\begin{equation} \label{e.s3t}
\mathcal S_3 (t)
= \{ H _{t,1} \times H _{t,2} \times H _{t,3}\,,\, s_{t:i} \,,\, r_{t:j,k} \,,\, t\mid  1\le i\le 4\,,\,  1\le j<k\le 4 \}
\end{equation}
to be the trivial partition-system associated to $ t$. Namely, we have 
%%  ITEMIZE
\begin{itemize}
\item $ t\subset H _{t,1} \times H _{t,2} \times H _{t,3}  $. 
Here, $ H _{t,1} \times H _{t,2} \times H _{t,3}$ may be the product of affine subspaces in $ H \times H \times H$, 
but all relevant notions extend to this setting. 
\item $ s _{t:j,k} \in \mathsf P _{j,k}$, with $ s _{t:j,k}\subset  H _{t,1} \times H _{t,2} \times H _{t,3}$, 
 and $ t = \bigwedge _{ 1\le j<k\le 4}  s _{t:j,k}$. 
\end{itemize}
%% ITEMIZE

This is the Lemma that will be applied in the next section.  

%%%%%%%%%%%%%%%%%%%%%%%%%%%%%% LEMMA LEMMA LEMMA
\begin{lemma}\label{l.2BU} Let $ C_1\ge 1$ be given. 
 There are  finite functions $ \Psi _{2-\Box} \;:\; [0,1] ^2 \times \mathbb N ^{2} 
\longrightarrow \mathbb N $  and $ \Psi _{\textup{codim}} \;:\; [0,1] ^2 \times \mathbb N ^2 
\longrightarrow \mathbb N $ so that the following holds for all 
$ 0< u_2, u_3 \tau  <1$.

For all partition-systems $ \mathcal S$, as in \eqref{e.3sys}, 
 there is a partition-system 
 \begin{equation} \label{e.2SYS}
\mathcal S'= \{H \times H \times H\,,\, \mathsf P _H'\,,\,    S_i\,,\,  \mathsf P '_i \,,\, R_{j,k}\,,\,
  \mathsf P ' _{j,k}\,,\, T \,,\, \mathsf P _T' \mid  1\le i\le 4\,,\,  1\le j<k\le 4 \}
\end{equation}
which refines $ \mathcal S$, so that these conditions are met. 
 For $  1\le i \le 4$ and $  1\le j,k \le 4$, 
\begin{gather} 
\label{e.BU1} 
\operatorname {codim} (\mathsf P '_H) \le \Psi _{\textup{codim}} (u_3,\tau ,  \mathbf  P _1(\mathcal S), \mathbf  P
_2(\mathcal S))\,,  
\\
\label{e.2BU2}
\operatorname {multi}(   \mathsf P '_i \mid \mathsf P_i)
\le \Psi _{2-\Box} (u_2, \tau, \mathbf  P _1(\mathcal S), \mathbf  P _2(\mathcal S))\,, 
\\ \label{e.2BU3}
\operatorname {multi} ( \mathsf P ' _{j,k} \mid  \mathsf P _{j} \wedge  \mathsf P _{k} )
\le \operatorname {multi} (\mathsf P _{j,k} \mid \mathsf P _j \times \mathsf P _k) \,, 
 \\ \label{e.2BU5}
 \mathbb P (E_{2, j,k}\mid S_j \times S_k)\le \tau \,, 
 \\ \notag 
 E_{2,j,k}=  
 \Biggl\{   r_{j,k}\in  \mathsf P '_{j,k}
 \mid    r_{j,k} \subset  s_j \cap  s_k\,,\ 
  s_v \in  \mathsf P '_{v}\,, v=j,k\,,
 \\ \notag \qquad \qquad \qquad 
 \norm  r_{j,k} - \mathbb P ( r_{j,k}\mid  s_j \times 
   s_k). \Box ^{ \{j,k\}}  s_j \times    s_k  . 
\ge u_2 [\mathbf P_T (\mathcal S')] ^{-C_1} \Biggr\}
\,, 
 \\ \label{e.2BU4}
 \mathbb P (E_{3, j}\mid S_j  )\le \tau \,, 
 \\ \notag 
 E_{3,j}=  
 \Biggl\{   s_{j}\in  \mathsf P '_{j}
 \mid    
 \norm  s_{j} - \mathbb P ( s _{j}\mid  A_j ) A_j. U (3), A_j . 
\ge u_3 [\mathbf P_T (\mathcal S')]  ^{-C_1} \Biggr\} \,. 
\end{gather}
Finally, $ \mathbf P_T (\mathcal S')=\mathbf P_T (\mathcal S)$. 
We are using the notation \eqref{e.boldP1}---\eqref{e.boldPT}.
\end{lemma}
%%%%%%%%%%%%%%%%%%%%%%%%%%%%%% LEMMA LEMMA LEMMA

The conclusion is that virtually all of the elements of the partitions $ \mathsf P _j' $ and $ \mathsf P _ {j,k}'$ 
are uniform with respect to Gowers Norm, and the Box Norm.

We emphasize that this Lemma provides us with a tower power bound. 
In \eqref{e.2BU2}, we have the estimates below, where 
note that we have a $ \log _{\ast }$, as in \eqref{e.log-star},  on the left. 
\begin{gather}
\label{e.2bu2}
\begin{split}
\log _{\ast } (
\sharp  \mathsf P '_i)
&\le  
 2 u_2 ^{-C_2} \tau ^{-1} 
 \mathbf  P _2 (\mathcal S)  ^{C_1 \cdot C_2}
+ \log _{\ast }  \mathbf  P _1 (\mathcal S) \,.
\end{split}
\end{gather}

Note that by \eqref{e.2BU3}, 
the multiplicity of the partitions $ \mathsf P '_{j,k} $, defined in \eqref{e.multi}, 
are not increased in this procedure, though we get a very substantial increase in the 
multiplicity of the $ \mathsf P _i'$, from the bound \eqref{e.2BU2}, forming the 
principal loss in the application of this Lemma. 
The sets $  s_i \not\in E _{1,i}$ 
are `very uniform,' even with respect to their probabilities in the respective 
cell of $ \mathsf P '$. 
The `tower' notation in \eqref{e.2BU2}  is defined in \eqref{e.tower}.

%%%%%%%%%%%%%%%%%%%%%%%%%%%%%% PROOF PROOF PROOF
\begin{proof}

We define a sequence of   partition-systems.  They are
\begin{equation}\label{e.Sm}
\begin{split}
\mathcal S (m)&= \{H \times H \times H\,,\, \mathsf P _H (m)\,,\,  S_i\,,\,  
\mathsf P _i (m) \,,\, R_{j,k}\,,\,   \mathsf P _{j,k} (m)\,,\, T \,,\, \mathsf P _{T} (m)
\\ & \qquad \mid  1\le i\le 4\,,\,  1\le j<k\le 4 \}
\end{split}
\end{equation}
where $ \mathcal S (0)$ is the  partition-system given to us by assumption.  
These partition-systems are \emph{refining}, in the sense that the corresponding sequences 
of partitions are refining.  

In this process, the only incremental change to the partitions $ \mathsf P_T (m)$ that are made 
are to make them subordinate to the other partitions.  Thus,  quantities that appear in \eqref{e.2BU5} and 
\eqref{e.2BU4} are constant.  Namely,  $ \mathbf Q= \mathbf P_T (\mathcal S (m))$ is independent of $ m$.

We also define a sequence of stopping times $ \sigma  (j,k ; m)$, and $ m (j,k)$
for $ 1\le j<k\le 4$, and $  m  \ge 0$. 
Initialize 
these stopping times  as follows, where  $  1\le j<k\le 4$.
\begin{gather*}
   m  \leftarrow 0,  \qquad   \sigma  ( j,k; 0 ) \leftarrow 0 \,,
  \qquad   m  (j,k)\leftarrow 0, \,.
\end{gather*}
We choose $ C_2$ as in Proposition~\ref{p.BoxTower}. 
The main recursion is this:  Set 
\begin{equation} \label{e.2Delta}
\Delta = u_2 ^{C_2} \tau 
= u_2 ^{C_2} \tau  \mathbf Q^{-C_1 \cdot C_2}
\end{equation}
\textsf{WHILE} there are $ 1\le j < k\le 4$  so that there is are two 
partitions $  \mathsf P'_j $ and $  \mathsf P'_k$ 
which satisfy \eqref{e.22good1} and \eqref{e.22good2} above for the quantity $ \Delta $. Namely, 
%%  ITEMIZE
\begin{itemize}
\item  $  \mathsf P'_j \wedge  \mathsf P'_k$ 
is $ ( \mathsf P _{j,k} ( m  ), \Delta ,    \mathsf P_j (m) \wedge  \mathsf P_k (m))$-good. 
\item  The multiplicity of $ P'_j$ satisfies 
\begin{equation}\label{e.upUp}
\begin{split}
\operatorname {mult} (  \mathsf P'_j \mid   \mathsf P_j (m) )
&\le 
2 \uparrow [ \operatorname {mult} (\mathsf  P_k (m) \mid  \mathsf P_H ( m  ))
 \cdot 
 \operatorname {multi} ( \mathsf P_{j,k}(  m  ) \mid  
 \mathsf  P_j (m)\times  \mathsf P_k (m))]
\\&
\le 2 \uparrow [ \operatorname {mult} (\mathsf  P_k (m) \mid  \mathsf P_H ( m  ))
 \cdot 
 \operatorname {multi} ( \mathsf P_{j,k}(  0) \mid  
 \mathsf  P_j (0)\times  \mathsf P_k (0))] \,, 
\end{split}
\end{equation}
and likewise for $   \mathsf P'_k $.    
\end{itemize}

%% ITEMIZE
We take these steps.  Update 
%%  ITEMIZE
\begin{enumerate}
\item (Keep track of stopping times.) 
\begin{equation*}
  m  \leftarrow  m  +1\,, \quad   m  (j,k)\leftarrow  m  (j,k)+1\,, \quad 
 \sigma ( j,k ;  m  ({j,k}))\leftarrow  m  \,. 
\end{equation*}

\item (Select affine partition.)  To each element of the affine partition $ \mathsf P_H (m)$, 
apply Lemma~\ref{l.uni2} to $\mathsf  P'_j$, $ 1\le j \le 4$, with 
the parameter $ \tau $ that is given to us, and the value of $ u$ in Lemma~\ref{l.uni2} equal to 
$
u= u_3 \mathbf Q ^{-C_1}
$. 
Set the partition that Lemma~\ref{l.uni2} 
supplies to us to be $ \mathsf P_H (m+1)$.   Observe that 
\begin{align}  \label{e.BUcodim} 
\operatorname {codim} (\mathsf P_H (m+1))
&\le 
\operatorname {codim} (\mathsf P_H (m))+ \bigl[ (u_3 \tau  ) ^{-1}  
\mathbf Q  \bigr] ^{D} 
\end{align}
This follows from Lemma~\ref{l.uni2} and \eqref{e.2good2}, for appropriate choice of 
constant $ D$. Note that the term $ \operatorname {multi}(\mathsf P'_j \mid \mathsf P_H (m))  $ 
is bounded in \eqref{e.upUp}.

\item (Updating the remaining partitions.) 
Set  $ \mathsf P _j (m+1)$ to be the maximal partition which refines $ \mathsf P'_j$ and is subordinate to 
$ \mathsf P _H (m+1)$.   Set $  \mathsf P _{j,k} (m+1)$ to be the maximal partition which refines $ \mathsf P _{j,k}
(m)$, and is subordinate to both $ \mathsf P_j (m+1)  $ and $ \mathsf P_k (m+1)$.  
The last partition $ \mathsf P_T (m+1) $ is then defined.  
\end{enumerate}
%% ENUMERATE

At the conclusion of the \textsf{WHILE} loop, return this data: 
For $  1\le j<k\le 4$, 
%%  ITEMIZE
\begin{itemize}
\item $  m  $, the integers $  m  (j,k)$.
\item The sequence of stopping times  
$ \sigma (j,k; \lambda )$, for 
$ 0\le \lambda  \le  m  ({j,k}) $. 
\end{itemize}
%% ITEMIZE

It remains to argue that the partitions returned satisfy the conclusions of the Lemma. 
We must have \eqref{e.2BU5}, else by the definition of $ \Delta $ in \eqref{e.2Delta} and 
Proposition~\ref{p.BoxTower},  the routine 
would not have stopped.  The conclusion \eqref{e.2BU3} follows from the construction.   
The conclusion \eqref{e.2BU4} follows from the manner in which we apply Lemma~\ref{l.uni2}, in 
in particular the point (2) above. 
The  remaining conclusions  \eqref{e.BU1} and \eqref{e.2BU2} require us 
to know how many recursions were performed.  We turn to this next. 

We claim that 
\begin{equation*}
 m  \le \Delta ^{-1} = u_2 ^{-C_2} \tau ^{-1} \mathbf  Q ^{C_1 \cdot C_2} \,. 
\end{equation*}
But this follows from Corollary~\ref{c.Stopping} applied to the construction, 
the sets in $ \mathsf P _{j,k}$, and the stopping times $ \sigma ( \{j,k\}, r_{j,k}, \lambda  )$. 

Therefore, we have, by induction, and \eqref{e.upUp}, we have  
\begin{align*}
\operatorname {multi} (\mathsf P ' _{i} \mid \mathsf P ')
& = \operatorname {multi} (\mathsf P _i (m) \mid \mathsf P ( m  ))
\\
& \le 2 \uparrow  [ \mathbf  P _2  \cdot   
\operatorname {multi} (\mathsf P _i (m -1) \mid \mathsf P (m -1 ))]
\\
&\le{} \stackrel {   m  \ \textup{times} }
{ 
\overbrace{2 \uparrow[ \mathbf  P _2  \cdot  2 \uparrow [ \mathbf  P _2  \cdots  [ \mathbf  P _2  \cdot 2 \uparrow 
\mathbf  \mathsf P _2 \cdot \mathbf  P _1
] \cdots ]]} } 
= \psi ( m  , \mathbf  P _1, \mathbf  P _2)\,, 
\end{align*}
Here, the notation is from \eqref{e.boldP1}, \eqref{e.boldP2}, and  Proposition~\ref{p.simpleBound}, which provides 
 crude bound given in  \eqref{e.2bu2}.  This proves \eqref{e.2BU2}.   The final conclusion \eqref{e.BU1} follows from 
 this last bound and \eqref{e.BUcodim}.  

\end{proof}
%%%%%%%%%%%%%%%%%%%%%%%%%%%%%% PROOF PROOF PROOF

%%%%%%%%%%%%%%%%%%%%%%%%%%%%%% SUBSECTION SUBSECTION SUBSECTION SUBSECTION
 %%%%%%%%%%%%%%%%%%%%%%%%%%%%%% SUBSECTION SUBSECTION SUBSECTION SUBSECTION 
\subsection{The Box Norm in Three Variables}%\label{ss.}

The goal of this section is to add the considerations about the Box Norm in three variables 
into our Lemmas, to build up an analog of Lemma~\ref{l.2BU} which also stipulates facts about the 
partition $ \mathsf P_T$, which as of yet we have not made any statements about.

%%%%%%%%%%%%%%%%%%%%%%%%%%%%%% LEMMA LEMMA LEMMA
\begin{lemma}\label{l.TBU} 
There are  finite functions $\Psi _{\textup{codim}}\,,\,   \Psi _{T} \;:\; [0,1] ^2 \times \mathbb N ^{2} 
\longrightarrow \mathbb N $ so that the following holds for all 
$ 0< u_T, \tau _T<1$.

For all trivial partition-systems $ \mathcal S$ there is a partition-system $ \mathcal S'$ as in \eqref{e.2SYS}, 
such that 
\begin{gather}
\label{e.TBU1}
\operatorname {codim} (\mathcal S') 
\le \Psi _{\textup{codim}} (u _{T}, \tau _T, \mathbb P (T \mid H\times H\times H))\,,
\\
\label{e.TBU2}
\mathbf P _T (\mathcal S')  
\le \Psi _T (u _{T}, \tau _T, \mathbb P (T \mid H\times H\times H))\,,
\\ \label{e.TBU4}
 \mathbb P (E   \mid H\times H\times H) \le \tau _T  \,, 
 \\ \notag 
 E  \coloneqq 
\Bigl\{   t \in    \mathsf P ' _{ T }   
\mid  \textup{ $\mathcal S _3 (t)$ 
is \emph{not}	 $u_T$-admissible } \Bigr\} \,. 
\end{gather} 
Here, $ \mathcal S_3 (t)$ is the trivial partition system associated with $ t$, as defined in \eqref{e.s3t}. 
\end{lemma}
%%%%%%%%%%%%%%%%%%%%%%%%%%%%%% LEMMA LEMMA LEMMA

 In \eqref{e.TBU4}, admissibility is as in Definition~\ref{d.admissible}.  
This proof will generate a second tower power in our estimate for the codimension 
in \eqref{e.TBU2}, but we don't detail this particular fact.

%%%%%%%%%%%%%%%%%%%%%%%%%%%%%% PROOF PROOF PROOF
\begin{proof}
For this proof, we  define a sequence of  partition-systems $ \mathcal S (m)$ as in \eqref{e.Sm}. 
These partition-systems are \emph{refining} in the sense that the corresponding sequences of 
partitions are refining.  We take $ \mathcal S (0)$ to be the trivial partition-system 
given by the hypothesis of the Lemma.

We also define a sequence of stopping times $ \sigma  ( \ell ,p_ \ell  )$
for $  1\le  \ell  \le 4$, with counters $ p _{\ell }\ge 0$.  
Initialize 
these variables  $ \sigma  ( \ell , 0 ) \leftarrow 0$ and $  p _{\ell }\leftarrow 0$, where  $  1\le  \ell  \le 4$.

Here is the recursive algorithm. 
\textsf{IF} $ m $ is even,  
apply of Lemma~\ref{l.2BU} to  $ \mathcal  S(m )$, with the values 
$ \kappa ( \tfrac 18 u _{T} \tau _T) ^{C} $ and $ \tfrac1{100}\tau_T$ specified at the beginning of Lemma~\ref{l.TBU}, the Lemma 
we are proving.  The value of $ C_1$ in Lemma~\ref{l.2BU} is the value of $ C+1$, where the constants $ \kappa $ and 
$ C$ are as in the definition of admissible, Definition~\ref{d.admissible}. 

We then update $ m \leftarrow m +1$, and take the 
updated data  $ \mathcal S ( m)$ to be the partition-system from  Lemma~\ref{l.2BU}.  
Observe that from  \eqref{e.2BU2} we have the estimates: 
\begin{gather}
 \label{e.T2}
\operatorname {multi}(   \mathsf P _i( m )\mid \mathsf P_i (m-1))
\le \Psi _{2-\Box} (u_T, \tfrac12\tau_T, \mathbf  P _1(m-1 ), \mathbf  P _2 (m -1))\,. 
\end{gather}

\textsf{IF} $ m $ is odd, by the previous step, the conclusions of 
Lemma~\ref{l.2BU} are in force.   The observation to make is that we have this condition.  
For the event $B$ defined below, we have $ \mathbb P (B )\le \tfrac 18  {\tau _T}  $.  
\begin{equation}\label{e.BB} 
\begin{split}
 B = \{t\in \mathsf P _T (m) \mid &
\textup{$ \mathcal S_3 (t)$ satisfies  \eqref{e.Ad2Box} and}
 \\
& \qquad \textup{ \eqref{e.Ad1Box} in the 
definition of $ u_T$-admissible.}
\}
\end{split}
\end{equation}
Recall that $ \mathcal S_3 (t)$ is given in \eqref{e.s3t}. 
That is, with very high probability,  if the trivial partition-system $ \mathcal S_3 (t)$  fails 
$ u_T$-admissibility, it must be the condition 
\eqref{e.Ad3Box} that fails. 

Let us see that this observation is true. The conditions \eqref{e.2BU5} and \eqref{e.2BU4} applied to $ \mathcal S (m)$ 
hold.  Thus, except on a set of probability at most $ \tfrac 1 {10} \tau _T$, we have, using the notation of 
\eqref{e.s3t}, 
\begin{gather*}
 \norm r _{t:j,k}- \mathbb P (r _{t:j,k}\mid s _{t:j} \times s _{t:k}). \Box ^{j,k}  s _{t:j} \times s _{t:k}. 
 \le \kappa (\tfrac 18 \tau _T u_T )^{C}
 [ \mathbf P_T (\mathcal S(m))] ^{-C-2} \,, 
 \\
 \norm s _{t:j}- \mathbb P (s _{t:j}\mid H _{t:j}). U (3). 
 \le \kappa (\tfrac 18 \tau _T u_T) ^{C}
 [\mathbf P_T (\mathcal S(m))] ^{-C-2} \,. 
\end{gather*}
Therefore, if the trivial partition-system $ \mathcal S_3 (t)$ fails either \eqref{e.Ad2Box} or \eqref{e.Ad1Box} 
in the definition of $ u_T$-admissibility, it must follow that $ t$ has very small probability in its affine cell. 
Namely, we must have 
\begin{equation}\label{e.tooSmall}
\mathbb P (t\mid H _{t:1}\times H _{t:2}\times H _{t:3}) \le 
\tfrac 18 \mathbf P_T (\mathcal S (m)) ^{-1} \tau _T \,. 
\end{equation}
But certainly, by the definition of $ \mathbf P_T (\mathcal S (m)$ in \eqref{e.boldPT}, we have 
\begin{equation*}
\sum _{t \;:\; \textup{$ t$ satisfies \eqref{e.tooSmall}}} \mathbb P (t \mid H \times H \times H) 
\le \tfrac 18 \tau _T \,. 
\end{equation*}
This means that $ \mathbb P (B )\le \tfrac18  {\tau _T}  $ for $ B$ as in \eqref{e.BB}. 

\smallskip 

\textsf{IF} there is an $  1\le  \ell  \le 4$ for which we have 
\begin{gather*}
\mathbb P (F_{\ell } \mid H\times H\times H) \ge \tfrac 18  \tau_T \,, 
 \\ 
 F_{\ell } \coloneqq 
 \Bigl\{   t \in    \mathsf P _{ T } (m) - B    
\mid  \textup{ $\mathcal S _3 (t)$ 
does not satisfy \eqref{e.Ad3Box} for this value of $ \ell $} \Bigr\} \,. 
\end{gather*}

For such a choice of $ \ell $, update $ p _{\ell }\leftarrow p _{\ell }+1$, 
and set $ \sigma (\ell , p _{\ell })\leftarrow m$.  
For each $  t  \in F _{\ell }$, we can apply Lemma~\ref{l.Tbox}. 
Write 
\begin{equation*}
t _{\ell }=   s _{t:\ell }\prod _{\substack{ 1\le j<k\le 4\\ j,k\neq \ell  }}
  r_{t:j,k}\,. 
\end{equation*}
Apply Lemma~\ref{l.Tbox} with $ V= t _{\ell }$, $ U= t$, and $ \tau =\kappa u_T ^{C} $.  
Since $ t\not\in B$, it follows that $ V=t _{\ell }$ satisfies the hypothesis of that Lemma, 
namely that $ V=t _{\ell }$ is $ (4, \vartheta , \ell )$-uniform, with $ \vartheta $ as in 
\eqref{e.zvoDef}. 

Then, from the conclusion of Lemma~\ref{l.Tbox}, we read this. 
There are partitions  $ \mathsf P (s _{t : j }, t _{\ell })$, $ 1\le j\le 4$, 
of $ s _{t: j }$ into two sets, and partitions 
\begin{equation*}
 \mathsf P ( r_{t:j,k},  t _{\ell })\,, \qquad 
 1\le j<k\le 4\,, \ j,k\neq \ell 
\end{equation*}
of $  r_{j,k}$ into two sets,  so that the there is an atom  $ V'$ in the partition 
\begin{equation*}
\mathsf P (s _{t: \ell } , t _{\ell }) \wedge \bigwedge _{\substack{1\le j < k \le 4\\  j,k\neq \ell }} 
 \mathsf P ( r_{t:j,k},  t _{\ell }) 
\end{equation*}
which has a higher correlation with $ t _{\ell }$.  Namely, 
\begin{gather*}
\mathbb P (V'\mid t) 
\ge c \bigl[ \kappa u_T ^{C} \mathbb P (t \mid t _{\ell }) \bigr]  ^{p} \,, 
\\
\mathbb P (t \mid V') \ge \mathbb P (t \mid  t _{\ell })+  
 c \bigl[ \kappa u_T ^{C} \mathbb P (t  \mid t _{\ell }) ^{C}  \bigr] ^{p} \,. 
\end{gather*}

Let 
\begin{equation*}
 \mathsf P ( t _{\ell })
= \bigwedge _{\substack{ 1\le j<k\le 4\\ j,k\neq \ell  }} \mathsf P ( r_{j,k},  t _{\ell })\,. 
\end{equation*}
It follows that we have 
\begin{equation}\label{e.tUP}
\mathbb E [\mathbb E (T \cap  t _{\ell  } \mid \mathsf P ( t _{\ell }) ]^2 
\ge 
\mathbb P (T \mid  t _{\ell }) ^2 + 
u_T ^C \mathbb P (T \mid   t _{\ell })^2 \,. 
\end{equation}

We update 
\begin{gather*}
 \mathsf P _i(m+1) \leftarrow   \mathsf P _i(m)\,, \qquad i\neq \ell \,,
\\
 \mathsf P (R_{j,k }, m)\wedge \bigwedge _{ t _{\ell }\in F _{\ell }} \mathsf P (R_{j,k },  t _{\ell })\,, 
\qquad  1\le j<k\le 4\, ,\ j,k\neq \ell \,. 
\end{gather*}
It is this last two steps that create a second tower. Observe that we have,
using the notation of \eqref{e.boldP1} and \eqref{e.boldP2}, 
\begin{equation}\label{e.bigUP1}
\mathbf  P _ u ( \mathcal S( m))  
\le  \mathbf P _u (\mathcal S(m-1)) 2 \uparrow [2  \mathbf  P _2 (\mathcal S(m-1)) ^{6}] \qquad u=1,2\,.  
\end{equation}

It follows from \eqref{e.tUP} that we have 
\begin{equation}\label{e.TTup}
\mathbb E \,[ \mathbb E (T\mid   \mathsf P _{T_\ell } (m)) ]^2 
\ge 
\mathbb E \, [\mathbb E (T\mid    \mathsf P _{T_\ell } (m-1)) ]^2
+ \tau _T u_T ^{C} \mathbb P (T\mid T_ \ell ) ^2 \,.  
\end{equation}
The recursion then loops.  

\bigskip

Once the recursion has stopped, it follows from the construction, in particular
\eqref{e.TTup}, and Proposition~\ref{p.STopping}, that we must have 
\begin{equation} \label{e.90}
p _{\lambda } \le \tau _T ^{-2} u_T ^{-2C} \,.  
\end{equation}
The sum $ 2\sum _{\ell =1} ^{4} p _{\ell }$ bounds the running time. 

At the end of the recursion, the conclusion \eqref{e.TBU4} holds.  The other conclusions 
are appropriate upper bounds on the multiplicities in terms of 
some (very quickly growing) function of $ u_T$, $ \tau _T$, and the  
multiplicities of the given partitions.  These estimates follow from 
 \eqref{e.T2}, and \eqref{e.bigUP1}. 

To supply some details, let us set 
\begin{align*}
\Gamma (1) & \coloneqq  \Psi _3 (u_T, \tfrac 12 \tau _T, \mathbf  P _1 (\mathcal S), \mathbf  P _2\mathcal S)
 \times [2 \uparrow [2 \mathbf  P _2 ^{6}]] \,, 
 \\
 \Gamma (p+1) & \coloneqq  
 \Psi _3 (u_T, \tfrac 12 \tau _T, \Gamma (p), \Gamma (p))
 \times [2 \uparrow [2 \Gamma (p) ^{6}]]\,. 
\end{align*}
From \eqref{e.boldP1}, \eqref{e.boldP2}, \eqref{e.T2}, \eqref{e.bigUP1}, 
and \eqref{e.90},  we have 
\begin{equation*}
\operatorname {mult} (\mathsf P _i(m) \mid \mathsf P (m))
\le 
\Gamma (m) \le \Gamma (8 \tau _T ^{-2} u_T ^{-2C})\,, i=1,2\,. 
\end{equation*}
Since $ \Psi _3$ is itself a power-tower, defined in terms of the  $ 2 \uparrow\uparrow  J$ 
function, we thus, have a second power-tower from this estimate.  
Since the partition $ \mathsf P_T$ is generated from the prior partitions, this last estimate 
proves \eqref{e.TBU2}.  The estimate \eqref{e.TBU1} follows from similar considerations, and 
the estimate \eqref{e.BU1}. 
\end{proof}
%%%%%%%%%%%%%%%%%%%%%%%%%%%%%% PROOF PROOF PROOF

%%%%%%%%%%%%%%%%%%%%%%%%%%%%%% SUBSECTION SUBSECTION SUBSECTION SUBSECTION
%%%%%%%%%%%%%%%%%%%%%%%%%%%%%% SUBSECTION SUBSECTION SUBSECTION SUBSECTION 
\subsection{Proof of Lemma~\ref{l.uni}}%\label{ss.}
Recall that $ A\subset T$, by assumption, and that $ \mathbb P (A\mid T)\ge \delta + \nu $. 
Apply Lemma~\ref{l.TBU} to the corner system $ \mathcal A$ as in \eqref{e.Asystem}. 
This Lemma also takes the parameters 
 \begin{equation*}
u _{T}= \delta \,, \qquad  \tau _T=  c  \nu ^ {C_T} \mathbb P (T\mid H\times H\times
H)\,. 
\end{equation*}
Here the constant $C_T $ is the constant that appears  Lemma~\ref{l.Tbox}, see \eqref{e.Tbox6}. 
Let $ \mathcal S'$ be the partition-system given to us by this Lemma, satisfying \eqref{e.TBU2} 
and \eqref{e.TBU4}.

Also consider the set 
\begin{equation*}
E'   \coloneqq 
\bigl\{  t  \in  \mathsf P '_ { T }   
\mid   \mathbb P ( t \mid H _{t,1} \times H _{t,2} \times H _{t,3}   )\le v [\mathbf  P _T (\mathcal S')] ^{-1}
 \mathbb P (T \mid H\times H\times H)
\bigr\} 
\end{equation*}
Here, we are using the notation of \eqref{e.s3t} and \eqref{e.boldPT}.  
Then, it is clear that $ \mathbb P \Bigl(\bigcup \{t \mid t\in E'\} \Bigr)\le \tau _T$.  
Hence, by the pigeonhole principle (See Proposition~\ref{p.partitions}.) we can select 
$ t\in \mathsf P'_T$ so that $ t\not\in E'$, and the $ \mathcal T$-system $ \mathcal S_3 (t)$ 
is $ \delta  $-admissible, which is \eqref{e.inc0} and $ \mathbb P (A\mid T)\ge \delta + \nu /4$ 
which is \eqref{e.inc5}. 
The estimate \eqref{e.H'} follows from the estimate \eqref{e.TBU1}.

%%%%%%%%%%%%%%%%%%%%%%%%%%%%%% SECTION  SECTION SECTION
%%%%%%%%%%%%%%%%%%%%%%%%%%%%%% SECTION  SECTION SECTION 
\section{The Algorithm to Conclude the Main Theorem} %\label{s.}
\label{s.algorithm}

This is a well-known argument.  To prove our main Theorem, we should show that 
for any $ 0<\delta <1$ there is an $ n (\delta )$ so that if 
$ \operatorname {dim} (H)\ge n (\delta )$, and $ A\subset H \times H \times H $ 
with $ \mathbb P (A \mid H \times H \times H )\ge \delta $, then $ A$ contains 
a corner.

We recursively construct a sequence of corner-systems 
\begin{equation*}
\mathcal A (m) 
=
\{H \,, \,  S_i (m)\,,\, R_{i,j} (m)\, ,\, T (m)\,, A (m) \mid 1\le i,j\le 4\}\,. 
\end{equation*}
$ \mathcal A (0)$ is the `trivial' corner-system 
\begin{equation*}
R_{i} (0)=H\,, \quad S _{i,j} (0)= H \times H\,, \quad 
T = H \times H \times H  \,, \quad A (0)=A\,. 
\end{equation*}
Moreover, at each stage,   $ A (m) \subset A $, so that a corner in $ A (m)$ is
a corner in $ A$.

The point is that the recursion, when it stops, provides us with 
an corner-system $ \mathcal A (m_0)$ so that (1)
$ \mathbb P (A (m_0)\mid T (m_0))\ge \delta $, 
(2)
$ \mathcal A (m_0)$ 
is $ \mathbb P (A (m_0)\mid T (m_0))$-admissible, 
(3) $ \mathcal A (m_0)$ satisfies \eqref{e.UniformEnough}, 
\begin{gather}
\label{e.codim} 
\operatorname {dim} (H (m_0))\ge \operatorname {dim} (H )- \Phi _{\textup{dim}} (\delta )\,, 
\\
\label{e.ProbA}
\mathbb P (T (m_0) \mid H(m_0) \times H(m_0) \times H(m_0) )
\ge \Phi _{A, \mathbb P } (\delta )\,. 
\end{gather}
Here, $ \Phi _{\textup{dim}}$ is a map from $ [0,]$ to $ \mathbb N $, and 
$ \Psi _{A, \mathbb P } (\delta )$ is a finite function from $ [0,1]$ to itself. 
Then,  it follows that Lemma~\ref{l.3dvon} implies $ A (m_0)$ has a corner 
provided \eqref{e.BigEnough} holds, that is 
\begin{align*}
\lvert  H (m_0)\rvert ^{4}  
\ge 100 
\Psi _{A, \mathbb P } (\delta ) ^{3} \,. 
\end{align*}
By \eqref{e.codim}, this will clearly hold provided $ \operatorname {dim} (H)>n (\delta )$, for 
a computable function $ n (\delta )$.  
Thus, our Main Theorem is proved. 

\medskip 

The recursion is this: Given the corner-system $ \mathcal A (m)$, 
it will be $ \mathbb P (A (m)\mid T (m))$-admissible.  If it does not 
satisfy \eqref{e.UniformEnough}, then we apply Lemma~\ref{l.dinc} to 
conclude the existence of an corner-system 
\begin{equation*}
\mathcal A ' (m)
=\{H '(m)\,,\, S_i '(m)\,,\, R_{i,j} '(m)\, ,\, T '(m)\,, A '(m) \mid 1\le i,j\le 4\}
\end{equation*}
satisfying these conditions: $ A' (m)\subset A (m)$, 
\begin{equation}\label{e.prime}
\begin{split}
\mathbb P (T' (m)\mid T (m)) &\ge \kappa [ \mathbb P (A (m) \mid T (m))] ^{1/ \kappa}
\,,
\\
\mathbb P (A' (m)\mid T '(m)) &\ge \mathbb P (A (m)\mid T (m))+  
\kappa [ \mathbb P (A (m) \mid T (m))] ^{1/ \kappa}\,. 
\end{split}
\end{equation} 
These are the conclusions of Lemma~\ref{l.dinc}.  

The corner-system $ \mathcal A' (m)$ need not be $\mathbb P (A' (m)\mid T '(m)) $-admissible, 
therefore, we apply Lemma~\ref{l.uni}, with 
\begin{equation*}
\delta = \mathbb P (A (m)\mid T (m))\,, 
\qquad 
v= \kappa [ \mathbb P (A (m) \mid T (m))] ^{1/ \kappa}\,. 
\end{equation*}
The conclusion of this Lemma gives us a new corner-system $ \mathcal A (m+1)$, which satisfies 
\begin{gather} 
\begin{split} 
\mathbb P (A (m+1)\mid T (m+1)) & \ge 
\mathbb P (A (m)\mid T (m))+ 
\kappa [ \mathbb P (A (m) \mid T (m))] ^{ 1/ \kappa}
\\   \label{e.Aup}
&\ge \delta + {\kappa} \delta ^{1/ \kappa}
\end{split}
\\   \label{e.PT>}
\begin{split}
\mathbb P (T (m+1) \mid & H (m+1) 
\times H (m+1) \times H (m+1) ))) 
\\& 
\ge  \widetilde \Psi _{T}
(\mathbb P (A (m)\mid T (m)) ,  \mathbb P (T (m)\mid H (m) 
\times H (m) \times H (m) ))\,, 
\end{split}
\\ 
\operatorname {codim} (H (m+1)) \le   \Psi_ {\textup{codim} }
(\mathbb P (A (m)\mid T (m)) ,  \mathbb P (T (m)\mid H (m) 
\times H (m) \times H (m) )) \,. 
\end{gather}
The functions $\Psi_ {\textup{codim} } $ and $ \widetilde \Psi _{T}$ are derived from 
those in \eqref{e.H'} and \eqref{e.inc2} by a change of variables.

Note that \eqref{e.Aup} implies that the recursion can continue for at most 
$  m_0\lesssim 4 (\kappa \delta ^{1/ \kappa }) ^{-1}  $ times before it must stop, as the 
density of $ A (m) $ in $ T (m)$ can never be more than $ 1$. 
Note that initially, we have $ T (0)=H (0) \times H (0) \times H (0)$, therefore the 
iteration of the estimate \eqref{e.PT>} can be phrased completely in terms of a fixed function 
of $ \delta = \mathbb P (A (0))$, therefore the estimate  
 \eqref{e.ProbA}  holds.  A similar argument applies to prove the estimate   \eqref{e.codim}, 
 completing the proof of our Main Theorem.

%%%%%%%%%%%%%%%%%%%%%%%%%%%%%%  Bib Section

\textsc{Michael Lacey, School of Mathematics, Georgia Institute of Technology, Atlanta GA 30332, USA}\hfill\break
\textsc{Email:} \verb|lacey@math.gatech.edu|

\textsc{William McClain, School of Mathematics, Georgia Institute of Technology, Atlanta GA 30332, USA}
\hfill\break
\textsc{Email:} \verb|bill@math.gatech.edu|

\end{document}